\documentclass[11pt,a4paper,leqno]{amsart}
\usepackage{amsmath,amsfonts,amssymb,amsthm}
\usepackage[alphabetic]{amsrefs}
\usepackage{hyperref}
\usepackage{graphicx, color}
\usepackage{pictexwd,dcpic,epsf}
\usepackage{enumerate}
\usepackage{fullpage}
\usepackage{xspace}
\usepackage[plain,vlined,english,norelsize]{algorithm2e} 

\newtheorem{theorem}{Theorem}[section]

\newtheorem{lemma}[theorem]{Lemma}

\newtheorem{corollary}[theorem]{Corollary}

\newtheorem{proposition}[theorem]{Proposition}

\newtheorem{Prob}[theorem]{Problem}

\newtheorem{claim}[theorem]{Claim}

\theoremstyle{definition}
\newtheorem{definition}[theorem]{Definition}
\newtheorem{example}[theorem]{Example}

\newtheorem{remark}[theorem]{Remark}

\newcommand{\tf}{\ensuremath{\widetilde{f}}\xspace}

\newcommand{\tv}{\ensuremath{\widetilde{v}}\xspace}

\newcommand{\tG}{\ensuremath{\widetilde{G}}\xspace}

\newcommand{\tX}{\ensuremath{\widetilde{X}}\xspace}

\newcommand{\cA}{\ensuremath{\mathcal{A}}\xspace}
\newcommand{\cB}{\ensuremath{\mathcal{B}}\xspace}
\newcommand{\cC}{\ensuremath{\mathcal{C}}\xspace}
\newcommand{\cE}{\ensuremath{\mathcal{E}}\xspace}
\newcommand{\cF}{\ensuremath{\mathcal{F}}\xspace}
\newcommand{\cG}{\ensuremath{\mathcal{G}}\xspace}
\newcommand{\cH}{\ensuremath{\mathcal{H}}\xspace}

\newcommand{\cS}{\ensuremath{\mathcal{S}}\xspace}
\newcommand{\cT}{\ensuremath{\mathcal{T}}\xspace}
\newcommand{\cX}{\ensuremath{\mathcal{X}}\xspace}

\newcommand{\N}{\ensuremath{\mathbb{N}}\xspace}
\newcommand{\R}{\ensuremath{\mathbb{R}}\xspace}
\newcommand{\Z}{\ensuremath{\mathbb{Z}}\xspace}

\newcommand{\bfG}{\ensuremath{G\xspace}}

\newcommand{\fcom}[1]{\ensuremath{F(#1)}\xspace}      \newcommand {\thi}[1]{\ensuremath{Th(#1)}\xspace} \DeclareMathOperator{\He}{Helly}
\DeclareMathOperator{\rk}{rk}

\DeclareMathOperator{\cP}{\mathcal{P}}

\newcommand{\hR}{\ensuremath{\widehat{R}}\xspace}

\newcommand{\dm}{\ensuremath{\bar{d}}\xspace}
\newcommand{\du}[3]{\ensuremath{{#1} \bowtie_{#3} {#2}}\xspace}

\newcommand{\mr}{\mathrm}

\newcommand{\conf}{\mathop{\rm conf}\xspace}
\newcommand{\helly}{\mathop{\rm Helly}\xspace}

\def\lgate{{\langle\langle}}
\def\rgate{{\rangle\rangle}}

\numberwithin{equation}{section}

\newcommand{\diam}{\mathop{\rm diam} }

\newcommand{\catz}{$\mathrm{CAT}(0)$ }
\newcommand{\cftf}{$\mathrm{C}(4)\mathrm{-T}(4)$ }

\begin{document}

\title[Helly groups]{Helly groups}

\begin{abstract}
Helly graphs are graphs in which every family of pairwise intersecting balls has a non-empty intersection.
This is a classical and widely studied class of graphs. In this article
we focus on groups acting geometrically on Helly graphs -- \emph{Helly groups}. We provide numerous examples
of such groups: all (Gromov) hyperbolic, \catz cubical, finitely presented graphical \cftf small cancellation groups, and type-preserving uniform lattices in Euclidean buildings of type $C_n$ are
Helly; free products of Helly groups with amalgamation over finite subgroups, graph products of Helly groups,
some diagram products of Helly groups, some right-angled graphs of Helly groups, and quotients of Helly groups by finite normal subgroups are Helly.
We show many properties of Helly groups: biautomaticity, existence of finite dimensional models for
classifying spaces for proper actions, contractibility of asymptotic cones, existence of EZ-boundaries, satisfiability of the Farrell-Jones conjecture
and of the coarse Baum-Connes conjecture. This leads to new results for some classical families of groups (e.g.\ for FC-type Artin groups)
and to a unified approach to results obtained earlier.
\end{abstract}

\author[J.\ Chalopin]{J\' er\'emie Chalopin}
\address{Laboratoire d'Informatique et Syst\`emes, CNRS and Aix-Marseille Universit\'e,  Marseille, France}
\email{jeremie.chalopin@lis-lab.fr}

\author[V.\ Chepoi]{Victor Chepoi}
\address{Laboratoire d'Informatique et Syst\`emes, Aix-Marseille Universit\'e and CNRS, Marseille, France}
\email{victor.chepoi@lis-lab.fr}

\author[A. Genevois]{Anthony Genevois}
\address{D\'epartement de Math\'ematiques B\^atiment 307, Facult\'e des Sciences
d'Orsay Universit\'e Paris-Sud, F-91405 Orsay, France}
\email{anthony.genevois@math.u-psud.fr}

\author[H.\ Hirai]{Hiroshi Hirai}
\address{Graduate School of Mathematics,
	Nagoya University, 
	Furocho, Chikusaku, Nagoya, 464-8602, Japan}
\email{hirai.hiroshi@math.nagoya-u.ac.jp}

\author[D.\ Osajda]{Damian Osajda}
\address{Instytut Matematyczny,
Uniwersytet Wroc\l awski\\
pl.\ Grunwaldzki 2/4,
50--384 Wroc{\l}aw, Poland}
\address{Institute of Mathematics, Polish Academy of Sciences\\
	\'Sniadeckich 8, 00-656 War\-sza\-wa, Poland}
\email{dosaj@math.uni.wroc.pl}

\date{\today}

\maketitle

\tableofcontents

\section{Introduction}\label{s:intro}

\subsection{Motivations and main results}
\label{s:intro_main}
A geodesic metric space is \emph{injective} if any family of pairwise intersecting balls has a non-empty
intersection \cite{AroPan1956}. Injective metric spaces appear independently in various fields of mathematics and computer science:
in topology and metric geometry -- also known as \emph{hyperconvex spaces} or \emph{absolute retracts} (in the category of
metric spaces with $1$-Lipschitz maps); in combinatorics -- also known as \emph{fully spread spaces}; in functional analysis and fixed point theory -- also known as \emph{spaces with binary intersection property}; in the theory of algorithms -- known as \emph{convex hulls}, and elsewhere. They form a very natural and important class
of spaces and have been studied thoroughly. The distinguishing  feature of injective spaces is that any metric space admits an {\it injective hull}, i.e.,
the smallest injective space into which the input space isometrically embeds;  this important result was  rediscovered several times in the past \cite{Isbell64,Dress1984,ChLa94}.
\medskip

A discrete counterpart of injective metric spaces are \emph{Helly graphs} -- graphs in which any family of pairwise intersecting (combinatorial) balls has a non-empty
intersection. Again, there are many equivalent definitions of such graphs, hence they are also known as e.g.\ \emph{absolute retracts} (in the category of graphs with nonexpansive maps) \cite{BP-absolute,BaPr,JaMiPou,Qui,Pesch87,Pesch88}.
\medskip

As the similarities in the definitions suggest, injective metric spaces and Helly graphs exhibit a plethora of analogous features. A simple but important example of an injective metric space is $(\mathbb R^n,d_{\infty})$, that is, the $n$-dimensional real vector space with the metric coming from the supremum norm.
The discrete analog is $\boxtimes_1^n L$, the direct product of $n$ infinite lines $L$, which embeds isometrically into $(\mathbb R^n,d_{\infty})$ with vertices being the points with integral coordinates.
The space $(\mathbb R^n,d_{\infty})$ is quite different from the `usual' Euclidean $n$-space $\mathbb{E}^n=(\mathbb R^n,d_{2})$. For example, the geodesics between two points in $(\mathbb R^n,d_{\infty})$ are not unique, whereas such uniqueness
is satisfied in the `nonpositively curved' $\mathbb{E}^n$. However, there is a natural `combing' on $(\mathbb R^n,d_{\infty})$ -- between any two points there is a unique `straight' geodesic line. More generally, every
injective metric space admits a unique geodesic bicombing of a particular type (see Subsection~\ref{s:bicombing} for details). The existence of such a bicombing allows us to conclude many properties
typical for nonpositively curved -- more precisely, for \catz -- spaces. Therefore, injective metric spaces
can be seen as metric spaces satisfying some version of `nonpositive curvature'.
Analogously, Helly graphs and the associated \emph{Helly complexes} (that is, flag completions of Helly graphs), enjoy many nonpositive-curvature-like features. Some of them were exhibited in our earlier work: in \cite{CCHO} we prove e.g.\ a version of the Cartan-Hadamard theorem for Helly complexes. Moreover, the construction of the injective hull associates with every Helly graph an  injective metric space into which the graph embeds isometrically and coarsely surjectively. For the example presented above, the injective hull of
$\boxtimes_1^n L$ is $(\mathbb R^n,d_{\infty})$.
\medskip

Exploration of groups acting nicely on nonpositively curved complexes is one of the main activities in
Geometric Group Theory. In the current article we initiate the study of groups acting geometrically (that is, properly and cocompactly, by automorphisms) on Helly graphs. We call them \emph{Helly groups}. We show that the class is vast -- it contains many large classical families of groups (see Theorem~\ref{t:classes} below), and is closed
under various group theoretic operations (see Theorem~\ref{t:operations}). In some instances, the Helly group structure is the only known nonpositive-curvature-like structure. Furthermore, we show in Theorem~\ref{t:properties1} that Helly groups satisfy some strong algorithmic, group theoretic, and
coarse geometric properties. This allows us to derive new results for some classical groups and present
a unified approach to results obtained earlier.

\begin{theorem}
	\label{t:classes}
	Groups from the following classes are Helly:
	\begin{enumerate}
		\item \label{t:classes(1)} groups acting geometrically on graphs with `near' injective metric hulls, in particular, (Gromov) hyperbolic groups;
		\item \label{t:classes(2)} \catz cubical groups, that is, groups acting geometrically on \catz cube complexes;
		\item \label{t:classes(3)} finitely presented graphical \cftf small cancellation groups;
		\item \label{t:classes(4)} groups acting geometrically on swm-graphs, in particular, type-preserving uniform lattices in Euclidean buildings of type $C_n$.
	\end{enumerate}
\end{theorem}

As a result of its own interest, as well as a potentially very useful tool for establishing Hellyness of groups (in particular, used successfully in the current paper) we prove the following theorem. The \emph{coarse Helly} property is a natural `coarsification' of the Helly property, and the property of \emph{$\beta$-stable intervals} was
introduced by Lang \cite{Lang2013} in the context of injective metric spaces and is related to Cannon's property of having finitely many cone types (see Subsection~\ref{s:intro_org} of this Introduction
for further explanations).

\begin{theorem}\label{t:prop-coarse-Helly-groups-Helly}
	A group acting geometrically on a coarse Helly graph with
	$\beta$-stable intervals is Helly.
\end{theorem}

Furthermore, it has been shown recently in \cite{HuangO} that FC-type Artin groups and weak Garside groups of finite type are Helly. The latter class contains e.g. fundamental groups of the complements of complexified finite simplicial arrangements of hyperplanes; braid groups of well-generated complex reflection groups; structure groups of non-degenerate, involutive and braided set-theoretical solutions of the quantum Yang-Baxter equation; one-relator groups with non-trivial center and, more generally, tree products of cyclic groups. Conjecturally, there are many more Helly groups -- see the discussion in Section~\ref{s:questions}.

\begin{theorem}
	\label{t:operations}
	Let $\Gamma,\Gamma_1,\Gamma_2,\ldots,\Gamma_n$ be Helly groups. Then:
	\begin{enumerate}
		\item \label{t:operations(1)} a free product $\Gamma_1 \ast_F \Gamma_2$ of $\Gamma_1,\Gamma_2$ with amalgamation over
		a finite subgroup $F$, and the HNN-extension $\Gamma_1\ast _F$ over $F$ are Helly;
		\item \label{t:operations(2)} every graph product of $\Gamma_1,\ldots,\Gamma_n$ is Helly, in particular, the direct product $\Gamma_1 \times \cdots \times \Gamma_n$ is Helly;
		\item \label{t:operations(3)} the $\square$-product of $\Gamma_1,\Gamma_2$, that is,
		$\Gamma_1 \square \Gamma_2 = \langle \Gamma_1,\Gamma_2,t \mid [g,h]=[g,tht^{-1}]=1, \ g \in \Gamma_1, h \in \Gamma_2 \rangle$ is Helly;
		\item \label{t:operations(4)} the \emph{$\rtimes$-power} of $\Gamma$, that is,
		$\Gamma^{\rtimes} = \langle \Gamma,t \mid [g,tgt^{-1}]=1, \ g \in \Gamma \rangle$ is Helly;
		\item \label{t:operations(5)} the quotient $\Gamma/N$ by a finite normal subgroup $N\lhd \Gamma$ is Helly.
	\end{enumerate}
\end{theorem}

Observe also that, by definition, finite index subgroups of Helly groups are Helly. Again, we conjecture that Hellyness is closed under other group theoretic constructions -- see the discussion in Section~\ref{s:questions}.
The items (\ref{t:operations(2)})-(\ref{t:operations(4)}) in Theorem~\ref{t:operations} are consequences of the following combination theorem for actions on quasi-median graphs with Helly stabilisers. Further consequences of the same result are presented in Subsection~\ref{s:qmedian} in the main body of the article.

\begin{theorem}\label{thm:HellyGeneral_int}
	Let $\Gamma$ be a group acting topically-transitively on a quasi-median graph $G$. Suppose that:
	\begin{itemize}
		\item any vertex of $G$ belongs to finitely many cliques;
		\item any vertex-stabiliser is finite;
		\item the cubical dimension of $G$ is finite;
		\item $G$ contains finitely many $\Gamma$-orbits of prisms;
		\item for every maximal prism $P= C_1 \times \cdots \times C_n$, $\mathrm{stab}(P)= \mathrm{stab}(C_1) \times \cdots \times \mathrm{stab}(C_n)$.
	\end{itemize}
	If clique-stabilisers  are Helly, then $\Gamma$ is a Helly group.
\end{theorem}

The results above show that the class of Helly groups is vast. Nevertheless, we may prove a number
of strong properties of such groups. One very interesting and significant aspect of the theory is that
the Helly group structure equips the group not only with a specific combinatorial structure being the source of
important algorithmic and algebraic features (as e.g.\ (\ref{t:properties1(1)}) in the theorem below) but -- via the Helly hull construction -- provides a more concrete `nonpositively curved' object acted upon the group: a metric space with convex geodesic bicombing (see (\ref{t:properties1(5)}) below). Such spaces
might be approached using methods typical for the \catz setting, and are responsible for many `$\mathrm{CAT}(0)$-like' results
on Helly groups, such as e.g.\ items (\ref{t:properties1(6)})--(\ref{t:properties1(9)}) in the following theorem.

\begin{theorem}
	\label{t:properties1}
	Let $\Gamma$ be a group acting geometrically on a Helly graph $G$, that is, $\Gamma$ is a Helly group. Then:
	\begin{enumerate}
		\item \label{t:properties1(1)} $\Gamma$ is biautomatic.
		\item \label{t:properties1(2)} $\Gamma$ has finitely many conjugacy classes of finite subgroups.
		\item \label{t:properties1(3)} $\Gamma$ is (Gromov) hyperbolic if and only if $G$ does not contain an isometrically embedded infinite
		$\ell_{\infty}$--grid.
		\item \label{t:properties1(4)} The clique complex $X(G)$ of $G$ is a finite-dimensional cocompact model for the classifying space
		$\underline{E}\Gamma$ for proper actions. As a particular case, $\Gamma$ is always of type $F_\infty$ (see e.g. \cite[Theorem~7.3.1]{MR2365352}); and of type $F$ when it is torsion-free.
		\item \label{t:properties1(5)} $\Gamma$ acts geometrically on a proper injective metric space of finite combinatorial dimension, and hence on a metric space with a convex geodesic bicombing.
		\item \label{t:properties1(6)} $\Gamma$ admits an EZ-boundary $\partial G$.
		\item \label{t:properties1(7)} $\Gamma$ satisfies the Farrell-Jones conjecture with finite wreath products.
		\item \label{t:properties1(8)} $\Gamma$ satisfies the coarse Baum-Connes conjecture.
		\item \label{t:properties1(9)} The asymptotic cones of $\Gamma$ are contractible.
	\end{enumerate}
\end{theorem}

As immediate consequences of the main theorems above we obtain new results on some classical group classes.
For example it follows that FC-type Artin groups and finitely presented graphical \cftf small cancellation groups are biautomatic. Further discussion of important consequences of our main results is presented in Subsection~\ref{s:intro_discuss} below. Note also that by Theorem~\ref{t:properties1}(\ref{t:properties1(5)}) further properties of Helly groups can be deduced from e.g.\ \cite{DesLang2015,DesLang2016,Desc2016} (see also the discussion in \cite[Introduction]{HuangO}).

The above Theorems~\ref{t:classes}--\ref{t:properties1} are proved by
the use of corresponding more general results on Helly graphs. A
fundamental property that we use is the following local-to-global
characterization of Helly graphs from~\cite{CCHO}: A graph $G$ is Helly
if and only if $G$ is clique-Helly (i.e., any family of pairwise
intersecting maximal cliques of $G$ has a non-empty intersection) and
its clique complex $X(G)$ is simply connected.  Here, we present some
of the results we obtained about Helly graphs (or complexes) in a
simplified form (see Subsection~\ref{s:intro_org} of this Introduction
for further explanations).

\begin{theorem}\label{t:main-graphs}
The following constructions give rise to Helly graphs:
  \begin{enumerate}
  \item \label{t:main-graphs(1)} A union of graph-products (UGP)
of clique-Helly graphs satisfying the $3$-piece condition is
    clique-Helly.  If its clique complex is simply connected then it
    is Helly.
\item \label{t:main-graphs(2)} Thickenings of simply connected \cftf graphical small cancellation
    complexes are Helly.
  \item \label{t:main-graphs(3)} Rips complexes and face complexes of Helly graphs are Helly.
\item \label{t:main-graphs(4)} Nerve complexes of the cover of a Helly graph or of a
    $7$-systolic graph by maximal cliques are Helly.
  \end{enumerate}
\end{theorem}

It was already known that the thickening operation allows to obtain
Helly graphs from several classes of graphs: the thickenings of
locally finite median graphs~\cite{BavdV3}, swm-graphs~\cite{CCHO},
and hypercellular graphs~\cite{ChKnMa19} are Helly. In fact all these
three results can be seen as particular cases of the following
proposition.

\begin{proposition}\label{prop-cell-helly}
  If $G$ is a graph endowed with a cell structure $X$ such that each
  cell of $X$ is gated in $G$ and the family of cells satisfies the
  3-cell and the graded monotonicity conditions, then the thickening
  of the cells of $G$ is a clique-Helly graph and each maximal clique
  of the thickening corresponds to a cell of $X$.
\end{proposition}

The 3-piece condition from Theorem~\ref{t:main-graphs}(1) and
the 3-cell condition from Proposition~\ref{prop-cell-helly} can be
viewed as generalizations of Gromov's flagness condition for \catz
cube complexes~\cite{Gr} and as a strengthening of Gilmore's condition
for conformality of hypergraphs~\cite{Berge}.

\subsection{Historical note and general context}
As already mentioned, injective metric spaces have been introduced by
Aronszajn and Panitchpakdi~\cite{AroPan1956} and they show the
equivalence between injective metric spaces and hyperconvex
spaces. Isbell~\cite{Isbell64} proves that for any metric space
$(X,d)$ there exists a smallest injective space which contains $(X,d)$
as an isometric subspace. This smallest injective space is called the
injective hull of $(X,d)$. Later, this result was independently
rediscovered by Dress~\cite{Dress1984} and for finite metric spaces by
Chrobak and Larmore~\cite{ChLa94}.

Dress provided other characterizations of injective hulls and
developed the theory of combinatorial dimension of injective hulls
viewed as cell complexes.  This concept of dimension was further
developed by Lang~\cite{Lang2013} who was also the first to use
injective metric spaces in the context of geometric group theory. Lang
also introduced the important concept of $\beta$-stable
intervals~\cite{Lang2013} and showed that the injective hulls of
locally finite graphs with $\beta$-stable intervals are proper and
have the structure of a locally finite polyhedral complex with
finitely many isometry types of cells of each dimension.
This result of Lang is particularly important in the proof of
Theorem~\ref{t:classes}(\ref{t:classes(1)}) and
Theorem~\ref{t:prop-coarse-Helly-groups-Helly}. In these proofs, we
also use his concept of bounded distance property~\cite{Lang2013},
that we show to be equivalent to the coarse Helly property introduced
in~\cite{ChEs}. As a matter of fact, $\delta$-hyperbolic geodesic
spaces and graphs satisfy the bounded distance
property~\cite{Lang2013} and the coarse Helly property~\cite{ChEs}.

The fact that CAT(0) cubical groups are Helly
(Theorem~\ref{t:classes}(\ref{t:classes(2)})) follows from the
bijection between CAT(0) cube complexes and median
graphs~\cite{Chepoi2000,Roller1998} and the result of Bandelt and van
de Vel~\cite{BavdV3} establishing that the thickenings of median
graphs are Helly graphs.  This result was generalized in~\cite{CCHO} to swm-graphs, thus yielding Theorem~\ref{t:classes}(\ref{t:classes(4)}).

The Helly property is an ubiquitous property in combinatorics which is
captured by the concept of Helly hypergraphs~\cite{Berge}. Berge and
Duchet~\cite{BeDu75} presented a simple ``local'' characterization of
Helly hypergraphs that is useful to show that the maximal cliques of a
graph satisfy the Helly property. This result and the local-to-global
characterization of Helly graphs of Chalopin et al.~\cite{CCHO}
provide a useful tool to establish the Hellyness of a graph. This
method is used in the proof of Theorems~\ref{t:classes}
and~\ref{t:main-graphs} and of Proposition~\ref{prop-cell-helly}.

Besides the local-to-global characterization of Helly graphs, other
characterizations of Helly graphs have been obtained earlier in the
papers~\cite{BP-absolute,BaPr,HeRi} (see Theorem~\ref{Helly1}). 
The proof of
Theorem~\ref{t:properties1}(\ref{t:properties1(2)})-(\ref{t:properties1(9)})
uses other properties of Helly graphs and injective spaces.
Theorem~\ref{t:properties1}(\ref{t:properties1(2)}) follows from
Polat's~\cite{Polat1993} fixed point result for Helly
graphs. Theorem~\ref{t:properties1}(\ref{t:properties1(4)}) uses the
fact that Helly graphs are dismantlable~\cite{BP-absolute} and that
fixed point sets in dismantlable graphs are
contractible~\cite{BarmakMinian2012}. The proof of
Theorem~\ref{t:properties1}(\ref{t:properties1(3)}) relies on the
characterization of (Gromov) hyperbolic weakly modular graphs
of~\cite{CCHO,ChDrEsHaVa}.

Theorem~\ref{t:properties1}(\ref{t:properties1(5)}) follows from the
fact that a geometric action on a Helly graph extends to a geometric
action on its injective hull. The second assertion then follows from a
property of injective spaces of finite combinatorial dimension
established by Descombes and Lang~\cite{DesLang2015} that they admit a
convex geodesic
bicombing. Theorem~\ref{t:properties1}(\ref{t:properties1(6)})-(\ref{t:properties1(9)})
follows from the existence of this geodesic bicombing and results
established in~\cite{DesLang2015,KasRup2017,FukayaOguni2019}.

To establish the biautomaticity of Helly groups
(Theorem~\ref{t:properties1}(\ref{t:properties1(5)})), we use the
technique introduced by {\'S}wi{\c{a}}tkowski~\cite{Swiat} of locally
recognized path systems in a graph. In this setting, one design a
canonical path system satisfying a combinatorial bicombing property
(this bicombing is different from the convex geodesic bicombing of~\cite{DesLang2015}). 
That groups acting geometrically on Helly graphs are different from 
groups acting on injective spaces follows from the recent 
result of Hughes and Valiunas,~\cite{HuVa} showing 
that there exist groups acting geometrically on injective spaces
that are neither Helly nor biautomatic.

\subsection{Discussion of consequences of main results}\label{s:intro_discuss}
Biautomaticity is an important algorithmic property of a group. It implies, among others, that the Dehn
function is at most quadratic, and that the Word Problem and the Conjugacy Problem are solvable; see e.g.\ \cite{ECHLPT}.

Biautomaticity of classical  \cftf small cancellation groups was proved in \cite{GerShort1990}. Our results (Theorem~\ref{t:classes}(\ref{t:classes(3)}) and Theorem~\ref{t:properties1}(\ref{t:properties1(1)})) imply biautomaticity in the more general graphical small cancellation case.

Biautomaticity of all FC-type Artin groups is a new result of the current paper together with \cite{HuangO}. Also new are the solution to the Conjugacy Problem and the quadratic bound on the Dehn function.
Altobelli \cite{Alto1998} showed that FC-type Artin groups are asynchronously automatic, and hence have solvable Word Problem.
Biautomaticity for few classes of Artin groups was
shown before in \cite{Pride1986,GerShort1990,Charney1992,Peifer1996,BradyMcCammond2000,HuaOsa1} (see \cite[Subsection 1.3]{HuangO} for a more detailed account).

Although the classical \cftf small cancellation groups have
been thoroughly investigated and quite well understood (see e.g.\
\cite{LySch,GerShort1990}), there was no nonpositive curvature
structure similar to \catz known for them. Note that Wise
\cite{Wisecub} equipped groups satisfying the stronger $\mathrm{B}(4)\mathrm{-T}(4)$ small
cancellation condition with a structure of a \catz cubical group,
but the question of a similar cubulation of \cftf groups is open
\cite[Problem 1.4]{Wisecub}. Our results --
Theorem~\ref{t:properties1} and
Theorem~\ref{t:classes}(\ref{t:classes(3)}) -- equip such groups with
a structure of a group acting geometrically on an injective metric
space.  This allows us to conclude that the Farrell-Jones conjecture and
the coarse Baum-Connes conjecture hold for them.  These results are
new, and moreover, we prove them in the much more general setting of
graphical small cancellation. Note that -- although quite similar in
definition and basic tools -- the graphical small cancellation
theories provide examples of groups not achievable in the classical
setting (see e.g.\ \cite{Osajda-scliga,Osajda-rfneg,OsajdaPrytula} for
details and references).

Important examples to which our theory applies are presented in \cite{HuangO}. These -- besides the FC-type Artin groups mentioned above -- are the weak Garside groups of finite type. This class includes among others: fundamental groups of the complements of complexified finite simplicial arrangements of
	hyperplanes, spherical Artin groups, braid groups of well-generated complex reflection groups, structure groups of non-degenerate, involutive and braided set-theore\-ti\-cal solutions of
	the quantum Yang-Baxter equation, one-relator groups with non-trivial center and, more generally, tree products of cyclic groups.  To our best knowledge there were no other `$\mathrm{CAT}(0)$-like' structures known for these groups before. Consequently, such results as the existence of an EZ-structure, the validity of the Farrell-Jones conjecture and of the coarse Baum-Connes conjecture obtained by using our approach are new in these settings.

Yet another class to which our theory applies and provides new results are quadric groups introduced and investigated in \cite{HodaQuadric}. See e.g.\ \cite[Example 1.4]{HodaQuadric} for a class of quadric groups that are a priori neither \catz cubical nor \cftf small cancellation groups.

Finally, we believe that many other groups are Helly -- see the discussion in Section~\ref{s:questions}.
Proving Hellyness of those groups would equip them with a very rich discrete and continuous structures, and would immediately
imply a plethora of strong features described above. On the other hand, there are still many other properties to be discovered, with the hope that most \catz results can be shown in this setting.

\subsection{Organization of the article and further results}
\label{s:intro_org}
\medskip

The proofs of items (\ref{t:classes(1)})--(\ref{t:classes(4)}) in Theorem~\ref{t:classes} are provided as follows. Item (\ref{t:classes(1)}) follows from Proposition~\ref{prop-proper-inj-hull-Helly} and Corollary~\ref{cor-hyperbolic-groups-Helly}. Items (\ref{t:classes(2)}) and (\ref{t:classes(4)}) follow from
Proposition~\ref{p:ccswmhc} and Corollary~\ref{c:buildC}. Item (\ref{t:classes(3)}) is Corollary~\ref{c:c4t4}.
\medskip

The coarse Helly property is discussed in Subsection~\ref{s:coarseHelly}, and the proof of Theorem~\ref{t:prop-coarse-Helly-groups-Helly} (appearing as Proposition~\ref{prop-coarse-Helly-groups-Helly} in the text) is presented in Subsection~\ref{s:hypquad}.

\medskip

The proofs of items (\ref{t:operations(1)})--(\ref{t:operations(5)}) in Theorem~\ref{t:operations} are provided as follows. Item (\ref{t:operations(1)}) is proved in Subsection~\ref{s:free_prod}. Items (\ref{t:operations(2)})--(\ref{t:operations(4)}) are consequences of Theorem~\ref{thm:HellyGeneral_int} (i.e., Theorem~\ref{thm:HellyGeneral} in the text) and are shown
in Subsection~\ref{s:qmedian}. There, we also show more general results: Theorem~\ref{thm:DiagProducts} on diagram products of Helly groups, and Theorem~\ref{thm:RAGG} on right-angled graphs of Helly groups.
Item (\ref{t:operations(5)}) follows directly from Theorem~\ref{t:quotient}.
\medskip

Theorem~\ref{thm:HellyGeneral_int} is the same as Theorem~\ref{thm:HellyGeneral} in the main body of the article and is discussed and proved in Subsection~\ref{s:qmedian}.
\medskip

The proofs of items (\ref{t:properties1(1)})--(\ref{t:properties1(9)}) in Theorem~\ref{t:properties1} are provided as explained below.
The proof of (\ref{t:properties1(1)}) is presented in Section~\ref{s:biautomatic}.
Item (\ref{t:properties1(2)}) follows from the Fixed Point Theorem~\ref{t:fixedpt}, and is proved in Subsection~\ref{s:fixedpts}. The proof of (\ref{t:properties1(3)}) is presented in Subsection~\ref{s:flats-hyper}. Item (\ref{t:properties1(4)}) follows from Corollary~\ref{c:EG} in Subsection~\ref{s:contrfix}, (\ref{t:properties1(5)}) follows from Theorems~\ref{t:injbicomb} and~\ref{t:helly=inj}, and (\ref{t:properties1(6)}), (\ref{t:properties1(7)}), (\ref{t:properties1(8)}), (\ref{t:properties1(9)}) are proved, respectively, in Subsections~\ref{s:EZ}, \ref{s:FJC}, \ref{s:coarseBC}, \ref{s:AsCones}.
\medskip

The proofs of items (\ref{t:main-graphs(1)})--(\ref{t:main-graphs(4)})
in Theorem~\ref{t:main-graphs} are provided as follows.  A union of
graph-products (UGP) is defined and studied in
Subsection~\ref{s:dirandamal}, and (\ref{t:main-graphs(1)}) is a part
of Theorem~\ref{th-UGP-3piece} proved there.  Graphical small
cancellation complexes are studied in Subsection~\ref{s:c4t4}, and
(\ref{t:main-graphs(2)}) is proved there as Theorem~\ref{l:thicksc}.
Rips complexes and face complexes are discussed in, respectively,
Subsection~\ref{s:Rips} and \ref{s:face}, and (\ref{t:main-graphs(3)})
is shown there. We discuss nerve complexes and prove
(\ref{t:main-graphs(4)}) in Subsection~\ref{s:nerve}. This result is
used in Subsection~\ref{s:7sys} to establish that $7$-systolic groups
are Helly.  \medskip

In Section~\ref{sec-abstract}, we introduce the 3-cell and the graded
monotonicity conditions and we establish that flag simplicial
complexes, CAT(0) cube complexes, hypercellular complexes and
swm-complexes satisfy both conditions.
Proposition~\ref{prop-cell-helly} then follows from
Proposition~\ref{graph-cell-complex}.  \medskip

Due to the relevance to the subject of our paper, in Section \ref{s:bahgraphs} we present in details the Helly property in the general setting of hypergraphs (set systems). We also discuss the conformality property for hypergraphs, which is
dual to the Helly property and which is an analog of flagness for simplicial complexes. For the same reason, in Section \ref{s:injective_isbell} we present the main ideas of Isbell's proof of the existence of injective hulls.

 \section{Preliminaries}
\label{s:prel}

\subsection{Graphs}\label{s:bagraphs}

A \emph{graph} $G=(V,E)$ consists of a set of vertices $V:=V(G)$ and a set of edges $E:=E(G)\subseteq V\times V$. All graphs considered in this paper are
undirected, connected, contain no multiple edges, no loops, are not
necessarily finite, but will be supposed to be locally finite. (With the exception of the quasi-median graphs considered in Section \ref{s:qmedian}, which are allowed to be locally infinite.) That is, they are \emph{locally finite one-dimensional simplicial complexes}.
For two distinct vertices $v,w\in V$ we write $v\sim w$ (respectively, $v\nsim w$) when there is an (respectively, there is no) edge connecting
$v$ with $w$, that is, when $vw:=\{v,w \} \in E$.
For vertices $v,w_1,\ldots,w_k$, we write $v\sim w_1,\ldots,w_k$ (respectively, $v\nsim w_1,\ldots,w_k$) or $v\sim A$ (respectively,
$v\nsim A$) when $v\sim w_i$ (respectively, $v\nsim w_i$), for each $i=1,\ldots, k$, where $A=\{ w_1,\ldots,w_k\}$.
As maps between graphs $G=(V,E)$ and $G'=(V',E')$ we always consider \emph{simplicial maps}, that is functions of the form
$f\colon V\to V'$ such that if $v\sim w$ in $G$ then $f(v)=f(w)$ or $f(v)\sim f(w)$ in $G'$.
A $(u,w)$--path $(v_0=u,v_1,\ldots,v_k=w)$ of \emph{length} $k$ is a sequence of vertices with $v_i\sim v_{i+1}$. If $k=2,$
then we call $P$ a \emph{2-path} of $G$. If $x_i\ne x_j$ for $|i-j|\ge 1$, then $P$ is
called a \emph{simple $(a,b)$--path}.
A \emph{$k$--cycle} $(v_0,v_1,\ldots,v_{k-1})$ is a path $(v_0,v_1,\ldots,v_{k-1},v_0)$.
For a subset
$A\subseteq V,$ the subgraph of $G=(V,E)$  \emph{induced by} $A$
is the graph $G(A)=(A,E')$ such that $uv\in E'$ if and only if $uv\in E$
($G(A)$ is sometimes called a \emph{full subgraph} of $G$). A \emph{square} $uvwz$ (respectively, \emph{triangle} $uvw$) is an induced $4$--cycle $(u,v,w,z)$ (respectively, $3$--cycle $(u,v,w)$).
The {\em wheel} $W_k$ is a graph obtained by connecting a single
vertex -- the {\em central vertex} $c$ -- to all vertices of the
$k$--cycle $(x_1,x_2, \ldots, x_k)$.

The \emph{distance} $d(u,v)=d_G(u,v)$ between two vertices $u$ and $v$ of a graph $G$ is the
length of a shortest $(u,v)$--path.  For a vertex $v$ of $G$ and an integer $r\ge 1$, we  denote  by $B_r(v,G)$ (or by $B_r(v)$)
the \emph{ball} in $G$
(and the subgraph induced by this ball)  of radius $r$ centered at  $v$, that is,
$B_r(v,G)=\{ x\in V: d(v,x)\le r\}.$ More generally, the $r$--\emph{ball  around a set} $A\subseteq V$
is the set (or the subgraph induced by) $B_r(A,G)=\{ v\in V: d(v,A)\le r\},$ where $d(v,A)=\mbox{min} \{ d(v,x): x\in A\}$.
As usual, $N(v)=B_1(v,G)\setminus\{ v\}$ denotes the set of neighbors of a vertex
$v$ in $G$.
A graph $G=(V,E)$
is \emph{isometrically embeddable} into a graph $H=(W,F)$
if there exists a mapping $\varphi : V\rightarrow W$ such that $d_H(\varphi (u),\varphi
(v))=d_G(u,v)$ for all vertices $u,v\in V$.

A \emph{retraction}
$\varphi$ of a graph $G$ is an idempotent nonexpansive mapping of $G$ into
itself, that is, $\varphi^2=\varphi:V(G)\rightarrow V(G)$ with $d(\varphi
(x),\varphi (y))\le d(x,y)$ for all $x,y\in W$ (equivalently, a retraction is a
simplicial idempotent map $\varphi: G\rightarrow G$). The subgraph of $G$
induced by the image of $G$ under $\varphi$ is referred to as a \emph{retract} of $G$.

The \emph{interval}
$I(u,v)$ between $u$ and $v$ consists of all vertices on shortest
$(u,v)$--paths, that is, of all vertices (metrically) \emph{between} $u$
and $v$: $I(u,v)=\{ x\in V: d(u,x)+d(x,v)=d(u,v)\}.$  An induced
subgraph of $G$ (or the corresponding vertex-set $A$) is called \emph{convex}
if it includes the interval of $G$ between any pair of its
vertices. The smallest convex subgraph containing a given subgraph $S$
is called the \emph{convex hull} of $S$ and is denoted by conv$(S)$. An induced subgraph $H$ (or the corresponding vertex-set of $H$)
of a graph $G$
is \emph{gated}~\cite{DrSch} if for every vertex $x$ outside $H$ there
exists a vertex $x'$ in $H$ (the \emph{gate} of $x$)
such that  $x'\in I(x,y)$ for any $y$ of $H$. Gated sets are convex and
the intersection of two gated sets is
gated. By Zorn's lemma there exists a smallest gated subgraph
$\lgate S \rgate$ containing a given subgraph $S$, called the
\emph{gated hull} of $S$.

Let $G_{i}$, $i \in \Lambda$ be an arbitrary family of graphs. The
\emph{Cartesian product} $\prod_{i \in \Lambda} G_{i}$ is a graph whose vertices
are all functions $x: i \mapsto x_{i}$, $x_{i} \in V(G_{i})$ and
two vertices $x,y$ are adjacent if there exists an index $j \in \Lambda$
such that $x_{j} y_{j} \in E(G_{j})$ and $x_{i} = y_{i}$ for all $i
\neq j$. Note that a Cartesian product of infinitely many nontrivial
graphs is disconnected. Therefore, in this case the connected
components of the Cartesian product are called \emph{weak Cartesian products}. The \emph{direct
product}  $\boxtimes_{i \in \Lambda} G_{i}$  of graphs  $G_{i}$, $i \in \Lambda$ is a graph having the same set of vertices as the Cartesian product
and two vertices $x,y$ are adjacent if
$x_{i} y_{i} \in E(G_{i})$ or $x_{i} = y_{i}$ for all $i\in \Lambda$.

We continue with  definitions of weakly modular graphs and their subclasses. We follow the  paper \cite{CCHO} and the survey \cite{BaCh}.
Recall  that a graph is \emph{weakly modular} if it satisfies  the following two distance conditions (for every $k > 0$):
\begin{itemize}
	\item \emph{Triangle condition} (TC): For any vertex $u$ and
	any two adjacent vertices $v,w$ at distance $k$ to $u$,
	there exists a common neighbor $x$ of $v,w$ at distance $k-1$ to $u$.
	\item \emph{Quadrangle condition} (QC):
	For any vertices $u,z$ at distance $k$ and any two neighbors $v,w$ of $z$ at distance
	$k-1$ to $u$, there exists a common neighbor $x$ of $v,w$ at distance $k-2$ from $u$.
\end{itemize}

Vertices $v_1,v_2,v_3$ of a graph $G$ form a {\it metric triangle}
$v_1v_2v_3$ if the intervals $I(v_1,v_2), I(v_2,v_3),$ and
$I(v_3,v_1)$ pairwise intersect only in the common end-vertices, that is,
$I(v_i, v_j) \cap I(v_i,v_k) = \{v_i\}$ for any distinct $1 \leq i, j,  k \leq 3$.
If $d(v_1,v_2)=d(v_2,v_3)=d(v_3,v_1)=k,$ then this metric triangle is
called {\it equilateral} of {\it size} $k.$ A metric triangle
$v_1v_2v_3$ of $G$ is a {\it quasi-median} of the triplet $x,y,z$
if the following metric equalities are satisfied:
\begin{align*}
d(x,y)&=d(x,v_1)+d(v_1,v_2)+d(v_2,y),\\
d(y,z)&=d(y,v_2)+d(v_2,v_3)+d(v_3,z),\\
d(z,x)&=d(z,v_3)+d(v_3,v_1)+d(v_1,x).
\end{align*}
If $v_1$, $v_2$, and $v_3$ are the same vertex $v$,
or equivalently, if the size of $v_1v_2v_3$ is zero,
then this vertex $v$ is called a {\em median} of $x,y,z$.
A median may not exist and may not be unique.
On the other hand, a quasi-median of every $x,y,z$ always exists:
first select any vertex $v_1$ from $I(x,y)\cap I(x,z)$ at maximal
distance to $x,$ then select a vertex $v_2$ from $I(y,v_1)\cap
I(y,z)$ at maximal distance to $y,$ and finally select any vertex
$v_3$ from $I(z,v_1)\cap I(z,v_2)$ at maximal distance to $z.$
The following characterization of weakly modular graphs holds:

\begin{lemma} \cite{Ch_metric} \label{lem-weakly-modular} A graph $G$
  is weakly modular if and only if for any metric triangle $v_1v_2v_3$
  of $G$ and any two vertices $x,y\in I(v_2,v_3),$ the equality
  $d(v_1,x)=d(v_1,y)$ holds. In particular, all metric triangles of weakly modular graphs are equilateral.
\end{lemma}

In this paper we use some classes of weakly modular graphs defined either by forbidden isometric or
induced subgraphs or by restricting the size of the metric
triangles of $G$.

A graph is called {\it median} if
$|I(u,v)\cap I(v,w)\cap I(w,v)|=1$ for every triplet $u,v,w$ of
vertices, that is, every triplet of vertices has a unique median.
Median graphs can be characterized in several different ways and they
play an important role in geometric group theory.  By a result of
\cite{Chepoi2000,Roller1998}, median graphs are exactly the
$1$-skeletons of \catz cube complexes (see below).  For other
properties and characterizations of median graphs, see the survey
\cite{BaCh}; for some other results on \catz cube complexes, see the
paper \cite{Sag}.

A graph is called {\it modular} if
$I(u,v)\cap I(v,w)\cap I(w,v)\ne \emptyset$ for every triplet $u,v,w$
of vertices, that is, every triplet of vertices admits a
median. Clearly a median graph is modular.  In view of Lemma
\ref{lem-weakly-modular}, modular graphs are weakly modular. Moreover,
they are exactly the graphs in which all metric triangles have size 0.
The term ``modular'' comes from a connection to modular lattices:
Indeed, a lattice is modular if and only if its covering graph is
modular. A modular graph is {\em hereditary modular} if any of its
isometric subgraph is modular.  It was shown in \cite{Ba_hereditary}
that a graph is hereditary modular if and only if all its isometric cycles have
length $4$. A modular graph is called {\it strongly modular} if it does not contain $K^-_{3,3}$ as an isometric subgraph.

Those graphs contain {\it orientable modular graphs}, that is, modular graphs whose edges can be oriented is such a way that two opposite edges of any
square have the same orientation. For example, any median graph is orientable.

We will also consider a nonbipartite generalization of strongly
modular graphs, called {\em sweakly modular graphs} or {\em
  swm-graphs}, which are defined as weakly modular graphs without
induced $K_{4}^-$ and isometric $K_{3,3}^-$ ($K_{4}^-$ is $K_4$ minus
one edge and $K_{3,3}^-$ is $K_{3,3}$ minus one edge). The swm-graphs
have been introduced and studied in depth in \cite{CCHO}.  The cell
complexes of swm-graphs can be viewed as a far-reaching generalization
of \catz cube complexes in which the cubes are replaced by cells
arising from dual polar spaces.

According to Cameron \cite{Ca}, the dual polar spaces can be
characterized by the conditions (A1)-(A5), rephrased in \cite{BaCh} in
the following (more suitable to our context) way:
\begin{theorem} \cite{Ca} \label{th:cameron_dual_polar} A graph $G$ is the collinearity
graph of a dual polar space $\Gamma$ of rank $n$  if and only if the following axioms are satisfied:
\begin{itemize}
\item[(A1)] for any point $p$ and any line $\ell$ of $\Gamma$ (i.e., maximal clique of $G$), there
is a unique point of $\ell$ nearest to $p$ in $G$;
\item[(A2)] $G$ has diameter $n$;
\item[(A3$\&$4)] the gated hull $\lgate u,v\rgate$
of two vertices $u,v$ at distance $2$ has diameter $2$;
\item[(A5)] for every pair of nonadjacent vertices $u,v$ and
every neighbor $x$ of $u$ in $I(u,v)$ there exists a neighbor $y$ of $v$ in $I(u,v)$
such that $d(u,v)=d(x,y)=d(u,y)+1=d(x,v)+1$.
\end{itemize}
\end{theorem}

We call a (non-necessarily finite) graph $G$ a \emph{dual polar graph}
if it satisfies the axioms (A1),(A3$\&$A4), and (A5) of
Theorem~\ref{th:cameron_dual_polar}, that is, we do not require
finiteness of the diameter (axiom (A2)).  By \cite[Theorem~5.2]{CCHO},
dual polar graphs are exactly the thick weakly modular graphs not
containing any induced $K^-_4$ or isometric $K^-_{3,3}$ (a graph is
\emph{thick} if the interval between any two vertices at distance 2
has at least two other vertices).  A set $X$ of vertices of an
swm-graph $G$ is \emph{Boolean-gated} if $X$ induces a gated and thick
subgraph of $G$ (the subgraph induced by $X$ is called a
\emph{Boolean-gated subgraph} of $G$). It was established
in~\cite[Section~6.3]{CCHO} that a set $X$ of vertices of an swm-graph
$G$ is Boolean-gated if and only if $X$ is a gated set of $G$ that
induces a dual-polar graph.

A graph $G$ is called {\it pseudo-modular} if
any three pairwise intersecting balls of $G$ have a non-empty
intersection \cite{BaMu_pmg}. This condition easily implies both the triangle
and quadrangle conditions, and thus pseudo-modular graphs are weakly modular. In fact,
pseudo-modular graphs are quite specific weakly modular graphs: from the definition
also follows that all metric triangles of pseudo-modular graphs have size 0 or 1.
Pseudo-modular graphs can be also characterized  by a single metric condition similar
to (but stronger than)  both triangle and quadrangle conditions:

\begin{proposition} \cite{BaMu_pmg} \label{pseudo-modular}
A graph $G$ is pseudo-modular if and only if for any three vertices $u,v,w$ such that
$1\le d(u,w)\le 2$ and $d(v,u)=d(v,w)=k\ge 2,$ there exists
a vertex  $x\sim u,w$ and $d(v,x)=k-1$.
\end{proposition}

An important subclass of pseudo-modular graphs is constituted by Helly
graphs, which is the main subject of our paper and which will be defined below.

The {\it quasi-median graphs}  are the $K^-_4$ and $K_{2,3}$--free weakly modular
graphs; equivalently, they are exactly the retracts of Hamming graphs (weak Cartesian products
of complete graphs). From the definition it follows that quasi-median graphs are pseudo-modular and swm-graphs.
For many results about quasi-median graphs, see \cite{quasimedian} and \cite{Qm} and for a theory of groups
acting on quasi-median graphs,   see \cite{Qm}.

Bridged graphs constitute another important subclass of weakly modular
graphs. A graph $G$ is called {\it bridged} \cites{FaJa,SoCh} if it
does not contain any isometric cycle of length greater than $3$.
Alternatively, a graph $G$ is bridged if and only if the balls
$B_r(A,G)=\{ v\in V: d(v,A)\le r\}$ around convex sets $A$ of $G$ are
convex.  Bridged graphs are exactly weakly modular graphs that do not
contain induced $4$-- and $5$--cycles (and therefore do not contain
$4$-- and $5$--wheels) \cite{Ch_metric}.
A graph $G$ (or its clique-complex $X(G)$) is called \emph{locally
  systolic} if the neighborhoods of vertices do not induce $4$- and
$5$-cycles.  If additionally, the clique complex $X(G)$ of $G$ is
simply connected, then the graph $G$ (or its clique-complex $X(G)$) is
called \emph{systolic}. If the neighborhoods of vertices of a
(locally) systolic graph $G$ do not induce 6-cycles, then $G$ is
called \emph{(locally) 7-systolic}.  It was shown in \cite{Chepoi2000}
that bridged graphs are exactly the $1$-skeletons of {\it systolic
  complexes} of \cite{JS}. In the following, we will use the name {\it
  systolic graphs} instead of {\it bridged graphs}.

A graph $G=(V,E)$ is called \emph{hypercellular}~\cite{ChKnMa19} if
$G$ can be isometrically embedded into a hypercube and $G$ does not
contain $Q^-_3$ as a partial cube minor ($Q^-_3$ is the 3-cube $Q_3$
minus one vertex).  A graph $H$ is called a \emph{partial cube minor}
of $G$ if $G$ contains a finite convex subgraph $G'$ which can be
transformed into $H$ by successively contracting some classes of
parallel edges of $G'$. Hypercellular graphs are not weakly modular
but however, they represent another generalization of median
graphs~\cite{ChKnMa19}.

\subsection{Complexes}\label{s:bacom}

All complexes considered in this paper are locally finite CW complexes. Following~\cite[Chapter 0]{Hat}, we call them simply \emph{cell complexes} or
just \emph{complexes}. If all cells are simplices (respectively, unit solid cubes) and the non-empty intersection of two cells is their common face,
then $X$ is called a {\em simplicial} (respectively, \emph{cube}) \emph{complex}. For a cell complex $X$, by $X^{(k)}$ we denote its \emph{$k$--skeleton}.
All cell complexes considered in this paper will have graphs (that is, one-dimensional simplicial complexes)
as their $1$--skeleta. Therefore, we use the notation $G(X):=X^{(1)}$.  The \emph{star} of a vertex $v$ in a complex $X$, denoted $\mr{St}(v,X)$, is
the set of all cells containing $v$.

An {\em abstract simplicial complex} $\Delta$ on a set $V$
is a set of non-empty subsets of $V$ such that each member of $\Delta$, called a {\em simplex},
is a finite set, and any non-empty subset of a simplex is also a simplex.
A simplicial complex $X$ naturally gives rise
to an abstract simplicial complex $\Delta$ on the set of vertices ($0$--dimensional cells) of $X$
by setting $U \in \Delta$ if and only if there is a simplex in $X$ having $U$ as its vertices.
Combinatorial and topological structures of $X$ are completely recovered from $\Delta$.
Hence we sometimes identify simplicial complexes and abstract simplicial complexes.

The \emph{clique complex} of a graph $G$ is the abstract simplicial
complex $X(G)$ having the cliques (i.e., complete subgraphs) of $G$ as
simplices. A simplicial complex $X$ is a \emph{flag simplicial complex}
if $X$ is the clique complex of its $1$--skeleton. Given a simplicial
complex $X$, the \emph{flag-completion} $\widehat{X}$ of $X$ is the
clique complex of the $1$--skeleton $G(X)$ of $X$.

Let $C$ be a cycle in the
$1$--skeleton of a complex $X$. Then a cell complex $D$ is called a
\emph{singular disk diagram} (or Van Kampen diagram) for $C$ if the
$1$--skeleton of $D$ is a plane graph whose inner faces are exactly
the $2$--cells of $D$ and there exists a cellular map
$\varphi:D\rightarrow X$ such that $\varphi|_{\partial D}=C$ (for more
details see~\cite[Chapter V]{LySch}). According to Van Kampen's
lemma~\cite[pp.~150--151]{LySch}, a cell complex $X$ is simply
connected if and only if for every cycle $C$ of $X$, one can construct
a singular disk diagram. A singular disk diagram with no cut vertices (that is, its
$1$--skeleton is $2$--connected) is called a \emph{disk diagram.} A
\emph{minimal (singular) disk} for $C$ is a (singular) disk diagram
${D}$ for $C$ with a minimum number of $2$--faces.  This number is
called the {\it (combinatorial) area} of $C$ and is denoted Area$(C).$
If $X$ is a simply connected triangle, (respectively, square,
triangle-square) complex, then for each cycle $C$ all inner faces in a
singular disk diagram $D$ of $C$ are triangles (respectively, squares,
triangles or squares).

As morphisms between cell complexes we always consider \emph{cellular maps},
that is, maps sending the $k$--skeleton into the $k$--skeleton.
An \emph{isomorphism} is a bijective
cellular map being a linear isomorphism (isometry) on each cell. A
\emph{covering (map)} of a cell complex $X$ is a cellular surjection $p\colon
\widetilde{X} \to X$ such that $p|_{\mbox{St}(\tv,\widetilde{X})}\colon \mbox{St}(\tv,\widetilde{X})\to \mbox{St}(p(\tv),X)$ is
an isomorphism for every vertex $\tv$ in $\tX$; compare \cite{Hat}*{Section 1.3}.
The space $\widetilde{X}$ is then called a \emph{covering space}.

\subsection{\catz spaces and Gromov hyperbolicity} Let $(X,d)$ be a metric space.
A {\it geodesic segment} joining two
points $x$ and $y$ from $X$ is an isometric embedding $\rho \colon {\mathbb R}^1\supset [0,\ell] \to X$
such that
$\rho(0)=x, \rho(\ell)=y$, and $d(\rho(t),\rho(t')) = |t - t'|$ for any $t,t' \in [0,\ell]$ ($d(x,y) = \ell$ is the \emph{length} of the geodesic $\rho$). A metric space $(X,d)$ is {\it geodesic} if every pair of
points in $X$ can be joined by a geodesic segment. Every
graph $G=(V,E)$ equipped with its standard
distance $d_G$ can be transformed into a geodesic
space $(X_G,d)$ by replacing every edge $e=uv$ by a segment
$\gamma_{uv}=[u,v]$ of length 1; the segments may intersect only at
common ends.  Then $(V,d_G)$ is isometrically embedded in a natural
way into $(X_G,d)$.

A {\it geodesic triangle} $\Delta (x_1,x_2,x_3)$ in a geodesic
metric space $(X,d)$ consists of three points in $X$ (the vertices
of $\Delta$) and a geodesic  between each pair of vertices (the
edges of $\Delta$). A {\it comparison triangle} for $\Delta
(x_1,x_2,x_3)$ is a triangle $\Delta (x'_1,x'_2,x'_3)$ in the
Euclidean plane  ${\mathbb E}^2=(\mathbb R^2,d_{2})$ such that $d_{2}(x'_i,x'_j)=d(x_i,x_j)$ for $i,j\in \{ 1,2,3\}.$ A geodesic
metric space $(X,d)$ is defined to be a {\it \catz space}
\cite{Gr} if all geodesic triangles $\Delta (x_1,x_2,x_3)$ of $X$
satisfy the comparison axiom of Cartan--Alexandrov--Toponogov:

\medskip\noindent
{\it If $y$ is a point on the side of $\Delta(x_1,x_2,x_3)$ with
vertices $x_1$ and $x_2$ and $y'$ is the unique point on the line
segment $[x'_1,x'_2]$ of the comparison triangle
$\Delta(x'_1,x'_2,x'_3)$ such that $d_{2}(x'_i,y')=
d(x_i,y)$ for $i=1,2,$ then $d(x_3,y)\le d_{2}(x'_3,y').$}

\medskip\noindent
The \catz property is also equivalent to the convexity of the function
$f\colon[0,1]\rightarrow X$ given by $f(t)=d(\alpha (t \ell_{\alpha}),\beta (t \ell_{\beta})),$
for any two geodesics $\alpha$ and $\beta$ of respective lengths $\ell_{\alpha}$ and $\ell_{\beta}$ (which is further equivalent to
the convexity of the neighborhoods of convex sets). This implies that
\catz spaces are contractible. Any two points of a \catz space can be
joined by a unique geodesic.  See the book \cite{BrHa} for a detailed
account on \catz spaces and their isometry groups.

A cube complex $X$ is \catz if $X$ endowed with the intrinsic $\ell_2$
metric is a \catz metric space.  Gromov~\cite{Gr} characterized \catz
cube complexes in a very nice combinatorial way: those are precisely
the simply connected cube complexes such that the following \emph{cube
  condition} holds: if three $(k+2)$-dimensional cubes intersect in a
$k$-dimensional cube and pairwise intersect in $(k+1)$-dimensional
cubes, then they are all three contained in a $(k+3)$-dimensional
cube. The cube condition is equivalent to the \emph{flagness
  condition} that states that the geometric link of any vertex is a
flag simplicial complex. The $1$-skeletons of \catz cube complexes are
precisely the median graphs~\cite{Chepoi2000,Roller1998}.

A metric space $(X,d)$ is $\delta$--{\it hyperbolic}
\cites{Gr,BrHa} if for any four points $u,v,x,y$ of
$X$, the two larger of the three distance sums $d(u,v)+d(x,y)$,
$d(u,x)+d(v,y)$, $d(u,y)+d(v,x)$ differ by at most $2\delta \geq 0$. A
graph $G = (V,E)$ is $\delta$--\emph{hyperbolic} if $(V,d_G)$ is
$\delta$--{hyperbolic}. A metric space or a graph has \emph{bounded hyperbolicity} if it is $\delta$--hyperbolic for some finite $\delta$.
For geodesic metric spaces and graphs, $\delta$--hyperbolicity can be
defined (up to the value of the hyperbolicity constant $\delta$) as spaces in which all geodesic triangles
are $\delta$--slim. Recall that  a geodesic triangle $\Delta(x,y,z)$ is called
$\delta$--{\it slim} if for
any point $u$ on the side $[x,y]$ the distance from $u$ to $[x,z]\cup
[z,y]$ is at most $\delta$. Equivalently, $\delta$--hyperbolicity can be defined
via the linear isoperimetric inequality: all cycles in a  $\delta$--hyperbolic
graph or  geodesic
metric space admit a disk diagram of linear area and vice-versa all graphs or geodesic
metric spaces  in which all  cycles   admit disk diagrams of linear area are hyperbolic.

\subsection{Group actions}
\label{s:bagroups}

For a set $X$ and a group $\Gamma$, a \emph{$\Gamma$--action on $X$} is a group homomorphism $\Gamma \to \mr{Aut}(X)$.
If $X$ is equipped with an additional structure then $\mr{Aut}(X)$ refers to the automorphisms group of this structure.
We say then that \emph{$\Gamma$ acts on $X$ by automorphisms}, and $x\mapsto gx$ denotes the automorphism being the image of $g$.
In the current paper $X$ will be a graph or a cell complex, and thus $\mr{Aut}(X)$ will denote graph automorphisms or cellular
automorphisms. Let $\Gamma$ be a group acting by automorphisms on a cell complex $X$. Recall that the action is \emph{cocompact}
if the orbit space $X/G$
is compact. The action of $\Gamma$ on a locally finite cell complex $X$
is \emph{proper} if stabilizers of cells are finite. Finally, the action is \emph{geometric}
(or $\Gamma$ \emph{acts geometrically} on $X$) if it is cocompact and proper. If a group $\Gamma$ acts geometrically on
a graph $G$ or on a cell complex $X$, then $G$ and $X$ are locally finite. This explains why in this paper we consider locally finite
graphs, complexes, and hypergraphs.

\subsection{Hypergraphs (set families)}\label{s:bahgraphs}
In this subsection, we recall the main notions in hypergraph
theory. We closely follow the book by Berge~\cite{Berge} on
hypergraphs (with the single difference, that our hypergraphs may be
infinite). A \emph{hypergraph} is a pair $\cH=(V,\cE)$, where $V$ is a
set and $\cE=\{ H_i\}_{i\in I}$ is a family of non-empty subsets of
$V$; $V$ is called the set of vertices and $\cE$ is called the set of
edges (or hyperedges) of $\cH$.  Abstract simplicial complexes are
examples of hypergraphs. The \emph{degree} of a vertex $v$ is the number of
edges of $\cH$ containing $v$.  A hypergraph $\cH$ is called {\it
  edge-finite} if all edges of $\cH$ are finite and {\it
  vertex-finite} if the degrees of all vertices are finite. $\cH$ is
called a \emph{locally finite hypergraph} if $\cH$ is edge-finite and
vertex-finite. A hypergraph $\cH$ is \emph{simple} if no edge of $\cH$
is contained in another edge of $\cH$. The \emph{simplification} of a
hypergraph $\cH=(V,\cE)$ is the hypergraph
$\breve{\cH}=(V,\breve{\cE})$ whose edges are the maximal by inclusion
edges of $\cH$.

The \emph{dual} of a hypergraph $\cH=(V,\cE)$ is the hypergraph
$\cH^*=(V^*,\cE^*)$ whose vertex-set $V^*$ is in bijection with the
edge-set $\cE$ of $\cH$ and whose edge-set $\cE^*$ is in bijection
with the vertex-set $V$, namely $\cE^*$ consists of all
$S_v=\{ H_j\in \cE: v\in H_j\}, v\in V$. By definition,
$(\cH^*)^*=\cH$. The dual of a locally finite hypergraph is also
locally finite.  The \emph{hereditary closure} $\widehat{\cH}$ of a
hypergraph $\cH$ is the hypergraph whose edge set is the set of all
non-empty subsets $F\subset V$ such that $F\subseteq H_i$ for at least
one index $i$. Clearly, the hereditary closure $\widehat{\cH}$ of a
hypergraph $\cH$ is a simplicial complex and
$\widehat{\cH}=\widehat{\breve{\cH}}$. The \emph{2-section} $[\cH]_2$
of a hypergraph $\cH$ is the graph having $V$ as its vertex-set and
two vertices are adjacent in $[\cH]_2$ if they belong to a common edge
of $\cH$. By definition, the 2-section $[\cH]_2$ is exactly the
1-skeleton $\widehat{\cH}^{(1)}$ of the simplicial complex
$\widehat{\cH}$ and the 2-section of $\cH$ coincides with the
2-section of its simplification $\breve{\cH}$. The \emph{line graph}
$L(\cH)$ of $\cH$ has $\cE$ as its vertex-set and $H_i$ and $H_j$ are
adjacent in $L(\cH)$ if and only if $H_i\cap H_j\ne\emptyset$.  By
definition (see also~\cite[Proposition~1, p.\ 32]{Berge}), the line
graph $L(\cH)$ of $\cH$ is precisely the 2-section $[\cH^*]_2$ of its
dual $\cH^*$. A \emph{cycle of length} $k$ of a hypergraph $\cH$ is a
sequence $(v_{1},H_1,v_2,H_2,v_3,\ldots, H_k,v_1)$ such that
$H_1,\ldots,H_k$ are distinct edges of $\cH$, $v_1,v_2,\ldots,v_k$ are
distinct vertices of $V$, $v_i,v_{i+1}\in H_i,$ $i=1,\ldots,k-1$, and
$v_k,v_1\in H_k$.  A \emph{copair hypergraph} is a hypergraph $\cH$ in
which $V\setminus H_i\in \cE$ for each edge $H_i\in \cE$.

The \emph{nerve complex} of a hypergraph  $\cH = (V,\cE)$ is
the simplicial complex $N(\cH)$ having $\cE$ as its vertex-set
such that a finite subset $\sigma \subseteq \cE$ is a simplex of
$N(\cH)$ if $\bigcap_{H_i \in \sigma} H_i \neq \emptyset$
(see~\cite{Bj}). The \emph{nerve graph} $NG(\cH)$ of a hypergraph $\cH$
is the $1$--skeleton of the nerve complex $N(\cH)$.
The following result is straightforward:
\begin{lemma} For any hypergraph $\cH$, $N(\cH)={\widehat{\cH}}^*$ and $NG(\cH)=[{\cH}^*]_2=({\widehat{\cH}}^*)^{(1)}$.
\end{lemma}

A family of subsets $\mathcal F$ of a set $V$ satisfies the
\emph{(finite) Helly property} if for any (finite) subfamily
${\mathcal F}'$ of $\mathcal F$, the intersection
$\bigcap {\mathcal F}'=\bigcap \{ F: F\in {\mathcal F}'\}$ is
non-empty if and only if $F\cap F'\ne \emptyset$ for any pair
$F,F'\in {\mathcal F}'$. A hypergraph $\cH=(V,\cE)$ is called
\emph{(finitely) Helly} if its family of edges $\cE$ satisfies the
(finite) Helly property. We continue with a characterization of Helly hypergraphs. In the
finite case this result is due to Berge and
Duchet~\cite{Berge,BeDu75}. The case of edge-finite hypergraphs
follows from a more general result~\cite[Proposition 1]{BaChEp}.

\begin{proposition}[\cite{Berge,BeDu75}]\label{Helly_triangle}
  An edge-finite hypergraph $\cH$ is Helly if and only if for any
  triplet $x,y,z$ of vertices the intersection of all edges containing
  at least two of $x,y,z$ is non-empty.
\end{proposition}

We call the condition in Proposition~\ref{Helly_triangle} the \emph{Berge-Duchet} condition.

A hypergraph $\cH=(V,\cE)$ is \emph{conformal} if all maximal cliques of the 2-section $[\cH]_2$ are edges of $\cH$. In other words,  $\cH$ is conformal
if and only if its hereditary closure $\widehat{\cH}$ is a flag simplicial complex. The following result establishes the duality between conformal
and Helly hypergraphs:

\begin{proposition}[\cite{Berge}*{p.\ 30}]\label{conformal-Helly}  A hypergraph $\cH$ is conformal if and only if its dual $\cH^*$ is Helly.
\end{proposition}

Analogously to the Helly property, the conformality can be characterized in a local way, via the following \emph{Gilmore condition}
(the proof follows from Propositions~\ref{Helly_triangle} and~\ref{conformal-Helly}):

\begin{proposition}[\cite{Berge}*{p.\ 31}]\label{Gilmore}  A vertex-finite hypergraph $\cH$ is conformal if and only if for any three
edges $H_1,H_2,H_3$ of $\cH$ there exists an edge $H$ of $\cH$ containing the set $(H_1\cap H_2)\cup (H_1\cap H_3)\cup (H_2\cap H_3)$.
\end{proposition}

A hypergraph $\cH$ is \emph{balanced}~\cite{Berge} if any cycle of
$\cH$ of odd length has an edge containing three vertices of the
cycle.  Balanced hypergraphs represent an important class of
hypergraphs with strong combinatorial properties (the K\"onig
property)~\cite{Berge,BeLV}.  It was noticed in~\cite[p.\ 179]{Berge}
that the finite balanced hypergraphs are at the same time Helly and
conformal; the duals of balanced hypergraphs are also balanced.  In
fact, those three fundamental properties still hold for a larger class
of hypergraphs: we call a hypergraph $\cH$ \emph{triangle-free} if any
cycle of $\cH$ of length three has an edge containing the three
vertices of the cycle, that is, for any three distinct vertices
$x,y,z$ and any three distinct edges $H_1,H_2,H_3$ such that
$x,y\in H_1, y,z\in H_2, z,x\in H_3$, one of the edges $H_1,H_2,H_3$
contains the three vertices $x,y,z$. Equivalently, a hypergraph $\cH$
is triangle-free if and only if it satisfies a stronger version of the
Gilmore condition: for any three edges $H_1,H_2,H_3$ of $\cH$ there
exists an edge $H_i$ in $\{H_1,H_2,H_3\}$ that contains
$(H_1\cap H_2)\cup (H_1\cap H_3)\cup (H_2\cap H_3)$. Since the dual of
a triangle-free hypergraph is also triangle-free, the following holds:

\begin{proposition}[\cite{BeLV,Berge}]\label{balanced}
  Locally finite triangle-free hypergraphs are conformal and Helly.
\end{proposition}

Another important class of Helly hypergraphs, extending the class of balanced hypergraphs is the class of normal hypergraphs. A hypergraph $\cH$ is
called \emph{normal}~\cite{Berge,Lov79} if it satisfied the Helly
property and its line graph $L(\cH)$ is perfect (i.e., by the Strong Perfect Graph Theorem, $L(\cH)$ does not contain odd cycles of length $>3$ and their complements as induced subgraphs).

With any graph $G=(V,E)$ one can associate several hypergraphs,
depending on the studied problem and of the studied class of
graphs. In the context of our current work, we consider the following
combinatorial and geometric hypergraphs: (1) the
\emph{clique-hypergraph} ${\mathcal X}(G)$ of all maximal cliques of
$G$, (2) the \emph{ball-hypergraph} ${\mathcal B}(G)$ of all balls of
$G$, and (3) the \emph{$r$-ball-hypergraph} ${\mathcal B}_r(G)$ of all
balls of a given radius $r$ of $G$.  The ball-hypergraph can be
considered for an arbitrary metric space $(X,d)$. The
clique-hypergraph ${\mathcal X}(G)$ of any graph $G$ is simple and
conformal and its hereditary closure $\widehat{{\mathcal X}}(G)$
coincides with the clique complex $X(G)$ of $G$. In the case of median
graphs $G$ (and \catz cube complexes), together with the cube complex
(cube hypergraph) an important role is played by the copair hypergraph
${\mathcal H}(G)$ of all halfspaces of $G$ (convex sets with convex
complements).  Since convex sets of median graphs are
gated~\cite[Theorem~1.22]{Isbell-med} and gated sets satisfy the
finite Helly property, the hypergraph ${\mathcal H}(G)$ is
finitely Helly. For a graph $G$ we will also consider the nerve complex $N(\cX(G))$ of
the clique-hypergraph $\cX(G)$ as well as the nerve complex
$N(\cB_r(G))$ of the $r$-ball-hypergraph $\cB_r(G)$ for $r \in \N$.

\subsection{Abstract cell complexes}\label{sec-abstract}

An \emph{abstract cell complex} $X$ (also called \emph{convexity
  space} or \emph{closure space}) is a locally finite hypergraph
$\cH(X)=(V,\cE)$ with $\emptyset\in \cE$ and whose edges are closed
under intersections, i.e., if $H_i,i\in I$ are edges $\cH$, then
$\cap_{i\in I} H_i$ is also an edge of $\cH(X)$. We call the edges of
$\cH(X)$ the \emph{cells} of $X$ and $\cH(X)$ the
\emph{cell-hypergraph} of $X$. The cells of $X$ contained in a given
cell $C$ are called the \emph{faces} of $C$. The faces of a cell $C$
ordered by inclusion define the face-lattice $F(C)$ of
$C$. $C'\varsubsetneq C$ is a \emph{facet} of $C$ if $C'$ is a maximal
by inclusion proper face of $C$; in other words, $C'$ is a coatom of
the face-lattice $F(C)$. The \emph{dimension} $\dim(C)$ of a cell $C$
is the length of the longest chain in the face-lattice of $C$.
Locally finite abstract simplicial complexes are examples of abstract
cell complexes.  In fact, simplicial complexes are the cell complexes
in which the face-lattices are Boolean lattices. The dimension of a
simplex with $d+1$ vertices is $d$. Cube complexes also lead to
abstract cell complexes: it suffices to consider the vertex-set of
each cube as an edge of the cell-hypergraph; the dimension of a cube
is the standard dimension.

Abstract cell complexes also arise from swm-graphs and hypercellular
graphs. The cells of an swm-graph are its Boolean-gated sets and the
dimension of a Boolean-gated set is its diameter. Observe that in a
swm-graph, any maximal clique is boolean-gated. In the corresponding
abstract cell complex, each such clique is a 1-dimensional cell whose
0-cells are the vertices of the clique.  It was shown in~\cite{CCHO}
that one can also associate a contractible geometric cell complex to
any swm-graph $G$, in which the cells are the orthoscheme complexes of
the Boolean-gated subgraphs of $G$. Note that the geometric dimension
of this geometric complex is larger than the dimension of the abstract
cell complex.  The cells of a hypercellular graph $G$ are the gated
subgraphs of $G$ which are the convex hulls of the isometric cycles of
$G$. It was shown in~\cite{ChKnMa19} that those cells are Cartesian
products of edges and even cycles. It was established
in~\cite{ChKnMa19} that the geometric realization of the abstract cell
complex of a hypercellular graph is contractible. The dimension of
such a cell is the number of edge-factors plus twice the number of
cycle-factors. Notice that swm-graphs and hypercellular graphs
represent two far-reaching and quite different generalizations of
median graphs. Swm-graphs do not longer have hyperplanes (i.e.,
classes of parallel edges) and halfspaces, and their cells
(Boolean-gated subgraphs) have a complex combinatorial structure;
nevertheless, they are still weakly modular and admit a
local-to-global characterization. On the other hand, hypercellular
graphs are no longer weakly modular but they still admit hyperplanes
(whose carriers are gated) and halfspaces, and each triplet of
vertices admit a unique median cell.

We say that an abstract cell complex $X$ satisfies the \emph{3-cell
  condition} if for any three cells $C_1,C_2,C_3$ such that
\begin{itemize}
\item $C_1\cap C_2$ is a facet of $C_1$ and $C_2$;
\item $C_1\cap C_3$ is a facet of $C_1$ and $C_3$;
\item $C_2\cap C_3$ is a facet of $C_2$ and $C_3$; 
\item $C_1\cap C_2 \cap C_3$ is a facet of $C_1\cap C_2$,
  $C_1\cap C_3$, and $C_2\cap C_3$;
\end{itemize}
then the union $C_1\cup C_2\cup C_3$ is contained in a common cell $C$
of $X$.  For cube complexes, observe that the 3-cell condition is
equivalent to the cube condition. Simplicial complexes do not always
satisfy the 3-cell condition, but we show in
Lemma~\ref{lem-3cell-simpl} that flag simplicial complexes
do. For hypercellular complexes, the $3$-cell condition has been
established in~\cite[Theorem~B]{ChKnMa19}. We establish that
swm-complexes satisfy the $3$-cell condition in
Lemma~\ref{lem-3cell-swm}.

We say that an abstract cell complex $X$ satisfies the \emph{graded
  monotonicity condition (GMC)} if for any cell $C$ of $X$ and any two
intersecting faces $A,B$ of $C$ with $B \not\subseteq A$, there exists
a face $D$ of $C$ such that $A$ is a facet of $D$ with
$\dim(D) = \dim(A)+1$ and $\dim(D\cap B) = \dim(A\cap B) +1$. Note
that (GMC) is only about intersecting subcells of a cell and in
particular, it does not imply that the face-lattice $F(C)$ of a cell
$C$ is graded (i.e., that all maximal chains in $F(C)$ have the same
length).  We establish that simplicial complexes, cube complexes,
hypercellular complexes, and swm complexes satisfy the graded
monotonicity condition in Lemmas~\ref{GMC}, \ref{GMC-hypercell},
and~\ref{GMC-swm}.

Since the cells of $X$ are finite, we can apply iteratively the graded monotonicity condition to get the following lemma:

\begin{lemma}\label{lem-GMC}
  If an abstract cell complex $X$ satisfies the graded monotonicity
  condition, then for any cells $A,B,C$ such that $A\cap B \neq \emptyset$,
  $A\cup B \subseteq C$, and $B \not\subseteq A$, there exists a face
  $E$ of $C$ such that $A \cup B \subseteq E$ with
  $\dim(E) - \dim(A) = \dim(E \cap B)- \dim(A\cap B)$.
\end{lemma}

We say that an abstract cell complex satisfies the \emph{Helly
  property for three cells} if any three pairwise intersecting cells
have a non-empty intersection.

\begin{proposition}\label{3-cell}
  If an abstract cell complex $X$ satisfies the 3-cell condition, the
  graded monotonicity condition, and the Helly property for three
  cells, then its cell-hypergraph $\cH(X)$ is conformal.
\end{proposition}

\begin{proof}
  We show that $\cH(X)$ satisfies the Gilmore condition. Let
  $C_1,C_2,C_3$ be three arbitrary cells of $X$. We proceed by
  induction on
  $\alpha(C_1,C_2,C_3) := \dim(C_1) +\dim(C_2) + \dim(C_3) - \dim(C_1
  \cap C_2) - \dim(C_1 \cap C_3) - \dim(C_2 \cap C_3)$ and then on
  $\beta(C_1,C_2,C_3) := |C_1|+|C_2|+|C_3|$.

  If any of the pairwise intersections $C_1 \cap C_2$, $C_1 \cap C_3$,
  $C_2 \cap C_3$ is empty, then
  $(C_1 \cap C_2) \cup (C_1 \cap C_3) \cup (C_2 \cap C_3)$ is
  contained in one of the three cells $C_1, C_2, C_3$. Thus, we
  suppose that the pairwise intersections are non-empty and, by the
  Helly property for three cells, we can assume that
  $C_1 \cap C_2 \cap C_3 \neq \emptyset$. If $C_1 \cap C_2 \cap C_3$
  is not a proper face of $C_1 \cap C_2$, i.e., if
  $C_1 \cap C_2 \cap C_3 = C_1 \cap C_2$, then
  $C_1 \cap C_2 \subseteq C_3$ and
  $(C_1 \cap C_2) \cup (C_1 \cap C_3) \cup (C_2 \cap C_3) \subseteq
  C_3$. Thus, we can assume that $C_1 \cap C_2 \cap C_3$ is a proper
  face of $C_1 \cap C_2$, and for similar reasons of $C_1 \cap C_3$
  and of $C_2 \cap C_3$.
If there exists a proper face $D_1$ of $C_1$ such that
  $(C_1 \cap C_2) \cup (C_1 \cap C_3) \subseteq D_1$, then
  $\alpha(D_1, C_2, C_3) < \alpha(C_1,C_2,C_3)$ and by the induction
  hypothesis applied to $D_1, C_2, C_3$, we are done. Thus we can
  assume that there is no proper face $D_i$ of $C_i$ such that
  $(C_i \cap C_j) \cup (C_i \cap C_k) \subseteq D_i$ for any
  $\{i,j,k\} = \{1,2,3\}$.
Suppose that $C_1 \cap C_2$ is a facet of $C_1$ and $C_2$,
  $C_1 \cap C_3$ is a facet of $C_1$ and $C_3$, $C_2 \cap C_3$ is a
  facet of $C_2$ and $C_3$. By GMC applied to the faces $C_1\cap C_2$
  and $C_1\cap C_3$ of $C_1$, there exists a face $D_1$ of $C_1$
  containing strictly $C_1 \cap C_2$ such that $C_1 \cap C_2 \cap C_3$
  is a facet of $D_1 \cap C_3$. Since $C_1 \cap C_2$ is a facet of
  $C_1$, necessarily, $D_1 = C_1$. Consequently,
  $C_1 \cap C_2 \cap C_3$ is a facet of $C_1 \cap C_3$ and for similar
  reasons of $C_1\cap C_2$ and $C_2\cap C_3$. Then the Gilmore
  property follows from the 3-cell condition applied to $C_1$, $C_2$,
  and $C_3$.  Therefore, without loss of generality, we can suppose
  that $C_1 \cap C_2$ is not a facet of $C_1$.

  By GMC applied to the faces $C_1 \cap C_2$ and
  $C_1 \cap C_3$ of $C_1$, there exists a face $D_1$ of $C_1$ such
  that $C_1 \cap C_2 \subsetneq D_1$,
  $\dim(D_1) = \dim (C_1 \cap C_2) + 1$, and
  $\dim(D_1 \cap C_3) = \dim (C_1 \cap C_2\cap C_3) + 1$. We assert that $\alpha(D_1,C_2,C_3)=\alpha(C_1,C_2,C_3)$.
  Indeed, first observe that
  \begin{align*}
    \alpha(C_1,C_2,C_3) - \alpha(D_1,C_2,C_3) =
    \dim(C_1) &- \dim(D_1) - \dim(C_1 \cap C_2) + \dim(D_1 \cap C_2) \\
    &- \dim(C_1 \cap C_3) +
      \dim(D_1 \cap C_3).
  \end{align*}
  Note that $D_1 \cap C_2 = C_1 \cap C_2$. Moreover, by applying
  Lemma~\ref{lem-GMC} to $D_1$, $C_1 \cap C_3$, and $C_1$, we can find a
  face $E_1$ of $C_1$ such that
  $D_1 \cup (C_1 \cap C_3) \subseteq E_1$ and
  $\dim(E_1) - \dim(D_1) = \dim(E_1 \cap C_3)- \dim(D_1 \cap
  C_3)$. Since $E_1$ cannot be a proper face of $C_1$, we conclude that $E_1 = C_1$,
  and thus
  $\dim(C_1) - \dim(D_1) = \dim(C_1 \cap C_3)- \dim(D_1 \cap C_3)$.
  Consequently, we get
$\alpha(C_1,C_2,C_3) = \alpha(D_1,C_2,C_3)$,
  establishing our assertion.
  Since $C_1 \cap C_2$ is a facet of $D_1$ but not of $C_1$, $D_1$ is
  a proper face of $C_1$ and thus
  $\beta(D_1,C_2,C_3) < \beta(C_1,C_2,C_3)$. Therefore we can apply the induction hypothesis
  to $D_1$, $C_2$, and $C_3$, and  conclude that there exists a cell $D'_2$ such
  that
  $(D_1 \cap C_2) \cup (D_1 \cap C_3) \cup (C_2 \cap C_3) \subseteq
  D'_2$.

  We assert that $C_2 \subsetneq D'_2$. Indeed, $C_2 \cap D'_2$ is a
  face of $C_2$ containing $(C_1\cap C_2) \cup (C_2 \cap C_3)$ and
  since it cannot be a proper face of $C_2$, we have $C_2 \cap D'_2 =
  C_2$. Since
  $C_1 \cap C_2 \cap C_3 \subsetneq D_1 \cap C_3 \subseteq D'_2$, the
  inclusion of $C_2$ in $D_2'$ is strict. We apply GMC to $C_2$, $D_1 \cap C_3$, and $D_2'$ to get a
  face $D_2$ of $D_2'$ such that $C_2 \subsetneq D_2$, $\dim(D_2) =
  \dim(C_2)+1$, $\dim(D_2 \cap D_1 \cap C_3) = \dim(C_2 \cap D_1 \cap
  C_3) + 1$. Observe that
  \begin{align*}
    \alpha(C_1,C_2,C_3) - \alpha(C_1,D_2,C_3) =  \dim(C_2)
    &- \dim(D_2) - \dim(C_1 \cap C_2) + \dim (C_1 \cap D_2)\\
    &- \dim(C_2 \cap C_3) + \dim (D_2 \cap C_3).
  \end{align*}

  Since $C_1\cap C_2 \subsetneq C_1 \cap D_2$ and
  $C_2\cap C_3 \subsetneq D_2 \cap C_3$, we have
  $\dim (C_1 \cap D_2) - \dim(C_1 \cap C_2) \geq 1$ and
  $\dim (D_2 \cap C_3) - \dim(C_2 \cap C_3) \geq 1$. Since
  $\dim(D_2) - \dim(C_2) = 1$, we get
  $\alpha(C_1,C_2,C_3) - \alpha(C_1,D_2,C_3) \geq 1$. Therefore, we
  can apply the induction hypothesis to $C_1, D_2, C_3$ and find a
  cell $C$ containing
  $(C_1 \cap D_2) \cup (C_1 \cap C_3) \cup (D_2 \cap C_3)$. Since
  $D_2$ contains $C_2$, we conclude that $(C_1 \cap C_2) \cup (C_1 \cap C_3) \cup (C_2
  \cap C_3) \subseteq C$, and we are done.
\end{proof}

We now show that flag simplicial complexes satisfy the 3-cell
condition.

\begin{lemma}\label{lem-3cell-simpl}
  Flag simplicial complexes satisfy the 3-cell condition.
\end{lemma}

\begin{proof}
  Consider a flag simplicial complex $X$ and any three simplices
  $C_1, C_2, C_3$ such that $C_1 \cap C_2$ (respectively,
  $C_1 \cap C_3$, $C_2 \cap C_3$) is a facet of $C_1$ and $C_2$
  (respectively, $C_1$ and $C_3$, $C_2$ and $C_3$) and
  $C_1 \cap C_2 \cap C_3$ is a facet of
  $C_1\cap C_2, C_1\cap C_3, C_2\cap C_3$. If there exists
  $v \in C_1 \setminus (C_2 \cup C_3)$, then
  $C_1 \cap C_2 = C_1 \cap C_3 = C_1 \setminus \{v\}$ and
  $C_1 \cap C_2 \cap C_3$ is not a facet of $C_1 \cap C_2$ or
  $C_1 \cap C_3$. Consequently,
  $C_1 = (C_1 \cap C_2) \cup (C_1 \cap C_3)$ and similarly,
  $C_2 = (C_1 \cap C_2) \cup (C_2 \cap C_3)$ and
  $C_3 = (C_1 \cap C_3) \cup (C_2 \cap C_3)$. Therefore, any two
  vertices of $C_1 \cup C_2 \cup C_3$ both belong to a common $C_i$,
  $i \in \{1,2,3\}$.  Since $X$ is a
  flag simplicial complex,  $C_1 \cup C_2 \cup C_3$ is a simplex of
  $X$, establishing the 3-cell condition for $X$.
\end{proof}

We now establish that simplicial complexes, cube complexes, and
hypercellular complexes satisfy the graded monotonicity condition.

\begin{lemma}\label{GMC}
  Simplicial complexes and cube complexes satisfy the graded
  monotonicity condition.
\end{lemma}

\begin{proof}
  We need to show that for any cell $C$ of $X$, if $A,B$ are two
  intersecting faces of $C$ with $B \not\subseteq A$, then there
  exists a face $D$ of $C$ such that $A$ is a facet of $D$ with
  $\dim(D) = \dim(A)+1$ and $\dim(D\cap B) = \dim(A\cap B) +1$.
If $X$ is a simplicial complex, as $D$ it suffices to take
  $A\cup \{ x\}$ for any $x\in B\setminus A$. If $X$ is a cube
  complex, then as $D$ we can take the smallest face of $C$ containing
  $A\cup \{ x\}$, where $x$ is a vertex of $B\setminus A$ adjacent to
  a vertex of $A\cap B$ ($D$ can be viewed as the gated hull of
  $A\cup \{ x\}$). Such $x$ exists because $A$ and $B$ are convex and
  thus connected.  Indeed, from the definition of $D$ it follows that
  $A$ is a facet of $D$ and $A\cap B$ is a facet of $D\cap B$.
\end{proof}

\begin{lemma}\label{GMC-hypercell}
  Hypercellular complexes satisfy the graded monotonicity condition.
\end{lemma}

\begin{proof}
  Consider a cell $C$ in a hypercellular complex $X$. Then $C$, viewed
  as a graph, is the Cartesian product $C=F_1\square\cdots\square F_k$
  of even cycles and edges. Since each cell $C'$ of $C$ is a gated
  subgraph of $C$, $C'$ is a Cartesian product
  $F'_1\square\cdots\square F'_k$, where each $F'_i$ is a gated
  subgraph of $F_i$, $i=1,\ldots,k$. Since each proper gated subgraph
  of an even cycle is a vertex or an edge, each $F'_i$ either
  coincides with $F_i$ or is a vertex or an edge of $F_i$.  The
  dimension $\dim(C')$ of $C'=F'_1\square\cdots\square F'_k$ is the
  number of edge-factors $F'_i$ plus twice the number of cycle-factors
  $F'_i$.

  Let $A=F'_1\square\cdots\square F'_k$ and
  $B=F''_1\square\cdots\square F''_k$, where $F'_i$ and $F''_i$ are
  gated subgraphs of $F_i$. Notice also that
  $A\cap B=F'''_1\square\cdots\square F'''_k$, where
  $F'''_i=F'_i\cap F''_i$ for $i=1,\ldots,k$.  As for cube complexes,
  let $x$ be a vertex of $B\setminus A$ adjacent to a vertex $y$ of
  $A\cap B$ and suppose that the edge $xy$ of $C$ arises from the
  factor $F_j$.  Let $D$ be the gated hull of $A \cup \{ x\}$. Then
  one can see that $D=F'_1\square\cdots\square F^+_j\square\cdots\square F'_k$,
  where $F^+_j$ is the edge of $F_j$ corresponding to the edge $xy$ if
  $F'_j$ is a single vertex and $F^+_j=F_j$ if $F'_j$ is an edge.  One
  can also see that
  $D\cap B=F'''_1\square\cdots\square F^+_j\square\cdots\square F'''_k$. Therefore,
  $A$ is a facet of $D$ and $A\cap B$ is a facet of $D\cap B$.  This
  establishes that hypercellular complexes satisfy the graded
  monotonicity condition.
\end{proof}

We now establish that swm-complexes satisfy the 3-cell condition and
the graded monotonicity condition.  Recall that in swm-complexes the
cells are the Boolean-gated sets of the corresponding swm-graphs and
that they induce dual polar graphs. We first establish some useful
properties satisfied by the cells of swm-complexes.

\begin{lemma}\label{lem-diametralpair-swm}
  For any cell $A$ of an swm-graph $G$ and any $x \in A$, there exists
  $y \in A$ such that $A = \lgate x,y \rgate$.
\end{lemma}

\begin{proof} Since $A$ is a cell of $G$, $A$ is a gated set inducing a dual polar subgraph of $G$.
By~\cite[Lemma~5.12]{CCHO}, for any $x, y \in A$,
  $\lgate x,y \rgate = A$ if and only if $d(x,y) = \diam(A)$, where
  $\diam(A)$ is the diameter of $A$.

  Given a vertex $x \in A$, we choose $x',y' \in A$ such that
  $d(x',y') = \diam(A)$ and $d(x,x')$ is minimized. If $x = x'$, we are
  done by~\cite[Lemma~5.12]{CCHO}. Suppose now that $x \neq x'$. Pick
  a neighbor $u$ of $x'$ in $I(x',x)$. By our choice of $x'$ and $y'$
  and since $d(x,u) < d(x,x')$, we must have $d(u,y') = d(x',y')-1$,
  i.e., $u \in I(x',y')$.  But then, by the axiom (A5) of
  dual polar graphs, there exists $v \sim y'$ such that
  $d(u,v) = d(x',y')$, contradicting our choice of $x',y'$ since
  $d(u,x) < d(x,x')$.
\end{proof}

\begin{lemma}\label{lem-swm-gate}
  Consider two cells $A,B$ of an swm-graph $G$ such that
  $B \subseteq A$ and any two vertices $x \in B$ and $y \in A$. If
  $A = \lgate x,y \rgate$, then $B = \lgate x,y^* \rgate$ where $y^*$
  is the gate of $y$ on $B$.
\end{lemma}

\begin{proof}
  Let $\cS$ denotes the set of all maximal cliques of the gated dual
  polar subgraph $A$ of $G$.  Since dual polar graphs are
  $K^-_4$--free, $|K \cap K'| \leq 1$ for all $K,K' \in
  \mathcal{S}$. For a vertex $u$, let $\cS(u)$ denote the set of all
  maximal cliques of $A$ containing $u$.  For two vertices $u,v$ of
  $A$, let $\cS(u,v)$ denote the set of cliques $K$ of $\cS(u)$
  meeting $I(u,v) \setminus \{u\}$.  Note that
  $\mathcal{S}(u,u)=\emptyset$. The gated hull
  $\lgate\bigcup_{K \in {\mathcal S}(u,v)} K \rgate$ will be denoted
  by $\lgate{\mathcal S}(u,v)\rgate$. From~\cite[Lemmas 5.10 \&
  5.11]{CCHO}, we know that
  $\lgate u,v \rgate = \lgate \cS(u,v) \rgate = \{z \in A: \cS(u,z)
  \subseteq \cS(u,v)\}$ induces a dual polar graph of diameter
  $d(u,v)$.

  Since $B$ is gated and since $x,y^* \in B$, we have 
  $\lgate x, y^* \rgate \subseteq B$. In order to establish the
  reverse inclusion, we show that for any $z \in B$,
  $\cS(x,z) \subseteq \cS(x,y^*)$.  Since $x,z \in B$, $B$ is
  gated, and  $\bigcup_{K \in \cS(x,z)} K \subseteq \lgate \cS(x,z) \rgate =
  \lgate x,z \rgate \subseteq B$, any maximal clique $K \in \cS(x,z)$
  is contained in $B$. Pick any clique $K \in \cS(x,z)$.
  Since $z \in A =\lgate x,y \rgate$, we have
  $K \in \cS(x,z) \subseteq \cS(x,y)$. Thus there exists a neighbor
  $t$ of $x$ in $K \cap I(x,y)$. Since $t \in K \subseteq B$ and since
  $y^*$ is the gate of $y$ in $B$, we have $y^* \in I(t,y)$. Since
  $t \in I(x,y)$, we thus have $t \in I(x,y^*)$, yielding
  $K \in \cS(x,y^*)$. Consequently, $\cS(x,z) \subseteq \cS(x,y^*)$
  and thus $B = \lgate x, y^* \rgate$.
\end{proof}

\begin{lemma}\label{lem-3cell-swm}
  Swm-complexes satisfy the 3-cell condition.
\end{lemma}

\begin{proof}
  Consider three cells $C_1, C_2, C_3$ such that $C_1 \cap C_2$
  is a facet of $C_1$ and $C_2$,  $C_1 \cap C_3$ is a facet of $C_1$ and $C_3$,
  $C_2\cap C_3$ is a facet of $C_3$, and, finally, $C_1 \cap C_2 \cap C_3$ is a  facet of
  $C_1\cap C_2, C_1\cap C_3, C_2\cap C_3$. This implies
  that
  $\dim(C_1) = \dim(C_2) = \dim(C_3) = \dim(C_1 \cap C_2) +1 =\dim(C_1
  \cap C_3) +1 =\dim(C_2 \cap C_3) +1 = \dim (C_1\cap C_2\cap
  C_3)+2$. Let $k = \dim(C_1 \cap C_2)$.

  Since cells of swm-complexes are gated, they satisfy the Helly
  property and there exists $z \in C_1 \cap C_2 \cap C_3$. By
  Lemma~\ref{lem-diametralpair-swm}, there exists $u \in C_1$ such
  that $C_1 = \lgate u,z \rgate$, i.e., such that $d(u,z) =
  k+1$. Since $C_1 \cap C_2$ and $C_1 \cap C_3$ are Boolean-gated sets
  of diameter $k$, $u \notin C_2 \cup C_3$. Let $u_2$ and $u_3$ be the
  gates of $u$ in $C_2$ and $C_3$, respectively.  By
  Lemma~\ref{lem-swm-gate}, $C_1 \cap C_2 = \lgate z, u_2 \rgate$ and
  $C_1 \cap C_3 = \lgate z, u_3 \rgate$. Consequently,
  $d(z,u_2) = d(z,u_3) = k$ and $u \sim u_2,u_3$. Since
  $C_1\cap C_2 \cap C_3$ is a facet of $C_1 \cap C_2$ and
  $C_1\cap C_3$, necessarily $u_2 \notin C_3$, and $u_3 \notin C_2$,
  and thus $u_2 \neq u_3$.  By the quadrangle condition, there exists
  $v \sim u_2,u_3$ with $d(z,v) = k-1$. Since $C_1$, $C_2$ and $C_3$
  are gated and thus convex, $v \in C_1\cap C_2 \cap C_3$. By
  Lemma~\ref{lem-diametralpair-swm}, there exists $w \in C_2\cap C_3$
  such that $\lgate v,w \rgate = C_2\cap C_3$ and $d(v,w) = k$. Since
  $u_2 \notin C_3$, $u_2 \notin \lgate v,w \rgate = C_2 \cap C_3$ and
  since $v \sim u_2$, $v$ is the gate of $u_2$ on $C_2 \cap
  C_3$. Consequently, $d(w,u_2) = d(w,v)+1 = k+1$ and similarly
  $d(w,u_3) = k+1$.  Since $d(w,u_2) = d(w,u_3) = k+1$,
  $\lgate w,u_2 \rgate = C_2$ and $\lgate w,u_3 \rgate = C_3$
  by~\cite[Lemma~5.12]{CCHO}. Consequently, $\lgate w,u_2 \rgate$ and
  $\lgate w,u_3 \rgate$ are Boolean-gated sets of $G$. Since $C_2$ is
  gated, $u \notin C_2$, $u_2 \in C_2$ and $u \sim u_2$, we get that
  $d(w,u) = k+2$. By~\cite[Proposition~6.5 \& Lemma 6.6]{CCHO},
  $\lgate w,u \rgate$ is thus a Boolean-gated set of $G$ of diameter
  $k+2$.

  Since $w, u_2 \in I(w,u) \subseteq \lgate w,u \rgate$, we have
  $C_2 = \lgate w,u_2\rgate \subseteq \lgate w,u\rgate$ and similarly,
  $C_3 \subseteq \lgate w,u\rgate$. Since
  $z \in C_2 \subseteq \lgate w,u\rgate$ and since
  $C_1 = \lgate u,z\rgate$, we also have
  $C_1 \subseteq \lgate w,u\rgate$. Consequently, $\lgate w,u \rgate$
  is a cell of dimension $k+2$ containing $C_1 \cup C_2 \cup C_3$.
\end{proof}

\begin{lemma}\label{GMC-swm}
  Swm-complexes satisfy the graded monotonicity condition.
\end{lemma}

\begin{proof}
  Consider two intersecting cells $A,B$ that are
  faces of a cell $C$ such that $B \not\subseteq A$. As in case of
  cube complexes, pick a vertex $x \in B \setminus A$ that is adjacent
  to a vertex $y \in A\cap B$. By Lemma~\ref{lem-diametralpair-swm},
  there exists $y' \in A $ such that $\lgate y,y' \rgate = A$. Let
  $y''$ be the gate of $y'$ on $B$ (and on $A\cap B$) and note that by
  Lemma~\ref{lem-swm-gate}, $\lgate y,y'' \rgate = A\cap B$. Let
  $D' = \lgate x,y' \rgate$ and $D'' = \lgate x,y'' \rgate$. By
  Lemma~\ref{lem-swm-gate} applied to $D'$ and $D'\cap B$, we have
  $D'' = D'\cap B$.  By \cite[Lemma~5.11]{CCHO},
  $D' = \lgate x,y' \rgate$ and $D'' = \lgate x,y'' \rgate$ are dual
  polar graphs of dimensions $d(x,y')=d(y,y')+1=\dim(A)+1$ and
  $d(x,y'')=d(y,y'')+1=\dim(A\cap B)+1$, respectively. This
  establishes the graded monotonicity condition for swm-complexes.
\end{proof}

\subsection{Helly graphs and Helly groups}

We continue with the definitions of the main objects studied in this article: Helly and clique-Helly graphs, Helly and clique-Helly complexes, and Helly groups.

\begin{definition}\label{d:Helly}
  A graph $G$ is a \emph{Helly graph} if the ball-hypergraph
  ${\mathcal B}(G)$ is Helly.  A graph $G$ is a \emph{$1$--Helly
    graph} if the 1-ball-hypergraph ${\mathcal B}_1(G)$ is Helly.  A
  \emph{clique-Helly graph} is a graph $G$ in which the hypergraph
  ${\mathcal X}(G)$ of maximal cliques is Helly.
\end{definition}

Observe that a Helly graph is $1$--Helly and that a $1$--Helly graph
is clique-Helly but that the reverse implications do not hold: a cycle
of length at least $7$ is $1$--Helly but not Helly and a cycle of
length $4$ is clique-Helly but is not $1$--Helly. Notice also that
Helly graphs are pseudo-modular and thus weakly-modular.

For arbitrary graphs, the following compactness result for the Helly
property has been proved by Polat and Pouzet:

\begin{proposition} \label{Helly_Polat} \cite{Po_helly} A graph $G$ not containing infinite cliques
is Helly if and only if $G$ is finitely Helly.
\end{proposition}

\begin{definition}\label{d:Helly-complex}
  A \emph{Helly complex} is the clique
  complex of some Helly graph. A \emph{clique-Helly complex} is the clique
  complex of some clique-Helly graph.
\end{definition}

\begin{remark} If in Definitions~\ref{d:Helly}
  and~\ref{d:Helly-complex} instead of a Helly property we consider
  the corresponding finite Helly property, then the graphs satisfying
  it are called finitely Helly. For example, finitely clique-Helly
  graphs are graphs $G$ in which the hypergraph ${\mathcal X}(G)$ has
  the finite Helly property. For locally finite graphs, the finite Helly properties for balls and
  cliques implies the Helly property, thus finitely Helly
  (respectively, clique-Helly) graphs and complexes are Helly
  (respectively, clique-Helly).  By Proposition
  \ref{Helly_Polat}, the same implication holds for arbitrary graphs not containing
  infinite cliques.
\end{remark}

We continue with the definition of Helly groups:

\begin{definition}
	\label{d:Hellygroup}
	A group $\Gamma$ is \emph{Helly}
if it acts geometrically on a Helly complex
$X$.
\end{definition}

If a group $\Gamma$ acts geometrically on a Helly complex $X$, then $X$ is locally finite, moreover $X$ has uniformly bounded degrees.

In case of the clique-Helly property, the Berge-Duchet condition in Proposition~\ref{Helly_triangle} can be specified in the following way:

\begin{proposition}[\cite{Dragan,Szwarc}]\label{clique_Helly_triangle} A graph $G$ with finite cliques is clique-Helly if and only
if for any triangle $T$ of $G$ the set $T^*$ of all vertices of $G$ adjacent with at least two vertices of $T$
contains a vertex adjacent to all remaining vertices of $T^*$.
\end{proposition}

\begin{remark}
  Proposition~\ref{clique_Helly_triangle} does not hold for graphs
  containing infinite cliques. For example, consider the graph $G$
  defined as follows. First, consider an infinite clique
  $K = \{v_0,v_1,v_2, \ldots, v_k, \ldots\}$ whose vertex-set is
  indexed by $\N$. For each $i \in \N$, we add a vertex $u_i$ that is
  adjacent to all $v_j$ such that $j \geq i$. Observe that any two
  maximal cliques of $G$ have a non-empty intersection but there is no
  universal vertex in $G$. Consequently, $G$ is not clique-Helly. On
  the other hand, one can easily check that $G$ satisfies the
  criterion of Proposition~\ref{clique_Helly_triangle}.
\end{remark}

For any locally finite graph $G$, the clique-hypergraph $\cX(G)$ is
conformal and $G$ is isomorphic to the $2$-section of $\cX(G)$.
Moreover, if $G$ is clique-Helly, then $\cX(G)$ is Helly.
We conclude this subsection with the following simple but useful
converse result that is well-known (see e.g.\ \cite{BaPr}).
\begin{proposition}\label{hypergraph-clique-Helly} For a locally finite
hypergraph $\cH=(V,\cE)$ the following conditions are equivalent:
\begin{itemize}
\item[(i)] the 2-section $[\cH]_2$ of $\cH$ is a clique-Helly graph
  and $\cH$ is conformal (i.e., each maximal clique of $[\cH]_2$ is an
  edge of $\cH$);
\item[(ii)] the simplification $\breve{\cH}$ of $\cH$ is  conformal and Helly;
\item[(iii)] $\breve{\cH}$ satisfies Berge-Duchet and Gilmore conditions.
\end{itemize}
In particular, the 2-section of any locally finite triangle-free hypergraph is clique-Helly.
\end{proposition}

\begin{proof} Since $[\cH]_2=[\breve{\cH}]_2$, we can suppose that $\cH$ is simple. The equivalence (ii)$\Leftrightarrow$(iii) follows from Propositions~\ref{Helly_triangle} and~\ref{Gilmore}.
If (i) holds, then $\cH$ coincides with the hypergraph of maximal cliques of $[\cH]_2$, thus $\cH$ is Helly. Also $\cH$
is conformal as the clique-hypergraph of a graph. This establishes (i)$\Rightarrow$(ii). Conversely, if (ii) holds, since $\cH$ is conformal,
each clique of $[\cH]_2$ is included in an edge of $\cH$. Thus the maximal cliques of $[\cH]_2$ are in bijection with the edges of $\cH$. This shows that
$[\cH]_2$ is clique-Helly.
\end{proof}

From Propositions~\ref{3-cell} and~\ref{hypergraph-clique-Helly} we obtain the following result:

\begin{proposition}\label{cell-complex-cliqueHelly} If $X$ is an abstract cell complex for which the cell-hypergraph $\cH(X)$ satisfies the Helly property,  the 3-cell and the graded
  monotonicity conditions,
then the 2-section $[\cH]_2$ of $\cH$ is a clique-Helly graph and each maximal clique of $[\cH]_2$
is an edge of $\cH$.
\end{proposition}

\subsection{Hellyfication} There is a canonical way to extend any
hypergraph $\cH=(V,\cE)$ to a conformal hypergraph
$\conf(\cH)=(V,\cE')$: $\cE'$ consists of $\cE$ and all maximal by
inclusion cliques $C$ in the 2-section $[\cH]$ of $\cH$. Any conformal
hypergraph $\cH''$ extending $\cH$ and having the same 2-section
$[\cH''] = [\cH]$ as $\cH$ also contains $\conf(\cH)$ as a
sub-hypergraph, thus $\conf(\cH)$ can be called the \emph{conformal closure} of
$\cH$. Since the Helly property and conformality are dual to each
other, any hypergraph $\cH=(V,\cE)$ can be extended to a Helly hypergraph
$\helly(\cH)=(V',\cE')$: for every
maximal pairwise intersecting set $\cF$ of edges of $\cH$ with empty
intersection, add a new vertex $v_{\mathcal F}$ to $V$ and to each
member of $\mathcal F$. In the thus extended hypergraph $\helly(\cH)$
any two edges intersect exactly when their traces on $V$
intersect. Hence $\helly(\cH)$ satisfies the Helly property and we
call $\helly(\cH)$ the \emph{Hellyfication} of $\cH$. Again,
$\helly(\cH)$ is contained in any hypergraph satisfying the Helly
property, extending $\cH$ and having the same line graph as $H$.  This
kind of Hellyfication approach was used in~\cite{BaChEp} to Hellyfy discrete copair
hypergraphs  and to relate this Hellyfication procedure
with the cubulation (median hull) of the associated wall space; see
\cite[Proposition 3]{BaChEp}.

 \section{Injective spaces and injective hulls}\label{s:injective}

In this section we discuss injective metric spaces and Isbell's construction of injective hulls. Those notions are strongly related
to Helly graphs:  roughly, Helly graphs and ball-Hellyfication can be seen as discrete analogues of, respectively, (continuous) injective metric spaces and injective hulls.

\subsection{Injective spaces}
Recall that a metric space $(X,d)$ is called {\it hyperconvex} if every family of closed balls $B_{r_i}(x_i)$ of radii $r_i\in {\mathbb R}^{+}$ with centers $x_i$
satisfying $d(x_i, x_j)\leq r_i + r_j$, has a non-empty intersection.
Rephrasing the definition, $(X,d)$ is hyperconvex if it is
\emph{Menger-convex} (that is,
$B_r(x)\cap B_{d(x,y)-r}(y)\neq \emptyset$, for all $x,y\in X$ and
$r\in [0,d(x,y)]$) and the family of closed balls in $(X,d)$ satisfies
the Helly property. A metric space $(X,d)$ is called {\it
  integer-valued} if $d(x,y)$ is an integer for any $x,y\in X$.
An integer-valued metric
space $(X,d)$ is \emph{discretely geodesic} if for any two points
$x,y\in X$ with $d(x,y)=n$ there exists a sequence of points
$x_0:=x,x_1,x_2,\ldots,x_n:=y$ such that $d(x_i,x_{i+1})=1$.
The set of vertices of a connected graph equipped with a graph
distance is an example of an integer-valued and discretely geodesic
metric space.

Let $(Y,d')$ and $(X,d)$ be two metric spaces. For $A\subset Y$, a map $f:A\rightarrow X$ is \emph{$1$-Lipschitz} if $d(f(x),f(y))\le d'(x,y)$ for all $x,y\in A$.
The pair $(Y,X)$ has the \emph{extension property} if for any $A\subset Y$, any 1-Lipschitz map $f:A\rightarrow X$ admits a 1-Lipschitz extension, i.e., a
1-Lipschitz map $\widetilde{f}: Y\rightarrow X$ such that $\widetilde{f}|_A=f$. A metric space $(X,d)$ is \emph{injective} if for any metric space $(Y,d')$,
the pair $(Y,X)$ has the extension property.

For $Y\subset X$, the map $f:X\rightarrow Y$ is a \emph{(nonexpansive) retraction} if $f$ is 1-Lipschitz and $f(y)=y$ for any $y\in Y$. A
metric space $(Y,d')$ is an \emph{absolute retract} if whenever $(Y,d')$ is isometrically embedded in a metric space $(X,d)$, there exists a
retraction $f$ from $X$ to $Y$.

In 1956, Aronszajn and Panitchpakdi  established the following equivalence between hyperconvex spaces, injective spaces, and absolute retracts (in the category of
metric spaces with $1$-Lipschitz maps):

\begin{theorem}[\cite{AroPan1956}]\label{injective=hyperconvex}
  A metric space $(X,d)$ is injective if and only if $(X,d)$ is
  hyperconvex if and only if $(X,d)$ is an absolute ($1$-Lipschitz) retract.
\end{theorem}

\subsection{Injective hulls}\label{s:injective_isbell}
By a construction of Isbell~\cite{Isbell64} (rediscovered twenty years later by Dress~\cite{Dress1984} and yet another ten years later by Chrobak and Larmore~\cite{ChLa94} in computer science), for every metric space $(X,d)$ there exists a smallest (wrt.\ inclusions) injective metric space containing $X$. More precisely, an \emph{injective hull} (or \emph{tight span}, or  \emph{injective envelope}, or  \emph{hyperconvex hull}) of $(X,d)$ is a pair $(e,E(X))$ where $e\colon X\to E(X)$ is an isometric embedding into an injective metric space $E(X)$, and such that no injective proper subspace of $E(X)$ contains $e(X)$. Two injective hulls $e:X\rightarrow E(X)$ and $f:X\rightarrow E'(X)$ are {\it equivalent} if they are related by an isometry $i: E(X)\rightarrow E'(X)$.
Below we describe Isbell's construction in some details  and  we remind few important features of injective hulls --- all this will be of use in Section~\ref{Helly_groups}.

\begin{theorem}[\cite{Isbell64}]\label{injective-hull-isbell}
  Every metric space $(X,d)$ has an injective hull and all its
  injective hulls are equivalent.
\end{theorem}

We continue with the main steps in the proof of Theorem~\ref{injective-hull-isbell}. We follow the proof of Isbell's paper~\cite{Isbell64} but also use some notations and results from Dress~\cite{Dress1984}
and Lang~\cite{Lang2013}) (see these three papers for a full proof).
Let $(X,d)$ be a metric space. A \emph{metric form} on $X$ is a real-valued function $f$ on $X$ such that $f(x)+f(y)\geq d(x,y)$, for all
	$x,y\in X$. Denote by $\Delta(X)$ the set of all metric forms on $X$, i.e., $\Delta(X)=\{ f\in {\mathbb R}^X: f(x)+f(y)\ge d(x,y) \mbox{ for all } x,y\in X\}$.
For $f,g\in \Delta(X)$ set $f\le g$ if $f(x)\le g(x)$ for each $x\in X$. A metric form is called {\it extremal} on $X$ (or {\it minimal}) if there is no $g\in \Delta(X)$ such that $g\ne f$ and $g\le f$.
Let $E(X)=\{ f\in \Delta(X): f \text{ is extremal}\}$.

\begin{claim}\label{1-lip} If $f\in E(X)$, then $f(x)+d(x,y)\ge f(y)$ for any $x,y\in X$, i.e., $f$ is $1$-Lipschitz.
\end{claim}

Indeed, if this was false for some $x,y \in X$, then defining $g$ to coincide with $f$ everywhere except at $y$, where $g(y)=f(x)+d(x,y)$, we conclude that $g\in \Delta(X)$. 
Since $g\le f$, we must conclude $g=f$. 

The difference $d_{\infty}(f,g)=\sup_{x\in X} |f(x)-g(x)|$ between any
two extremal forms $f,g$ is bounded; any number $f(x)+g(x)$ is a
bound. Thus $(E(X),d_\infty)$ is a metric space. For a point $x\in X$,
let $d_x$ be defined by setting $d_x(y)=d(x,y)$ for any $y\in X$. 

\begin{claim}\label{Kuratowski}
  For any $x \in X$, the map $d_x: y \mapsto d(x,y)$ is extremal on
  $X$ and the map $e: X\rightarrow E(X)$ sending $x$ to $d_x$ is an
  isometric embedding of $(X,d)$ into $(E(X),d_{\infty})$.
\end{claim}

The map $e$ is often called the {\it Kuratowski embedding}. 

From the definition of extremal metric forms, the following useful property of $E(X)$ easily follows (this explains why extremal maps have been called {\it tight extentions} in~\cite{Dress1984}):

\begin{claim}\label{tight} If $(X,d)$ is compact then for any $f\in E(X)$ and $x\in X$, there exists $y$ in $X$ such that $f(x)+f(y)=d(x,y)$. In general metric spaces, for any $x\in X$ and any $\epsilon>0$ there exists $y$ in $X$ such that $f(x)+f(y)<d(x,y)+\epsilon$.
\end{claim}

The inequalities $f(x)+f(y)\ge d(x,y)$ and $f(x)+d(x,y)\ge f(y)$ together are equivalent to:

\begin{claim}\label{f(x)} If $f\in E(X)$, then $f(x)=d_{\infty}(f,e(x))$ for all $x\in X$.
\end{claim}

The following claim is the main technical tool in Isbell's proof. Let $\Delta(E(X))$ denote the set of all metric forms on $E(X)$ and let $E(E(X))$ denote the set of all extremal metric forms on $E(X)$.

\begin{claim}\label{composition} If $s$ is extremal  on $E(X)$, then $se$ is extremal on $X$.
\end{claim}

First notice that $se\in \Delta(X)$. To prove Claim~\ref{composition},
we suppose by way of contradiction that $se$ is not extremal and we
obtain a contradiction with the assumption that $s$ is extremal on
$E(X)$. Then there exists $h\in E(X)$ such that $h\le se$ and
$h(x)<se(x)$ for some $x\in X$. Define the map $t: E(X)\rightarrow \R$
by setting $t(f)=s(f)$ for all $f\in E(X)$ different from $e(x)$. Set
$t(e(x))=h(x) < s(e(x))$. Since $t<s$, to contradict the extremality
of $s$ on $E(X)$, it remains to show that $t\in \Delta(E(X))$, i.e.,
$t(f)+t(g)\ge d_{\infty}(f,g)$ for any $f,g\in E(X)$. Since
$s\in \Delta(E(X))$, from the definition of $t$ it suffices to establish
the previous inequality for any $f\in E(X)$ and $g=e(x)$ with
$f \neq e(x)$, i.e., to show that $te(x)+t(f)\ge
d_{\infty}(f,e(x))$. This is done using the definition of $e(x)$ and
the Claims~\ref{1-lip} and~\ref{f(x)}.  Indeed, for any $\epsilon>0$
pick $y\in X$ such that $f(x)+f(y)<d(x,y)+\epsilon$.  Then
$te(x) + t(f) = te(x) + se(y) - se(y)+s(f) \geq h(x) + h(y) -
d_{\infty}(e(y),f) \geq d(x,y) - f(y) > f(x) - \epsilon =
d_{\infty}(e(x),f) - \epsilon$.
Since $\epsilon>0$ is arbitrary, $te(x)+t(f)\ge d_{\infty}(f,e(x))$, as required. 

\begin{claim}\label{EX-injective} The metric space $(E(X),d_{\infty})$ is injective.
\end{claim}

To prove Claim~\ref{EX-injective}, in view of Theorem
\ref{injective=hyperconvex} it suffices to show that
$(E(X),d_{\infty})$ is hyperconvex, i.e., if
$f_i\in E(X), r_i\in {\mathbb R}^+, i\in I$ such that
$d_{\infty}(f_i,f_j)\le r_i+r_j$, then
$\bigcap_{i\in I} B(f_i,r_i)\ne \emptyset$. We may suppose that
$r: E(X)\rightarrow \Delta(E(X))$ is a metric form on $E(X)$ extending
the radius function $r_i$, i.e., $r(f_i)=r_i$ (this extension exists by Zorn lemma).
Let $s\in E(E(X))$ such
that $s\le r$. By Claim~\ref{composition}, $se$ belongs to $E(X)$. We
assert that $se$ belongs to any $r(f)$-ball centered at $f\in
E(X)$. Indeed, for any $x\in X$, we have
$se(x)-f(x)=se(x)-d_{\infty}(f,e(x))\le s(f)\le r(f)$, where the
equality follows from Claim~\ref{f(x)} and the first inequality
follows from Claim~\ref{1-lip} (applied to $E(X)$ and $E(E(X))$
instead of $X$ and $E(X)$). On the other hand,
$f(x)-se(x)=d_{\infty}(f,e(x))-se(x)\le s(f)\le r(f)$, where the
equality follows from Claim~\ref{f(x)} and the inequality follows by
the choice of $s$ in $\Delta(E(X))$. This establishes Claim
\ref{EX-injective}.

\begin{claim}\label{equivalence} $e:X\rightarrow E(X)$ is an injective hull and is equivalent to every injective hull of $X$.
\end{claim}

Let $\alpha: E(X)\rightarrow E(X)$ be a 1-Lipschitz map such that
$\alpha(e(x))=e(x)$ for any $x\in X$. Let $f\in E(X)$ and let
$g=\alpha(f)$. By Claim~\ref{f(x)}, for any $x\in X$ we have
$g(x)=d_{\infty}(g,e(x)) = d_\infty(\alpha(f),\alpha(e(x))) \le
d_{\infty}(f,e(x))=f(x)$. Hence $g\le f$, whence $\alpha$ is the
identity map.  Therefore $E(X)$ cannot be retracted to any proper
subset $S\subset E(X)$ containing the image of $X$ under $e$, hence
$S$ is not injective.

Finally, consider an arbitrary injective hull $e':X\rightarrow E'(X)$
of $(X,d)$. Let $f$ be an isometry from $e(X)$ to $e'(X)$ and let $f'$
be its inverse. Since both $E(X)$ and $E'(X)$ are injective spaces,
there exist $1$-Lipschitz maps $\tf: E(X) \to E'(X)$ and
$\tf': E'(X) \to E(X)$ extending respectively $f$ and $f'$. Observe
that the composition $\tf'\tf$ is a $1$-Lipschitz map from $E(X)$ to
$E(X)$ that is the identity on $e(X)$. Therefore, $\tf'\tf$ is the
identity map by what has been shown above and thus $\tf$ is injective
and $\tf'$ is surjective. Since $\tf$ and $\tf'$ are $1$-Lipschitz and
since $\tf'\tf$ is the identity on $E(X)$, necessarily $\tf$ is an
isometric embedding of $E(X)$ in $E'(X)$. Then since $E(X)$ is
injective, the image of $\tf$ contains $e'(X) = \tf(e(X))$ and $E'(X)$
is an injective hull, necessarily $\tf$ must be surjective and thus
$\tf$ is an isometry. Necessarily $\tf'$ is injective as otherwise,
$\tf'\tf$ cannot be the identity map on $E(X)$. Thus both $\tf$ and
$\tf'$ are isometries.  This concludes the proof of Theorem
\ref{injective-hull-isbell}.

Dress~\cite{Dress1984}  defined $E(X)$ as the set of all maps $f\in {\mathbb R}^X$ such that $f(x)=\sup \{ d(x,y)-f(y): y\in X\}$ for all $x\in X$. He established the following nice property  of $E(X)$ (which in fact
characterizes $E(X)$, see~\cite[Theorem 1]{Dress1984}):

\begin{claim}\label{dress} If $f,g\in E(X)$, then $d_{\infty}(f,g)=\sup \{ d_{\infty}(e(x),e(y))-d_{\infty}(e(y),f)-d_{\infty}(e(x),g): x,y\in X\}$.
\end{claim}

For simplicity, we will prove the Claim~\ref{dress} for compact metric spaces, for which  the supremum can be replaced by  maximum. The claim asserts that any pair of extremal functions $f,g$ lies on a geodesic
between the images $e(x),e(y)$ in $E(X)$ of two points $x,y$ of $X$. Let $x$ be a point of $X$ such that $d_{\infty}(f,g)=f(x)-g(x)$. By Claim~\ref{tight} there exists $y\in X$ such that $f(x)=d(x,y)-f(y)$. Hence
$d_{\infty}(f,g)=f(x)-g(x)=d(x,y)-f(y)-g(x)=d_{\infty}(e(x),e(y))-f(y)-g(x)$. By Claim~\ref{f(x)}, $f(y)=d_{\infty}(f,e(y))$ and $g(x)=d_{\infty}(g,e(x))$. Consequently, $d_{\infty}(f,g)=d_{\infty}(e(x),e(y))-f(y)-g(x)=d_{\infty}(e(x),e(y))-d_{\infty}(f,e(y))-d_{\infty}(g,e(x))$ and we are done.

One interesting property of injective hulls is their monotonicity:

\begin{corollary}\label{monotone} If $(X,d)$ is isometrically embeddable into $(X',d')$, then $E(X)$ is isometrically embeddable into $E(X')$.
\end{corollary}

\begin{proof} $(X,d)$ is isometrically embeddable into $(X',d')$ and into $E(X)$ and $(X',d')$ is isometrically embeddable into $E(X')$. Therefore there exists an isometric embedding of $e(X)\subset E(X)$
into $E(X')$. Since $E(X')$ is injective, this isometric embedding extends to a 1-Lipschitz map $\alpha$ from $E(X)$ to $E(X')$. If $d_{\infty}(\alpha(f),\alpha(g))<d_{\infty}(f,g)$ for $f,g\in E(X)$,
we will deduce that $d_{\infty}(\alpha(e(x)),\alpha(e(y)))<d_{\infty}(e(x),e(y))$ for points $x,y\in X$ occurring in Claim~\ref{dress}, contrary to the assumption that $\alpha$ isometrically embeds
$e(X)$.
\end{proof}

Dress~\cite{Dress1984} described the combinatorial types of injective
hulls of metric spaces on 3, 4, and 5 points. Sturmfels and
Yu~\cite{StuYu} described all 339 combinatorial types of injective
hulls of 6-point metric spaces.

As shown by Dress~\cite{Dress1984}, the injective hull of a finite
metric space is a finite polyhedral complex. Using this, he defined
the \emph{combinatorial dimension} of a general metric space $X$ as
the supremum of the dimensions of the polyhedral complexes $E(Y)$ for
all finite subspaces $Y$ of $X$. Any $f \in E(X)$ belongs to the
interior of a unique cell of the polyhedral complex. Dress
characterized combinatorially the cells of $E(X)$. Goodman and Moulton
gave a presentation of Dress's result in the finite
case~\cite{GoodmanMoulton}. Lang presented Dress's result in the case
of general metric spaces and formulated conditions under which the
injective hull is finite dimensional or has a finite number of types
of cells for each dimension~\cite{Lang2013}. In the following, we
continue with this combinatorial description following the
presentation of~\cite{Lang2013,DesLang2015}.

For any $f \in \Delta(X)$, we consider the graph $(X,A(f))$ where
$A(f)$ is the set of all unordered pairs $\{x,y\}$ of points in $X$
with the property that $f(x) + f(y) = d(x,y)$. If $X$ is finite (or
compact), then $f$ belongs to $E(X)$ if and only if $(X,A(f))$ has no
isolated vertices. This is no longer true when $X$ is not compact (see
Claim~\ref{tight}). For this, Dress and Lang introduced the subset
$E'(X) = \{f \in \Delta(X): \bigcup A(f) = X \}$ of $E(X)$. They show that $E'(x)$ is dense in $E(X)$ if the metric on $X$ is integer-valued.

A set $A$ of unordered pairs of points in $X$ is called
\emph{admissible} if there exists $f \in E'(X)$ with $A(f) = A$, and
denote by $\cA(X)$ the set of all such admissible sets. The family of
polyhedral faces of $E(X)$ is then given by $\{P(A)\}_{a \in \cA(X)}$
where $P(A) = \{f \in \Delta(X) : A \subseteq A(f)\}$. As noticed
in~\cite{Lang2013}, $P(A) = P(A) \cap E(X) = P(A) \cap E'(X)$. The
\emph{rank} $\rk(A)$ of an admissible set $A$ is the dimension of
$P(A)$. The rank $\rk(A)$ can be characterized as follows. If
$f, g \in P(A)$, then $f(x)+f(y) = d(x,y) = g(x) + g(y)$ for
$\{x,y\} \in A$ and thus $f(y) - g(y) = - (f(x) - g(x))$. Thus the
difference $f-g$ has alternating sign along all paths in the graph
$(X,A)$. Consequently, for each connected component of $(X,A)$, there
is at most one degree of freedom for the values of $f \in P(A)$. If
the connected component $C$ contains an odd cycle, then $f$ and $g$
coincide on all vertices of $C$. Alternatively, if the connected
component $C$ is bipartite, then the restrictions of all functions
$f \in P(A)$ on the vertices of $C$ form a 1-parameter family: given
the value $f(x)$ on one vertex of $C$, one can deduce all the other
values of $f$ on $C$. Then, the rank $\rk(A) = \dim(P(A))$ is
precisely the number of bipartite components of the graph $(X,A)$

Dress~\cite{Dress1984} charaterized spaces of combinatorial dimension
at most $n$ by a $2(n+1)$-point inequality. These notions are
important to state and establish some results of Lang~\cite{Lang2013}
that we will present and use in Section~\ref{s:hypquad}.

\subsection{Coarse Helly property}
\label{s:coarseHelly}

A metric space $(X,d)$ is {\it coarsely hyperconvex} if there exists
some $\delta \geq 0$ such that for any set of centers
$\{x_i\}_{i \in I}$ in $X$ and any set of radii $\{r_i\}_{i \in I}$ in
$\R^+$ satisfying $d(x_i,x_j) \leq r_i + r_j$, there exists $x \in X$
such that $d(x,x_i) \leq r_i + \delta$ for all $i \in I$, i.e., the
intersection $\bigcap_{i\in I}B_{r_i+\delta}(x_i)$ is not empty.  A
metric space $(X,d)$ has the {\it coarse Helly property} if there
exists some $\delta\ge 0$ such that for any family
$\cB = \{ B_{r_i}(x_i): i\in I\}$ of pairwise intersecting closed
balls of $X$, the intersection $\bigcap_{i\in I}B_{r_i+\delta}(x_i)$
is not empty. If the space $(X,d)$ is Menger-convex (in particular, if
$(X,d)$ is geodesic), both properties are equivalent.  In a discretely
geodesic metric space (in particular, in a graph), if
$d(x_i,x_j) \leq r_i+r_j$, then the balls $B_{\lceil r_i\rceil}(x_i)$
and $B_{\lceil r_j\rceil}(x_j)$ intersect. In particular if the
$\{r_i\}_{i \in I}$ are integers, then $d(x_i,x_j) \leq r_i+r_j$ if
and only if $B_{r_i}(x_i)$ and $B_{r_j}(x_j)$ intersect.
Consequently, a discretely geodesic metric space $(X,d)$ is coarsely
hyperconvex with some constant $\delta$ if and only if it satisfies
the coarse Helly property with some constant $\delta'$, where $\delta$
and $\delta'$ differ by at most $1$.
The injective hull $E(X)$ of a metric space $(X,d)$ has the {\it
  bounded distance property} if there exists $\delta\ge 0$ such that
for any $f\in E(X)$ there exists a point $x\in X$ such that
$d_{\infty}(f,e(x))\le \delta$.  The coarse Helly property has been
introduced in~\cite{ChEs} and the bounded distance property has been
introduced in~\cite{Lang2013}, in both cases, for $\delta$-hyperbolic
spaces and graphs.

We show that the coarse hyperconvexity of a metric space is equivalent
to the fact that its injective hull satisfies the bounded distance
property.\footnote{Independently, this was also observed by Urs Lang
  (personal communication).}.

\begin{proposition}\label{coarse}
  A metric space $(X,d)$ is coarsely hyperconvex if and only if its
  injective hull $E(X)$ satisfies the bounded distance property.
Consequently, if $(X,d)$ is a geodesic or discretely geodesic metric
  space, then the coarse hyperconvexity of $(X,d)$, the coarse Helly
  property for $(X,d)$ and the bounded distance property for $E(X)$
  are all equivalent.
\end{proposition}

\begin{proof} First suppose that $(X,d)$ is coarsely hyperconvex with
  some constant $\delta\ge 0$. Let $f\in E(X)$. Then
  $f(x)+f(y)\ge d(x,y)$ for any $x,y$. By the coarse hyperconvexity of
  $(X,d)$ applied to the radius function $f$, there exists a point
  $z\in X$ such that $d(z,x)\le f(x)+\delta$ for any $x\in X$. We
  assert that $d_{\infty}(f,e(z))\le \delta$.  Indeed,
  $d_{\infty}(f,e(z))=\sup_{x\in X} |f(x)-d(x,z)|$. By the choice of
  $z$ in $B_{f(x)+\delta}(x)$, $d(x,z)-f(x)\le \delta$. It remains to
  show the other inequality $f(x)-d(x,z)\le \delta$.  Assume by
  contradiction that $f(x) - d(x,z) > \delta$. Let
  $\epsilon = \frac{1}{2}(f(x) - d(x,z) - \delta)$ and observe that
  $f(x)>d(x,z)+\delta + \epsilon$. By Claim~\ref{tight}, there exists
  $y \in X$ such that $f(x) + f(y) < d(x,y) + \epsilon$. But since
  $z\in B_{f(y)+\delta}(y)$, we have $f(y)\ge d(y,z)-\delta$, and
  consequently, we have
  $f(x)+f(y)>d(x,z)+\delta + \epsilon +d(y,z)-\delta=d(x,z)+d(y,z) +
  \epsilon \ge d(x,y) + \epsilon$ (the last inequality follows from
  the triangle inequality), a contradiction.

  Conversely, suppose that $E(X)$ satisfies the bounded distance
  property with $\delta\ge 0$ and we will show that $(X,d)$ is
  coarsely hyperconvex. Let $B(x_i,r_i),i\in I$ be a collection of
  closed balls of $(X,d)$ such that $r_i+r_j\ge d(x_i,x_j)$ for all
  $i,j\in I$. Let $r\in \Delta(X)$ be a metric form on $X$ extending
  the radius function $r_i, i\in I$ (its existence follows from Zorn's
  lemma). Let $f\in E(X)$ such that $f(x)\le r(x)$ for any $x\in
  X$. By the bounded distance property, $X$ contains a point $z$ such
  that $d_{\infty}(f,e(z))\le \delta$. This implies that
  $|f(x)-e(z)(x)|=|f(x)-d(x,z)|\le \delta$ for any $x\in X$.  In
  particular, this yields $d(x,z)\le f(x)+\delta\le r(x)+\delta$, thus
  $z$ belongs to all closed balls $B_{r(x)+\delta}(x), x\in X$.
\end{proof}

\subsection{Geodesic bicombings}
\label{s:bicombing}
One important feature of injective metric spaces is
the existence of a nice (bi)combing.
Recall that a \emph{geodesic bicombing} on a metric space $(X,d)$ is a map
\begin{align}
\sigma \colon X\times X \times [ 0,1] \to X,
\end{align}
such that for every pair $(x,y)\in X\times X$ the function $\sigma_{xy}:=\sigma(x,y,\cdot)$ is
a constant speed geodesic from $x$ to $y$. We call $\sigma$ \emph{convex} if the function
$t\mapsto d(\sigma_{xy}(t),\sigma_{x'y'}(t))$ is convex for all $x,y,x',y'\in X$. The bicombing
$\sigma$ is \emph{consistent} if $\sigma_{pq}(\lambda)=\sigma_{xy}((1-\lambda)s+\lambda t)$, for all
$x, y \in X$, $0\leq s \leq t \leq 1$, $p := \sigma_{xy}(s)$, $q := \sigma_{xy}(t)$, and $\lambda \in [ 0, 1]$.
It is called \emph{reversible} if $\sigma_{xy}(t)=\sigma_{yx}(1-t)$ for all $x,y\in X$ and $t\in [0,1]$.

From the definition of injective hulls and~\cite[Lemma~2.1 and Theorems
1.1\&1.2]{DesLang2015} we have the following:

\begin{theorem}\label{t:injbicomb}
  A proper injective metric space of finite combinatorial dimension
  admits a unique convex, consistent, reversible geodesic bicombing.
\end{theorem}

 \section{Helly graphs and complexes}
\label{s:prop}

In this section, we recall the basic
properties and characterizations of Helly graphs. We also show that any graph admits a Hellyfication, a discrete counterpart of Isbell's
construction (again this is well-known).  

\subsection{Characterizations}

Helly graphs are the discrete analogues of
hyperconvex spaces: namely, the requirement that radii of balls are
nonnegative reals is modified by replacing the reals by the
integers.  In perfect analogy with hyperconvexity, there is a close
relationship between Helly graphs and absolute retracts. A graph is an
\emph{absolute retract} exactly when it is a retract of any larger
graph into which it embeds isometrically.  A vertex $x$ of a graph $G$ is \emph{dominated} by another
vertex $y$ if the unit ball $B_1(y)$ includes $B_1(x).$ A graph $G$ is
\emph{dismantlable} if its vertices can be well-ordered $\prec$ so
that, for each $v$ there is a neighbor $w$ of $v$ with $w\prec v$
which dominates $v$ in the subgraph of $G$ induced by the vertices
$u\preceq v$.  The following theorem summarizes some of the
characterizations of finite Helly graphs:

\begin{theorem} \label{Helly1} For a finite  graph $G$, the following
statements  are equivalent:

\begin{itemize}

\item[(i)] $G$ is a Helly graph;

\item[(ii)]  \cite{HeRi} $G$ is a retract of a direct product of
paths;

\item[(iii)] \cite{BaPr} $G$ is a dismantlable clique-Helly graph;

\item[(iv)] \cite{BP-absolute} $G$ is a weakly modular $1$--Helly graph.

\end{itemize}
\end{theorem}

The following result presents a local-to-global and a topological
characterization of all (not necessarily finite or locally finite)
Helly graphs, refining and generalizing Theorem~\ref{Helly1}
(iii),(iv).

\begin{theorem}[\cite{CCHO}]\label{t:lotogloHell}
For a graph $G$, the following conditions are equivalent:
\begin{itemize}
\item[(i)] $G$ is Helly;
\item[(ii)] $G$ is a weakly modular $1$--Helly graph;
\item[(iii)] $G$ is a dismantlable clique-Helly graph;
\item[(iv)] $G$ is clique-Helly with a simply connected clique complex.
\end{itemize}
Moreover, if the clique complex $X(G)$ of $G$ is finite-dimensional, then the conditions (i)-(iv) are equivalent to
\begin{itemize}
\item[(v)] $G$ is clique-Helly with a contractible clique complex.
\end{itemize}
\end{theorem}

The following result shows the connection between Helly
complexes and clique-Helly complexes:

\begin{theorem}[\cite{CCHO}]\label{t:lotogloHell_bis}
Let $G$ be a (finitely) clique-Helly graph and let $\tG$ be the
$1$--skeleton of the universal cover $\widetilde
X:=\widetilde{X}(G)$ of the clique complex $X:=X(G)$ of
$G$. Then $\tG$ is a (finitely) Helly graph.  In particular, $G$ is a
(finitely) Helly graph if and only if $G$ is (finitely) clique-Helly
and its clique complex is simply connected.
\end{theorem}

As noticed in~\cite{CCHO}, Theorem~\ref{t:lotogloHell_bis} and its proof
lead to two  conclusions. The first one is: if a simplicial complex  $X$ is clique-Helly
(for arbitrary families of maximal cliques),
then its universal cover $\widetilde{X}$ is Helly (for arbitrary families of balls of its 1-skeleton).
The second one is: if $X$ is finitely clique-Helly, then its universal cover
is finitely Helly (this holds even if $X$ contains infinite cliques). From~\cite[Theorem 9.1]{CCHO} it follows that Helly graphs satisfy a
quadratic isoperimetric inequality. It was shown in~\cite{Qui} that any finite Helly graph $G$
has the stabilized clique property, i.e., there exists
a complete subgraph of $G$ invariant under the action of the automorphism group of $G$.
Other properties of Helly graphs will be presented below.

\subsection{Injective hulls  and Hellyfication}\label{injhull+hellyfication}

We will show that for any graph
$G$ there exists a smallest Helly graph $\He(G)$ comprising $G$ as an isometric
subgraph; we call $\He(G)$ the \emph{Hellyfication} of $G$ (analogously, we will
denote by $\He(X(G))$ the clique complex of $\He(G)$ and refer to it as to the
\emph{Hellyfication} of $X(G)$).

Let $(X,d)$ be an integer-valued metric space. An \emph{integer metric form}
on $X$ is a function $f\colon X\to \mathbb{Z}$ such that
$f(v)+f(w)\geq d(v,w)$, for all $v,w\in X$. Let $\Delta^0(X)$
denote the set of all integer metric forms on $X$. An integer metric
form is \emph{extremal} if it is minimal pointwise.  We define the
metric space $E^0(X)\subset \Delta^0(X)$ as the set of all extremal
integer metric forms on $(X,d)$ endowed with the sup-metric
$d_{\infty}$.  The embedding $e\colon X\to E^0(X)$ is defined as
$v\mapsto d(v,\cdot)$. The pair $(e,E^0(X))$ is the \emph{discrete
  injective hull} of $X$. We define a graph structure on $E^0(X)$ by
putting an edge between two extremal forms $f,g \in E^0(X)$ if
$d_\infty(f,g) = 1$. With some abuse of notation, we also denote this
graph by $E^0(X)$. If $G=(V,E)$ is a graph with the path metric $d$,
we will denote by $E^0(G)$ and $E(G)$ the discrete injective hull
$E^0(V(G))$ and the injective hull of the metric space $(V(G),d)$,
respectively. Similarly, we write $e(G)$ instead of $e(V(G))$.

The following result is well known, see~\cite{JaMiPou,Pesch87,Pesch88}, and
is the discrete counterpart of Isbell's Theorem~\ref{injective-hull-isbell}.

\begin{theorem}\label{t:dinjhull}
  If $(X,d)$ is an integer-valued metric space, then
  $E^0(X) = E(X) \cap \Z^X$ is the smallest Helly graph into which
  $(X,d)$ is isometrically embedded.  In particular, the discrete
  injective hull $E^0(G)$ of a graph $G$ is contained as an isometric
  subgraph in any Helly graph $G'$ containing $G$ as an isometric
  subgraph and is the Hellyfication $\He(G)$ of $G$.
\end{theorem}

\begin{proof}
  First we show that the sets $E^0(X)$ and $E(X) \cap \Z^X$
  coincide. Observe that by the definitions of $E^0(X)$ and
  $E(X) \cap \Z^X$, we have $E(X)\cap {\mathbb Z}^X \subseteq E^0(X)$.
  To show the converse inclusion, first note that $E^0(X)$ satisfies
  the discrete analog of Claim~\ref{tight}: if $f\in E^0(X)$, then
  for any $x$ in $X$, there exists $y$ in $X$ such that
  $f(x)+f(y)=d(x,y)$. By way of contradiction, suppose there exist
  $f\in E^0(X)$ and $g\in E(X)$ such that $g\ne f$ and $g\leq f$. Then
  $g(x)<f(x)$ for some point $x$ of $X$. By the discrete analog of
  Claim~\ref{tight}, there exists $y$ in $X$ such that
  $f(x)+f(y)=d(x,y)$. But since $g(x)<f(x)$ and $g(y)\le f(y)$, we obtain
  $g(x)+g(y)<d(x,y)$, contrary to the assumption that $g\in
  E(X)$. Therefore, $E^0(X)\subseteq E(X)\cap {\mathbb Z}^X$ and thus
  $E^0(X)=E(X)\cap {\mathbb Z}^X$. Consequently, $(E^0(X),d_{\infty})$
  is also an integer-valued metric space.

  Next we show that the balls of $(E^0(X),d_{\infty})$ satisfy the
  Helly property. Let $f_i\in E^0(X), r_i\in {\mathbb Z}^+, i\in I$
  such that $d_{\infty}(f_i,f_j)\le r_i+r_j$. We may suppose that
  $r\in \Delta^0(E^0(X))$ is an integer metric form on $E^0(X)$
  extending the radius function $r_i$ (i.e., $r(f_i)=r_i, i\in I$) and
  $t \in E^0(E^0(X))=E(E^0(X))\cap {\mathbb Z}^{E^0(X)}$ is an
  integer metric form on $E^0(X)$ such that $t\le r$. Let
  $t'\in \Delta(E(X))$ be a metric form on $E(X)$ extending $t$, i.e.,
  for any $f \in E^0(X), t'(f)=t(f)$ (its existence follows by Zorn
  lemma).  Let $s\in E(E(X))$ such that
  $s\le t'$.  By the discrete analog of Claim~\ref{tight}, for any
  $f \in E^0(X)$, there exists $g \in E^0(X)$ such that
  $t(f) + t(g) = d_\infty(f,g)$. Since
  $s(f) + s(g) \leq t'(f) + t'(g) = t(f) + t(g) = d_\infty(f,g) \leq
  s(f) + s(g)$, we have that $s(f) = t'(f) = t(f)$ and
  $s(g) = t'(g) = t(g)$ since $s(h) \leq t'(h) = t(h)$ for any
  $h \in E^0(X)$. Consequently, $s_{|E^0(X)} = t$. By
  Claim~\ref{composition} and the proof of Claim~\ref{EX-injective},
  $se$ belongs to $E(X)$ and is a common point of all balls
  $B_{r_i}(f_i)$. Since $e(x) \in E^0(X)$ for any $x \in X$, and since
  $s$ and $t$ coincide on $E^0(X)$, $se = te$. Therefore, $te$ belongs
  to $E^0(X)$ and is a common point of all balls $B_{r_i}(f_i)$. This
  shows that the balls of $(E^0(X),d_\infty)$ satisfy the Helly
  property.

  We show by induction on the distance $d_\infty(f,g)$ that any two
  vertices $f,g \in E^0(X)$ are connected in the graph $E^0(X)$ by a
  path of length $d_\infty(f,g)$. Indeed, if $d_\infty(f,g) = k$,
  consider a ball of radius $1$ centered at $f$ and a ball of radius
  $k-1$ centered at $g$. By the Helly property, there exists
  $h \in E^0(X)$ such that $d_\infty(f,h) \leq 1$ and
  $d_\infty(h,g) \leq k-1$. By the triangle inequality, these two
  inequalities are equalities. Thus $E^0(X)$ is a Helly graph
  isometrically embedded in $E(X)$. The proof that $E^0(X)$ does not
  contain any Helly subgraph containing $X$ and that all discrete
  injective hulls are isometric is identical to the proof of
  Claim~\ref{equivalence}.  The proof that $E^0(X)$ is an isometric
  subgraph of any Helly graph $G'$ containing $G$ as an isometric
  subgraph is similar to the proof of Corollary~\ref{monotone}.
\end{proof}

\begin{remark}
  A direct consequence of the second assertion of Theorem~\ref{t:dinjhull} is
  that if $G$ is Helly, then $\He(G)$ coincides with $G$.
\end{remark}

\begin{remark}\label{rem-EE0}
  For an integer-valued metric space $(X,d)$, the injective hull $E(E^0(X))$
  of the discrete injective hull $E^0(X)$ of $X$ coincides with the
  injective hull $E(X)$ of $X$.
\end{remark}

\subsection{Hyperbolicity and Helly graphs}

In Helly graphs, hyperbolicity can be characterized by forbidding
isometric square-grids.

\begin{proposition}\label{prop-Helly-hyp}
  For a Helly graph $G$, the following are equivalent:
  \begin{enumerate}[(1)]
  \item $G$ has bounded hyperbolicity,
  \item the size of isometric  $\ell_1$--square-grids of $G$
    is bounded,
  \item the size of isometric $\ell_\infty$--square-grids of $G$
    is bounded.
  \end{enumerate}
\end{proposition}

\begin{proof}
  Since any Helly graph $G$ is weakly modular, by
  \cite[Theorem~9.6]{CCHO}, $G$ has bounded hyperbolicity if and only
  if the metric triangles and the isometric square-grids are of
  bounded size. Since Helly graphs are pseudo-modular, all metric triangles 
  of $G$ are of size at most one. Therefore $G$ has bounded hyperbolicity if and only the size
  of the isometric $\ell_1$--square-grids of $G$ are bounded. We now
  show that in a Helly graph $G$, the size of the isometric
  $\ell_1$--square-grids is bounded if and only if the size of the isometric
  $\ell_\infty$--square-grids is bounded.

  Suppose first that $G$ contains an isometric $2k \times 2k$
  $\ell_1$--grid $H_1$. Observe that we can represent $H_1$ as follows:
  $V(H_1) = \left\{(i,j) \in \Z^2 : |i|+|j| \leq 2k \text{ and } i+j
    \text{ is even}\right\}$ and $(i,j)(i',j') \in E(H_1)$ if and only
  if $|i-i'| = |j-j'| = 1$, i.e., if and only if
  $d_\infty((i,j),(i',j')) = 1$. Since $G$ is Helly, the Hellyfication
  $H_1'$ of $H_1$ is an isometric subgraph of $G$ and $H_1'$ can then
  be described as follows:
  $V(H_1') = \left\{(i,j) \in \Z^2 : |i|+|j| \leq 2k\right\}$ and
  $(i,j)(i',j') \in E(H_1')$ if and only if
  $d_\infty((i,j),(i',j')) = 1$. But then, observe that the set of
  vertices
  $\left\{(i,j) \in V(H_1'): |i| \leq k \text{ and } |j| \leq
    k\right\}$ induce a $2k \times 2k$ $\ell_\infty$--grid in $H_1'$
  and thus in $G$.

  Suppose now that $G$ contains an isometric $2k \times 2k$
  $\ell_\infty$--grid $H_2$. We can represent $H_2$ as
  follows:
  $V(H_2) = \left\{(i,j) \in \Z^2: |i| \leq k \text{ and } |j| \leq
    k\right\}$ and $(i,j)(i',j') \in E(H_1')$ if and only if
  $d_\infty((i,j),(i',j')) = 1$. Let $H_2'$ be the graph induced
  by
  $V(H'_2) = \left\{(i,j) \in \Z^2: |i| + |j| < k \text{ and }
    i+j \text{ is even}\right\}$. Observe that $H'_2$ is
  isomorphic to a $k \times k$ $\ell_1$--grid. Since $H'_2$ is an
  isometric subgraph of $H_2$, $G$ contains an isometric
  $k \times k$ $\ell_1$--grid.
\end{proof}

Dragan and Guarnera~\cite{DrGu} characterize precisely the
hyperbolicity of a Helly graph by presenting three families of
isometric subgraphs of the $\ell_\infty$--grid that are the only
obstructions to a small hyperbolicity.
 \section{Helly graphs constructions}\label{s:operations}

In the previous section,  with any connected graph $G$ we associated in a canonical way a Helly graph $\He(G)$. However, not every group acting geometrically on $G$ acts also geometrically on $\He(G)$.
In this section, we prove or recall that several standard graph-theoretical operations preserve Hellyness and that other operations applied to some non-Helly graphs lead to Helly graphs.
As we will show in the next section, those constructions also preserve the geometric action of the group, allowing to prove that some classes of groups are Helly.

\subsection{Direct products and amalgams}
\label{s:dirandamal}

We start with the following well-known result:

\begin{proposition}\label{direct-products} The classes of Helly and clique-Helly graphs are closed by taking direct products of finitely many factors and retracts.
\end{proposition}

The first assertion follows from the fact that the balls in a direct product are direct products of balls in the factors and that the maximal cliques of
a direct product are direct products of maximal cliques. The second assertion follows from the fact that retractions are 1-Lipschitz maps and therefore preserve the Helly property.

The amalgam of two Helly graphs along a Helly graph is not necessarily
Helly: the $3$-sun (which is not Helly) can be obtained as an amalgam over an edge of a
triangle and a $3$-fan (which are both Helly); see Figure~\ref{fig-3sun-amalgam}.
\begin{figure}[h]
\centering
\includegraphics{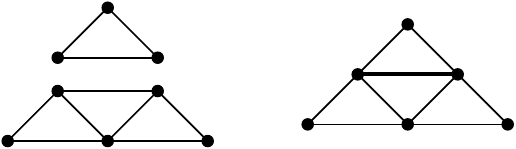}\caption{The $3$-sun can be obtained from the amalgam of a triangle
  and a 3-fan over an edge.}\label{fig-3sun-amalgam}
\end{figure}

Now we
consider amalgams of direct products of (clique-)Helly graphs and, more
generally, of graphs obtained by amalgamating together a collection of
direct products of (clique-)Helly graphs along common subproducts.  We
provide sufficient conditions for these amalgams to be (clique-)Helly.

Given a family $\cH = \{H_j\}_{j\in J}$ of locally finite graphs, a
\emph{finite subproduct} of the direct product
$\boxtimes \cH = \boxtimes_{j \in J} H_j$ is a subgraph
$G = \boxtimes_{j\in J} G^j$ of $\boxtimes \cH$ such that $G^j = H_j$
for finitely many indices and $G^j = \{v_j\}$ where $v_j \in V(H_j)$
for all other indices. For each vertex $v$ of $\boxtimes H$ (or any of
its subgraphs), we denote by $v_j$ the coordinate of $v$ in $H_j$.

A locally finite connected graph $G$ is a \emph{union of graph
  products (UGP)} over a family $\cH = \{H_j\}_{j\in J}$ of locally
finite graphs if there exists a family $\{G_i\}_{i \in I}$ of distinct
finite subproducts of $\boxtimes \cH$ such that $G = \bigcup_i G_i$.
The graphs $G_i$ are called the \emph{pieces} of $G$. Since each
$H_j \in \cH$ is locally finite and each piece of $G$ is a finite
subproduct of $\boxtimes \cH$, each piece of $G$ is also locally
finite.  Observe that $G$ is a subgraph of $\boxtimes \cH$ but not
necessarily an induced subgraph. However, each piece $G_i$ of $G$ is
an induced subgraph of $\boxtimes \cH$.

We say that the pieces of a collection $\{G_{i_k}\}_{k \in K}$ of pieces  of an UGP
$G = \bigcup_{i \in I} G_i \subseteq \boxtimes \cH$ over
$\cH = \{H_j\}_{j \in J}$ \emph{agree} on a factor $H_j$ if there
exists $v_j \in V(H_j)$ such that for each $k \in K$, either
$G_{i_k}^j = H_j$ or $G_{i_k}^j = \{v_j\}$.

\begin{lemma}\label{lem-UGP-agreement}
  Two pieces $G_1$ and $G_2$ of an UGP $G\subseteq \boxtimes \cH$
  have a non-empty intersection if and only if $G_1$ and $G_2$ agree on all
  factors $H_j \in \cH$.

  The set of pieces $\{G_i\}_{i \in I}$ satisfies the Helly property:
  any collection $\{G_{i_k}\}_{k\in K}$ of pairwise intersecting
  pieces has a non-empty intersection, i.e., there exists a vertex $w$
  of $G$ such that for each $k \in K$ and each factor $H_j \in \cH$,
  either $G_{i_k}^j = \{w_j\}$ or $G_{i_k}^j = H_j$.
\end{lemma}

\begin{proof}
  First note that if $G_1$ and $G_2$ agree on all factors $H_j$, then
  for each $j$ there exists $w_j' \in V(G_1^j) \cap
  V(G_2^j)$. Let $w$ be a vertex of $\boxtimes \cH$ such that
  $w_j = w'_j$ for all $j$. Since for each $j$, $G_1^j = \{w_j\}$ or
  $G_1^j = H_j$, the vertex $w$ belongs to $G_1$. Similarly, $w$
  belongs to $G_2$ and thus $G_1$ and $G_2$ have a non-empty
  intersection. Conversely, let $u \in V(G_1) \cap V(G_2)$ and note
  that $u_{j} \in V(G_1^j)$ for every $j$. Consequently, either
  $G_1^j = H_j$ or $G_1^j = \{u_j\}$. Similarly, either $G_2^j = H_j$
  or $G_2^j = \{u_j\}$. In both cases, $G_1$ and $G_2$ agree on $H_j$.

  Let $\{G_{i_k}\}_{k\in K}$ be a collection of pairwise
  intersecting pieces. By the first statement, any two
  pieces of this collection agree on all factors $H_j \in
  \cH$. Consequently, for any factor $H_j$, there exists
  $w_j' \in V(H_j)$ such that for any $k \in K$, $G_{i_k}^j = \{w_j\}$
  or $G_{i_k}^j = H_j$. Consider the vertex $w$ of $\boxtimes \cH$
  such that $w_j = w'_j$ for all $j$ and observe that $w$ belongs to every piece
  of the collection.
\end{proof}

We say that an UGP satisfies the \emph{$3$-piece condition} if for any
three pairwise intersecting pieces $G_1, G_2, G_3$, there exists a
piece $G_4$ intersecting $G_1$, $G_2$, and $G_3$ such that for every
factor $H_j \in \cH$, if for two pieces $G_{i_1}, G_{i_2}$ among
$G_1, G_2, G_3$ we have $G_{i_1}^j = G_{i_2}^j = H_j$, then
$G_4^j = H_j$.

\begin{proposition}\label{prop-clique-UGP}
  If an UGP $G$ over $\cH$   satisfies the $3$-piece condition, then
  every clique of $G$ is contained in a piece of $G$.
\end{proposition}

\begin{proof}
Since $G$ is locally finite, the cliques of $G$ are finite and we
  can proceed by induction on the size $k$ of the clique. By
  definition of $G$, each edge belongs to a piece of $G$. Suppose that
  the assertion holds for all cliques of size at most $k-1$ and
  consider a clique $K$ of size $k$. Let $u, v, w$ be three vertices
  of $K$. Since $K\setminus \{w\}$ is a clique of size $k-1$, there
  exists a piece $G_1$ containing all vertices of
  $K\setminus\{w\}$. If $w \in V(G_1)$, we are done. Assume now that
  $w \notin V(G_1)$. Similarly, we can assume there exist pieces $G_2$
  and $G_3$ such that $K \cap V(G_2) = K\setminus\{u\}$ and
  $K \cap V(G_3) = K\setminus\{v\}$.  Since
  $u \in V(G_1) \cap V(G_3)$, the pieces $G_1$ and $G_3$ agree on
  every factor $H_j \in \cH$. Similarly, $G_1$ and $G_2$ as well as
  $G_2$ and $G_3$ agree on every factor $H_j \in \cH$. Since
  $u \notin V(G_2)$, necessarily there exists a factor $H_{j_2}$ such
  that $G_2^{j_2}$ does not contain $u_{j_2}$. Thus $G_2^{j_2}$
  consists of a single vertex $v_2 \neq u_{j_2}$. Since both $G_1$ and
  $G_3$ agree with $G_2$ on $H_{j_2}$ and since they both contain
  $u_{j_2}$, necessarily $G_1^{j_2} = G_3^{j_2} = H_{j_2}$. Similarly,
  there exist $H_{j_1}, H_{j_3} \in \cH$ and vertices
  $v_1 \in H_{j_1}$ and $v_3 \in H_{j_3}$ such that
  $G_1^{j_1} = \{v_1\}$, $G_2^{j_1} = G_3^{j_1} = H_{j_1}$,
  $G_3^{j_3} = \{v_3\}$, and $G_1^{j_3} = G_2^{j_3} = H_{j_3}$.

  By the $3$-piece condition, there exists $G_4$ intersecting $G_1$,
  $G_2$, and $G_3$ such that for every factor $H_j \in \cH$, if for two
  pieces $G_{i_1}, G_{i_2}$ among $G_1, G_2, G_3$ we have
  $G_{i_1}^j = G_{i_2}^j = H_j$, then $G_4^j = H_j$. We assert that
  $K$ is a clique of $G_4$. Pick any vertex $x \in K$ and note that
  $x$ belongs to at least two pieces among $G_1, G_2, G_3$, say to
  $G_1$ and $G_2$. For each factor $H_j \in \cH$, if $G_4^j \neq H_j$,
  then since $G_4$ agrees with $G_1$ and $G_2$ and by the definition
  of $G_4$, either $G_4^j = G_1^j = \{x_j\}$ or
  $G_4^j = G_2^j = \{x_j\}$. Consequently, $x$ is a vertex of $G_4$
  and thus $K$ is a clique of $G_4$.
  Therefore, all vertices of $K$ belong to a piece of $G$ and since
  any piece is an induced subgraph of $\boxtimes \cH$, we conclude
  that $K$ is a clique of this piece.
\end{proof}

\begin{theorem}\label{th-UGP-3piece}
  If an UGP $G$ over $\cH$  satisfies the $3$-piece condition and every piece of $G$ is clique-Helly,
  then $G$ is a clique-Helly graph. Furthermore, if the clique complex $X(G)$ of $G$ is simply connected, then
  $G$ is a Helly graph.
\end{theorem}

\begin{proof}
Since $G$ has finite cliques, we can use
  Proposition~\ref{clique_Helly_triangle} to establish the
  clique-Helly property for $G$.  Pick any triangle $T = u_1u_2u_3$ of
  $G$ and let $T^*$ be the set of vertices of $G$ adjacent to at least
  two vertices of $T$. For any $v \in T^*$, by
  Proposition~\ref{prop-clique-UGP}, there exists a piece containing a
  triangle $vu_iu_j$; let $P^*$ be the set of all pieces containing
  such triangles. Since the pieces of $P^*$ piecewise intersect, by
  the first assertion of Lemma~\ref{lem-UGP-agreement}, they pairwise
  agree on every factor $H_j \in \cH$. By the second assertion of
  Lemma~\ref{lem-UGP-agreement}, there exists a vertex $w \in G$ such
  that either $G_i^j = \{w_j\}$ or $G_i^j = H_j$ for any piece $G_i$
  of $P^*$. Therefore, $w$ belongs to every piece of $P^*$.

  For each factor $H_j \in \cH$, let $T_j = \{u_j: u \in T\}$ and
  $T_j^* = \{v_j: v \in T_j^*\}$. Note that $T_j$ is either a vertex,
  an edge, or a triangle in $H_j$. Moreover, in the first two cases,
  there exists $u_j \in T_j$ that belongs to the $1$--ball of every
  vertex $v_j \in T_j^*$. If $T_j$ is a triangle, then every vertex
  $v_j \in T_j^*$ is in the $1$--ball of at least two vertices of
  $T_j$. Since $H_j$ is clique-Helly, in all three cases, there exists
  a vertex $w_j \in V(H_j)$ belonging to the $1$--ball of each vertex
  $v_j \in T_j^*$. Observe that if there exists a piece $G_i$ of $P^*$
  such that $G_i^j$ contains only one vertex, then necessarily, $T_j$
  is a vertex or an edge and we can choose $w_j \in V(H_j)$ such that
  $G_i^j = \{w_j\}$.

  Let $w^*$ be the vertex of $G$ such that $w^*_j = w_j$ for every
  factor $H_j \in \cH$. By our choice of $w_j$, for any piece $G_i$ of
  $P^*$ such that $G_i^j$ contains only one vertex, $G_i^j = \{w_j\}$
  and for any other piece $G_i$ of $P^*$, $w_j$ is a vertex of
  $G_i^j = H_j$. Therefore $w^*$ is a vertex that belongs to all
  pieces of $P^*$. For any vertex $v \in T^*$ and any factor
  $H_j \in \cH$, $v_j$ is in the $1$--ball of $w_j$ in $H_j$ by our
  choice of $w_j$. Since each piece $G_i$ of $G$ is an induced
  subgraph of $\boxtimes \cH$, $w^*$ is in the $1$--ball in $G$ of all
  vertices $v$ of $T^*$, establishing that $G$ is clique-Helly.

  The second assertion of the theorem follows from Theorem~\ref{t:lotogloHell}.
\end{proof}

Given a family $\cH = \{H_j\}_{j\in J}$ of locally finite graphs, an
\emph{abstract graph of subproducts (GSP)} $(\cH,\cG,\ell)$ is given by a connected graph
$\cG$ without infinite clique and a map $\ell: V(\cG) \to 2^\cH$
satisfying the following conditions:
\begin{enumerate}[{(A}1)]
\item $\ell(v)$ is a finite subset of $\cH$ for each $v \in V(\cG)$;
\item for each edge $uv \in E(\cG)$, $\ell(u) \neq \ell(v)$.
\end{enumerate}

A \emph{realization} of an abstract GSP $(\cH,\cG,\ell)$ is a set of
maps
\[\left\{p_v: \cH \setminus \ell(v) \to \bigcup_{j\in J}
    V(H_j)\right\}_{v \in V(\cG)}\]
satisfying the following conditions:
\begin{enumerate}[{(A}3)]
\item for each $v \in V(\cG)$, $p_v(H_j) \in V(H_j)$ for every factor
  $H_j \in \cH \setminus \ell(v)$;
\item[{(A}4)] for any vertices $u,v \in V(\cG)$, there is an edge
  $uv \in E(\cG)$ if and only if for every factor
  $H_j \in \cH \setminus (\ell(u) \cup \ell(v))$, $p_u(H_j) = p_v(H_j)$.
\end{enumerate}
A GSP admitting a realization is called a \emph{realizable GSP}.

\begin{proposition}
  For any realizable GSP $(\cH,\cG,\ell)$ and any of its realizations
  $\{p_v\}_{v \in V(\cG)}$, we can define an UGP
  $G(\cG) = \bigcup_{v \in V(\cG)} G_v$ where there is a piece
  $G_v = \boxtimes_{j\in J} G_v^j$ for each $v \in V(\cG)$ such that
  $G_v^j = H_j$ if $H_j \in \ell(v)$ and $G_v^j = \{p_v(H_j)\}$
  otherwise.

  Conversely, any UGP $G \subseteq \boxtimes \cH$ is the realization
  of a realizable GSP over $\boxtimes \cH$.
\end{proposition}

\begin{proof}
  First notice that condition (A4) is equivalent to the following
  condition on the pieces of $G(\cG)$:
  \begin{enumerate}[{(A}4')]
  \item for any vertices $u,v \in V(\cG)$, there is an edge
    $uv \in E(\cG)$ if and only if $V(G_u) \cap V(G_v) \neq \emptyset$.
  \end{enumerate}

  In order to show that $G(\cG)$ is an UGP, we must show that it is
  locally finite. Consider a vertex $u \in G(\cG)$ that has an
  infinite number of neighbors. Since each piece containing $u$ is
  locally finite, there are an infinite number of pieces containing
  $u$. By Condition (A4'), these pieces form an infinite clique in
  $\cG$, a contradiction. Moreover, if there exists two vertices
  $u,v \in V(\cG)$ such that the pieces $G_u$ and $G_v$ coincide, then
  $\ell(u) = \ell(v)$ and for any $H_j \in \cH \setminus \{\ell(u)\}$,
  $p_u(H_j) = p_v(H_j)$. Consequently, $uv \in E(\cG)$ and
  $\ell(u) = \ell(v)$, contradicting (A2).

  Conversely, given an UGP $G$ over $\cH$, we construct a realizable
  GSP as follows. In $\cG$, there is a vertex $v_i$ for each piece
  $G_i$ of $G$, and we set $\ell(v_i)= \{H_j \in \cH : G_i^j =
  H_j\}$. For each $H_j \notin \ell(v_i)$, there exists
  $w_j \in V(H_j)$ such that $G^j_i = \{w_j\}$ and we set
  $p_{v_i}(H_j) = w_j$.  For any vertices $v_i, v_{i'} \in V(\cG)$,
  there is an edge $v_iv_{i'} \in E(\cG)$ if and only if for every
  factor
  $H_j \notin \ell(v_i) \cup \ell(v_{i'}), p_{v_i}(H_j) =
  p_{v_{i'}}(H_j)$.

  Since each piece $G_i$ is a finite subproduct of $\boxtimes H$,
  $\ell(v_i)$ is finite for each $v_i \in V(\cG)$ and thus (A1)
  holds. By definition of $p_{v_i}$ and of the edges of $E(\cG)$, (A2)
  and (A4) also hold. Observe also that $G(\cG)$ and $G$ are
  isomorphic and thus $G$ is the realization of $\cG$.  It remains to
  show that $\cG$ does not contain infinite cliques. By (A4'), if
  there exists an infinite clique in $\cG$, then there exists an
  infinite collection $\{G_{i_k}\}_{k \in K}$ of pairwise intersecting
  pieces. By Lemma~\ref{lem-UGP-agreement}, this implies that there
  exists a vertex $w$ that belongs to every piece $G_{i_k}$. Since all
  pieces of $G$ are distinct and since $w$ belongs to an infinite
  number of pieces, there exists an infinite collection of factors
  $\{H_{j'}\}_{j' \in J'}$ such that for each $H_{j'}$ there exists a piece
  $G_{i_k}$ with $w \in G_{i_k}$ and
  $G_{i_k}^{j'}=H_{j'}$. Consequently, for each $j' \in J'$, one can
  find a vertex $w^{j'} \in \boxtimes \cH$ in $G$ obtained from $w$ by
  replacing the coordinate $w_{j'}$ by one of its neighbors in
  $H_{j'}$. All the $w^{j'}$ constructed in this way are distinct and
  they are all neighbors of $w$ in $G$.  Consequently, $w$ has
  infinitely many neighbors in $G$ and thus $G$ is not locally finite, a
  contradiction.
\end{proof}

We say that a GSP $(\cH,\cG,\ell)$ satisfies the \emph{product-Gilmore
  condition} if for every triangle $\cT=x_1x_2x_3$ of $\cG$ there
exists $y \in V(\cG)$ such that $y = x_i$ or $y \sim x_i$ for
$1 \leq i \leq 3$ and
$(\ell(x_1)\cap\ell(x_2)) \cup (\ell(x_2)\cap\ell(x_3)) \cup (\ell(x_1)\cap\ell(x_3)) \subseteq \ell(y)$.

\begin{proposition}\label{prop-GSP-Gilmore}
  For a realizable GSP $(\cH,\cG,\ell)$ and any of its realizations
  $\{p_v\}_{v \in V(\cG)}$, the UGP $G(\cG)$ obtained from $\cG$ and
  $\{p_v\}_{v \in V(\cG)}$ satisfies the $3$-piece condition if and
  only if $(\cH,\cG,\ell)$ satisfies the product-Gilmore condition.
\end{proposition}

\begin{proof}
  Assume that $(\cH,\cG,\ell)$ satisfies the product-Gilmore condition.
  By condition (A4'), two pieces in the UGP $G(\cG)$ obtained from a
  realization of a GSP $\cG$ intersect if and only if there is an edge
  between the corresponding vertices of $\cG$. Thus, it is enough to
  consider three pieces $G_{x_1}, G_{x_2}, G_{x_3}$ corresponding to
  three vertices $x_1, x_2, x_3$ that are pairwise adjacent in
  $\cG$. By our assumption, there exists a vertex $y \in V(\cG)$ such
  that $y = x_i$ or $y \sim x_i$ for any $1 \leq i \leq 3$ and such
  that
  $(\ell(x_1)\cap\ell(x_2)) \cup (\ell(x_2)\cap\ell(x_3)) \cup
  (\ell(x_1)\cap\ell(x_3)) \subseteq \ell(y)$.
Consider the piece $G_y$ in $G(\cG)$. By condition (A4'), $G_y$
  intersect $G_{x_1}$, $G_{x_2}$, and $G_{x_3}$. Moreover, since for
  any factor $H_j \in \cH$, if $G_{x_1}^j = G_{x_2}^j = H_j$, by the
  definition of $G(\cG)$, we obtain
  $H_j \in \ell(x_1) \cap \ell(x_2) \subseteq \ell(y)$.
Similarly, for any factor $H_j\in \cH$ such that
  $G_{x_2}^j = G_{x_3}^j = H_j$ or $G_{x_1}^j = G_{x_3}^j = H_j$, we
  have $H_j \in \ell(y)$.  This establishes the $3$-piece condition
  for $G(\cG)$.

  Conversely, suppose that $G(\cG)$ satisfies the $3$-piece condition
  and consider a triangle $x_1x_2x_3$ of $\cG$ and the three
  corresponding pieces $G_{x_1}, G_{x_2}, G_{x_3}$ of $G(\cG)$. By
  (A4'), $V(G_{x_1}), V(G_{x_2}), V(G_{x_3})$ pairwise intersect. By
  the $3$-piece condition, there exists a vertex $x_4 \in V(\cG)$ such
  that $V(G_{x_4})$ intersects $V(G_{x_1}), V(G_{x_2})$, and
  $V(G_{x_3})$, i.e., $x_4$ either coincides with or is adjacent to
  each $x_i$, $1\leq i\leq 3$. Moreover, for each
$H_j \in  \ell(x_1) \cap \ell(x_2)$,
  $G_{x_1}^j = G_{x_2}^j = H_j$ and the definition of $G_{x_4}$
  implies that $G_{x_4}^j = H_j$, i.e., $H_j \in
  \ell(x_4)$. Consequently, $\ell(x_1)\cap\ell(x_2) \subseteq \ell(x_4)$ and
  similarly,
  $(\ell(x_2)\cap \ell(x_3)) \cup (\ell(x_1)\cap \ell(x_3)) \subseteq
  \ell(x_4)$. This establishes the product-Gilmore condition for
  $(\cH,\cG,\ell)$.
\end{proof}

From Proposition~\ref{direct-products},
Proposition~\ref{prop-GSP-Gilmore}, and Theorem~\ref{th-UGP-3piece}  we
obtain the following corollary:

\begin{corollary}\label{cor-GSP}
  Consider a realizable GSP $(\cH,\cG,\ell)$ and any of its
  realizations $\{p_v\}_{v \in V(\cG)}$.
  If $(\cH,\cG,\ell)$ satisfies the product-Gilmore condition and if
  each factor $H \in \cH$ is clique-Helly, then $G(\cG)$ is a
  clique-Helly graph. Furthermore, if the clique complex $X(G(\cG))$ is simply connected,
  then $G(\cG)$ is a Helly graph.
\end{corollary}

Thickenings of locally finite median graphs (i.e., of \catz cube
complexes) is an instructive example of clique-Helly graphs that can
be obtained via Theorem~\ref{th-UGP-3piece} or
Corollary~\ref{cor-GSP}. The pieces of a median graph $G$ seens as an
UGP are the thickenings of the maximal cubes of $G$. The fact that it
satisfies the product-Gilmore condition follows from the fact that the
cell hypergraph is conformal which follows from the cube condition of
the \catz cube complex $X_{\text{cube}}(G)$, Lemma~\ref{GMC} and
Proposition~\ref{3-cell}.

\subsection{Thickening}

The direct product  of graphs considered above is the $l_{\infty}$ version of the Cartesian
product. Thus, when we turn all $k$--cubes of the Cartesian product
of $k$ paths into simplices, then we have the corresponding direct
product of $k$ paths. More generally, a similar operator
transforms median graphs into Helly graphs: let $G^{\Delta}$ be
the graph having the same vertex set as $G,$ where two vertices
are adjacent if and only if they belong to a common cube of $G$;
$G^{\Delta}$ is called the \emph{thickening} of $G$ (for $l_{\infty}$-metrization of
cube complexes, of median graphs and, more generally, of median spaces, see \cite{Bowditch,vdV_median}).

\begin{proposition}[\cite{BavdV3}]\label{p:medHel}
If $G$ is a locally finite median graph, then $G^{\Delta}$ is a Helly graph and each maximal clique of $G^{\Delta}$ is a cube of $G$.
\end{proposition}

The \emph{thickening} $X^{\Delta}$ of an abstract cell complex $X$ is
a graph obtained from $X$ by making adjacent all pairs of vertices of
$X$ belonging to a common cell of $X$. Equivalently, the thickening of
$X$ is the $2$-section $[\cH(X)]_2$ of the hypergraph $\cH(X)$.  We
say that an abstract cell complex $X$ is \emph{simply connected} if
the clique complex of its thickening $X^{\Delta}$ is simply connected.

Proposition~\ref{p:medHel} of Bandelt and van de Vel  was extended to the thickenings of the abstract cell complexes arising from swm-graphs and from hypercellular graphs.

\begin{proposition}[\cite{CCHO,ChKnMa19}]\label{swm-hypercellular}
  The thickening $G^{\Delta} := X(G)^{\Delta}$ of the abstract cell
  complex $X(G)$ associated to any locally finite swm-graph or any
  hypercellular graph $G$ is a Helly graph. Each maximal clique of
  $G^{\Delta}$ is a cell of $X(G)$.
\end{proposition}

The existing proofs of Propositions~\ref{p:medHel} and \ref{swm-hypercellular}
are based on the following global property of $G^{\Delta}$: each ball of
$G^{\Delta}$ defines a gated subgraph of $G$ thus $G^{\Delta}$ is
Helly because the gated sets of $G$ satisfy the finite
Helly property. Proposition
\ref{cell-complex-cliqueHelly} allows us to provide a new proof of
Propositions~\ref{p:medHel} and \ref{swm-hypercellular}.
Namely, the results of Section~\ref{sec-abstract} establish that CAT(0) cube complexes,
hypercellular complexes, and swm-complexes satisfy the 3-cell and the
graded monotonicity conditions.  Since all such complexes are simply connected
and their cells are gated,  Propositions~\ref{p:medHel} and \ref{swm-hypercellular} can be viewed
as particular cases of Theorem \ref{t:lotogloHell} and the following general result:

\begin{proposition}\label{graph-cell-complex}
  If $X$ is an abstract cell complex defined on the vertex-set of a
  graph $G$ such that each edge of $G$ is contained in a cell of $X$
  and each cell of the cell-hypergraph $\cH(X)$ is gated in $G$ and
  $\cH(X)$ satisfies the 3-cell and the graded monotonicity
  conditions, then the thickening $X^{\Delta}$ is a clique-Helly graph
  and each maximal clique of $X^{\Delta}$ is the thickening of a cell
  of $X$. Additionally, if $X$ is simply connected, then $X^{\Delta}$
  is Helly.
\end{proposition}

\subsection{Coarse Helly graphs}\label{sec-coarse}
The coarse Helly property of a graph $G$ is a property that can be
used to show via Hellyfication that a group acting on $G$
geometrically is Helly. In this subsection, we recall the result
of~\cite{ChEs} that $\delta$-hyperbolic graphs are coarse Helly and we
deduce from a result of~\cite{Chepoi1998} that several subclasses of
weakly modular graphs (in particular, cube-free median graphs,
hereditary modular graphs, and 7-systolic graphs) are coarse Helly.

\begin{proposition}[\cite{ChEs}]\label{coarse-Helly-hyperbolic}
  If $G$ is a $\delta$-hyperbolic graph, then $G$ is coarse Helly with
  constant $2\delta$.
\end{proposition}

The idea of the proof of Proposition~\ref{coarse-Helly-hyperbolic} comes from the proof of the Helly property for trees. Let ${\mathcal B}=\{ B_{r_i}(x_i): i\in I\}$ be a finite collection of pairwise intersecting balls of $G$. Pick an arbitrary basepoint vertex $z$ of $G$ and suppose that $B_{r_1}(x_1)$ is a ball of ${\mathcal B}$ maximizing $d(z,x_i)-r_i, i\in I$ (equivalently, $B_{r_1}(x_1)$ is a ball of $\mathcal B$ maximizing $d(z,B_{r_i}(x_i)), i\in I$). If $d(z,B_{r_1}(x_1))\le 2\delta$, then $z\in B_{r_i+2\delta}(x_i), i\in I$ and we are done. Let $c$ be a vertex on a shortest path between $z$ and $x_1$ at distance $r_1$ from $x_1$. Then using the hyperbolicity of $G$ and the choice of $B_{r_1}(x_1)$ it can be shown that $d(c,x_i)\le r_i+2\delta$.

\begin{proposition}[\cite{Chepoi1998}]\label{coarse-Helly-hwm}
  If $G$ is a weakly modular graph not containing isometric cycles of
  length $>5$, houses, or 3-deltoids (see
  Figure~\ref{fig-house-3deltoid}), then $G$ is coarse Helly with
  constant 1. In particular, cube-free median graphs, hereditary
  modular graphs, and 7-systolic graphs are coarse Helly.
\end{proposition}

\begin{figure}[h]
\centering
\includegraphics{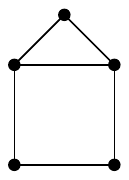}\qquad\qquad\qquad\qquad\includegraphics{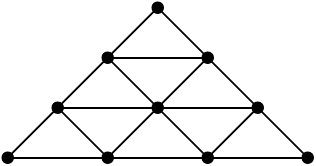}\caption{A house (left) and a $3$-deltoid (right).}\label{fig-house-3deltoid}
\end{figure}

In~\cite{Chepoi1998}, the established result was actually stronger: if
$S$ is a finite set of vertices of a graph $G$ as in
Proposition~\ref{coarse-Helly-hwm} and $d(x_i,x_j)\le r_i+r_j+1$ for
all $x_i, x_j \in S$, then there exists a clique of $G$ hitting all
balls $B_{r_i}(x_i), x_i\in S$. The idea of the proof is to show that
if a clique $C'$ of $G$ hits the balls of a subfamily
$B_{r_i}(x_i), x_i\in S'$ and $x_j\in S\setminus S'$, then the clique
$C'$ can be transformed into a clique $C$ which hits $B_{r_j}(x_j)$
and all balls centered at the vertices of $S'$.

It is known that the systolic (bridged) graphs satisfying the conditions of
Proposition~\ref{coarse-Helly-hwm} are all
hyperbolic~\cite{CCHO,ChDrEsHaVa}. Cube-free median graphs and, more
generally, hereditary modular graphs (which by a result
of~\cite{Ba_hereditary} are exactly the graphs in which all isometric
cycles have length 4) in general are not hyperbolic. On the other
hand, general median graphs are not coarse Helly: already the cubic
grid ${\mathbb Z}^3$ is not coarse Helly as shown by the following
example.

\begin{example}
In $\Z^3$, for any integer $n$, consider 4 balls of radius $2n$
centered at
$x_1 = (-2n,2n,-2n), x_2 = (2n,2n,2n), x_3=(-2n,-2n,2n), x_4 =
(2n,-2n,-2n)$. Observe first that for any two such nodes $x_l, x_{l'}$,
$d(x_l,x_{l'}) = 4n$ and thus the four balls pairwise intersect. We show
that for any node $y = (i,j,k) \in \Z^3$,
$\max \{d(y,x_l): 1\leq l \leq 4\} \geq 6n$. Assume that $y$ minimizes
this maximum. Observe that if $y \notin [-2n,2n]^3$, then its gate
$y'$ in the box $[-2n,2n]^3$ is strictly closer to each $x_l$,
contrary to our choice of $y$. Consequently, $i, j, k \in [-2n,2n]$
and
$d(y,x_1) = i+2n+2n-j+k+2n = 6n +i -j +k, d(y,x_2) = 6n -i -j -k,
d(y,x_3) = 6n +i + j -k, d(y,x_4) = 6n -i + j + k$ and thus
$\Sigma_{i=1}^4 d(x_i,y) = 24n$. Therefore
$\max \{d(y,x_l): 1\leq l \leq 4\} \geq 6n$.
\end{example}

Analogously, the triangular grid (alias, the systolic plane) is also not coarse Helly:

\begin{example} ${\mathbb T}_3$ is the graph of the tiling of the
  plane into equilateral triangles with side 1. ${\mathbb T}_3$ is a
  bridged graph.  Pick three vertices $x_1,x_2,z=x_3$ of
  ${\mathbb T}_3$ which define a deltoid $\Delta(x_1,x_2,x_3)$ of size
  $6n$, i.e., an equilateral triangle of ${\mathbb T}_3$ with side
  $6n$. Consider the three balls
  $B_{3n}(x_1), B_{3n}(x_2), B_{3n}(x_3)$.  We assert that
  $\max\{d(y,x_i): 1\le i\le 3\}\ge 4n$ for any vertex $y$ of
  $V({\mathbb T}_3)$. If $y\notin \Delta(x_1,x_2,x_3)$, then $y$ is in
  one of the halfplanes defined by the sides of $\Delta(x_1,x_2,x_3)$
  and not containing $\Delta(x_1,x_2,x_3)$, say in the halfspace
  defined by $x_1$ and $x_2$.  But then $d(x_3,y)\ge 6n$ because $x_3$
  has distance $\ge 6n$ to any vertex of ${\mathbb T}_3$ defined by
  the line between $x_1$ and $x_2$.  Now suppose that
  $y\in \Delta(x_1,x_2,x_3)$. It can be shown easily by induction on
  $k$ that if $\Delta(x_1,x_2,x_3)$ is a deltoid of size $k$ of
  ${\mathbb T}_3$, then $d(y,x_1)+d(y,x_2)+d(y,x_3)=2k$ for any
  $y\in \Delta(x_1,x_2,x_3)$. This shows that in our case
  $d(y,x_1)+d(y,x_2)+d(y,x_3)\ge 12n$, i.e.,
  $\max\{d(y,x_i): 1\le i\le 3\}\ge 4n$.
\end{example}

\subsection{Nerve graphs  of clique-hypergraphs}\label{s:nerve}

We first show that (clique-)Hellyness is preserved by taking the nerve
complex $N(\cX(G))$ of the clique-hypergraph $\cX(G)$ of a Helly graph
$G$. Nerve complexes of clique-hypergraphs are also called
\emph{clique graphs} in the literature, see e.g.~\cite{BaPr}.
Note that in general, the nerve complex $N(\cX(G))$ of the
clique-hypergraph of a graph $G$ is not a flag simplicial
complex. However, if $G$ is clique-Helly, then $N(\cX(G))$ is a flag
simplicial complex.

\begin{lemma}
  For any locally finite graph $G$, $N(\cX(G))$ is a flag simplicial
  complex if and only if $G$ is a (finitely) clique-Helly graph.
\end{lemma}

\begin{proof}
  By the definition of $N(\cX(G))$, $N(\cX(G))$ is a flag simplicial
  complex if and only if any finite set of pairwise intersecting
  cliques $K_1, K_2, \ldots, K_p$ of $G$ have a non-empty
  intersection. This is precisely the definition of a finitely
  clique-Helly graph. Since $G$ is locally finite, $G$ is finitely
  clique-Helly if and only if $G$ is clique-Helly.
\end{proof}

The first assertion of the following result was first proved by
Escalante~\cite{Esc73} (he also proved the converse that any
clique-Helly graph is the clique graph of some graph).

\begin{proposition}If $G$ is a locally finite clique-Helly graph, then the nerve graph
  $NG(\cX(G))$ of the clique-hypergraph $\cX(G)$ is a clique-Helly
  graph and its flag-completion is a clique-Helly complex.

  If $G$ is a locally finite Helly graph, then $NG(\cX(G))$ is a Helly
  graph and its flag-completion is a Helly complex.
\end{proposition}

\begin{proof}
  Let $G$ be a locally finite clique-Helly graph. Let $G'$ be the
  nerve graph of the clique-hypergraph $\cX(G)$. Since $G$ is locally
  finite, $G'$ is also locally finite. We prove that $G'$ is
  clique-Helly by using the triangle criterion from
  Proposition~\ref{clique_Helly_triangle}. Let $uvw$ be a triangle in
  $G'$.  It corresponds to three pairwise intersecting, and thus
  intersecting, maximal cliques in $G$, denoted by the same symbols
  $u,v,w$. Observe that all vertices of
  $(u\cap v) \cup (v\cap w) \cup (w\cap u)$ are pairwise adjacent in
  $G$ and thus $u\cap v,v\cap w, w\cap u$ are all contained in a
  common maximal clique $x$ in $G(X)$. We claim that every vertex $y$
  in $G'$ that is adjacent to $u$ and $v$ in $G'$ is also adjacent to
  $x$ in $G'$.  This is so because in $G$, the maximal clique $y$
  intersects $u$ and $v$, hence intersects $u\cap v$ since $G$ is a
  clique-Helly graph. Since $u \cap v \subseteq x$, $y$ intersects $x$
  in $G$ and thus $x \sim y$ in $G'$.  Similarly, the vertex $x\in G'$
  is a universal vertex for triangles containing $v,w$ and $w,u$ in
  $G'$. Consequently, the nerve graph $G'$ is clique-Helly.

  Suppose now that $G$ is a Helly graph, i.e., by
  Theorem~\ref{t:lotogloHell} that the clique complex $X(G)$ is simply
  connected and that $G$ is a clique-Helly graph. By the first part of
  the theorem, the $1$--skeleton $G'=G(Y)$ of the nerve complex $Y$ of
  the clique-hypergraph $\cX(G)$ is clique-Helly. By Borsuk's Nerve
  Theorem~\cite{Bor,Bj}, $X(G)$ and $Y$ have the same homotopy
  type. Consequently, $Y$ is also simply connected. By
  Theorem~\ref{t:lotogloHell}, this implies that $G' = G(Y)$ is Helly.
\end{proof}

We now show that the clique-Hellyness of the nerve graph of a
clique-hypergraph is preserved by taking covers.

\begin{theorem}\label{t:nerve-cover}
  Given two locally finite graphs $G, G'$ such that the clique complex
  $X(G)$ is a cover of the clique complex $X(G')$, then the nerve
  graph $NG(\cX(G))$ is clique-Helly if and only if the nerve graph
  $NG(\cX(G'))$ is clique-Helly.
\end{theorem}

By Theorem~\ref{t:lotogloHell}, we immediately get the following
corollary since the nerve complex of the maximal simplices of a simply
connected simplicial complex is simply connected by Borsuk's Nerve
Theorem~\cite{Bor,Bj}.

\begin{corollary}\label{c:nerve-cover}
  For a locally finite graph $G$, the nerve graph $NG(\cX(\tG))$ of
  the clique-hypergraph of the $1$--skeleton $\tG$ of the universal
  cover $\tX(G)$ of $X(G)$ is Helly if and only if the nerve graph
  $NG(\cX(G))$ of the clique-hypergraph $\cX(G)$ is clique-Helly.
\end{corollary}

The proof of Theorem~\ref{t:nerve-cover} follows from the following lemma
establishing that a covering map between the clique complexes of
two graphs extends to a covering map between the nerve complexes of
the corresponding clique-hypergraphs.

\begin{lemma} \label{lem:covering-nerves}
  Given two locally finite simple graphs $G, G'$, any covering map
  $\varphi: X(G) \to X(G')$ induces a covering map from $N(\cX(G))$ to
  $N(\cX(G'))$.
\end{lemma}

\begin{proof}
  In the nerve complex $N(\cX(G))$, the vertices are the maximal
  cliques of $G$ and a finite set $\sigma = \{K_1, \ldots, K_p\}$ of
  cliques of $G$ is a simplex of $N(\cX(G))$ if
  $\bigcap_{i=1}^p K_i \neq \emptyset$. Since $G$ is locally finite, each maximal clique $K$ of $G$ is
  finite. We extend the map $\varphi$ to all cliques of $G$: for any
  clique $K = \{u_1, \ldots, u_k\}$ of $G$, we set
  $\varphi(K) = \{\varphi(u_1),\ldots, \varphi(u_k)\}$. Observe that
  for any clique $K = \{u_1, \ldots, u_k\}$ of $G$, $u_i \sim u_j$ and
  thus $\varphi(u_i) \sim \varphi(u_j)$. Since $G'$ does not contain
  loops, $\varphi(K)$ is a clique of $G'$ and
  $|\varphi(K)| = |K|$.

  Consider two cliques $K$ of $G$ and $K'$ of $G'$ such that
  $K' = \varphi(K)$. For any $u \in K$, $\varphi$ induces a bijection
  between the cliques containing $u$ and the cliques containing
  $\varphi(u)$. Consequently, $K$ is a maximal clique of $G$ if and
  only if $K'$ is a maximal clique of $G'$. Therefore, $\varphi$
  induces a map from $V(N(\cX(G)))$ to $V(N(\cX(G')))$.

  We now prove a useful claim.

  \begin{claim}\label{claim-intersection-cliques-nerve}
    For any maximal cliques $K_1,K_2$ of $G$ such that
    $K_1 \cap K_2\ne \emptyset$, we have
    $\varphi(K_1) \cap \varphi(K_2) = \varphi(K_1 \cap K_2)$.
  \end{claim}

  \begin{proof}
    The inclusion
    $\varphi(K_1 \cap K_2) \subseteq \varphi(K_1) \cap \varphi(K_2)$
    is trivial. Suppose now that the reverse inclusion does not hold,
    i.e., that there exist $u_1 \in K_1 \setminus K_2$ and
    $u_2 \in K_2 \setminus K_1$ such that
    $\varphi(u_1) = \varphi(u_2)$. Pick $u \in K_1 \cap K_2$ and
    observe that $u \sim u_1$ since $K_1$ is a clique and $u \sim u_2$
    since $K_2$ is a clique. Consequently, the map $\varphi$ is not
    locally injective at $u$, a contradiction.
  \end{proof}

  Note that if $\sigma = \{K_1, \ldots, K_p\}$ is a simplex of
  $N(\cX(G))$, then there exists $u \in \bigcap_{i=1}^p
  K_i$. Consequently, $\varphi(u) \in \bigcap_{i=1}^p \varphi(K_i)$
  and thus the image of a simplex of $N(\cX(G))$ is a simplex of
  $N(\cX(G'))$. Thus $\varphi$ is a simplicial map from
  $N(\cX(G))$ to $N(\cX(G'))$. Moreover, for any
  $1 \leq i < j \leq p$, $u \in K_i \cap K_j$ and consequently, by
  Claim~\ref{claim-intersection-cliques-nerve},
  $\varphi(K_i) \cap \varphi(K_j) = \varphi(K_i \cap K_j)$. Since
  $|\varphi(K_i)| = |K_i|$ and $|\varphi(K_j)| = |K_j|$, this implies
  that if $K_i \neq K_j$, then $\varphi(K_i) \neq
  \varphi(K_j)$. Consequently, we have $|\varphi(\sigma)| = |\sigma|$.

  We now show that $\varphi$ is locally surjective. Let
  $K_0 \in V(N(\cX(G)))$ and $K'_0 = \varphi(K) \in V(N(\cX(G')))$ and
  consider a simplex $\sigma' = \{K'_0, K'_1, \ldots, K'_p\}$ in
  $N(\cX(G')$. By definition of $N(\cX(G')$, there exists
  $u' \in \bigcap_{i=1}^p K'_i$. Since $K'_0 = \varphi(K_0)$, there
  exists $u \in K_0$ such that $u' = \varphi(u)$. Since $\varphi$ is a
  covering map from $G$ to $G'$, for each $1 \leq i\leq p$, there
  exists $K_i \in V(N(\cX(G)))$ such that $u \in K_i$ and
  $K'_i = \varphi(K_i)$. Since $u \in \bigcap_{i=1}^p K_i$,
  $\sigma = \{K_0, K_1, \ldots, K_p\}$ is a simplex of $N(\cX(G')$
  that is mapped to $\sigma'$ by $\varphi$.

  We now show that $\varphi$ is locally injective.  Consider
  $K_0 \in V(N(\cX(G)))$ and assume that there exist two distinct
  simplices $\sigma_1, \sigma_2$ in $N(\cX(G))$ such that
  $K_0 \in \sigma_1 \cap \sigma_2$ and
  $\varphi(\sigma_1) = \varphi(\sigma_2)$. Since
  $|\varphi(\sigma_1)| = |\sigma_1|$ and
  $|\varphi(\sigma_2)| = |\sigma_2|$, it implies that there exist
  $K_1 \in \sigma_1\setminus \sigma_2$ and
  $K_2 \in \sigma_2 \setminus \sigma_1$ such that
  $\varphi(K_1) = \varphi(K_2)$. If $K_1 \cap K_2 \neq \emptyset$,
  since $|\varphi(K_1)| = |K_1|$ and $|\varphi(K_2)| = |K_2|$, by
  Claim~\ref{claim-intersection-cliques-nerve}, we have  $K_1 = K_2$, a
  contradiction. Consequently, $K_1 \cap K_2 = \emptyset$. Consider
  two distinct vertices $u_1 \in K_0 \cap K_1$ and
  $u_2 \in K_0 \cap K_2$. Since $|\varphi(K_0)| = |K_0|$, we have
  $\varphi(u_1) \neq \varphi(u_2)$. Since
  $\varphi(K_2) = \varphi(K_1)$, there exists $v_1 \in K_2$ such that
  $\varphi(v_1) = \varphi(u_1)$. But $u_2 \sim u_1$ since
  $u_1, u_2 \in K_0$ and $u_2 \sim v_1$ since $u_2, v_1 \in K_2$. This
  contradicts the local injectivity of $\varphi$ at $u_2$.

  Consequently, $\varphi$ defines a simplicial map from $N(\cX(G))$ to
  $N(\cX(G'))$ that induces a bijection between the simplices
  containing a vertex of $N(\cX(G))$ and the simplices containing its
  image, i.e., $\varphi$ defines a covering map from $N(\cX(G))$ to
  $N(\cX(G'))$.
\end{proof}

Since a covering map is locally bijective, from Lemma \ref{lem:covering-nerves}
and Proposition \ref{clique_Helly_triangle} we conclude that the nerve
graph $NG(\cX(G))$ is clique-Helly if and only if the nerve graph $NG(\cX(G'))$ is clique-Helly.
This concludes the proof of Theorem~\ref{t:nerve-cover}.

\medskip
Recall that a graph $G$ is \emph{locally 7-systolic} if the
neighborhoods of vertices do not induce $4$-, $5$-, and $6$-cycles.
If additionally, the clique complex $X(G)$ of $G$ is simply connected,
then the graph $G$ is \emph{7-systolic}.  A $7$-systolic graph does
not contain induced $4$-, $5$-, and $6$-cycles~\cite{Chepoi2000,JS}.
It was shown in~\cite{JS} that 7-systolic graphs are hyperbolic, and
in fact, they are 1-hyperbolic~\cite{ChDrEsHaVa}. Thus they are coarse
Helly by Proposition~\ref{coarse-Helly-hyperbolic} or
Proposition~\ref{coarse-Helly-hwm}. We now show that the nerve complex
of the clique-hypergraph of a 7-systolic graph is Helly.

\begin{theorem}\label{7-systolic}
  If $G$ is a locally 7-systolic graph that is locally finite, then
  the nerve graph $NG(\cX(G))$ of its clique-hypergraph $\cX(G)$ is
  clique-Helly.  In particular, if $G$ is a locally finite 7-systolic
  graph, then the nerve graph $NG(\cX(G))$ is Helly.
\end{theorem}

This is a generalization of a result by Larri\'on, Neumann-Lara, and
Piza\~na \cite{Larrion+2002}. Formulated in different terms, the result
of \cite{Larrion+2002} can be rephrased as follows: the nerve
graphs of the clique-hypergraphs of 2-dimensional (i.e., $K_4$-free)
locally 7-systolic complexes are clique-Helly.

Contrary to the usual approach to other local-to-global proofs (such
as Theorem~\ref{t:lotogloHell}), we first prove the second assertion of
Theorem~\ref{7-systolic} and then the first assertion follows from
Corollary~\ref{c:nerve-cover}.

In the proof, we need the following technical lemma. This is a
particular case of the result of~\cite{Chepoi1998}, providing a
characterization of graphs admitting $r$-dominating cliques.  We
present here a much simpler proof of this particular case.

\begin{lemma}\label{lem-VC-simple}
  Given a locally finite $7$--systolic graph $G$, for any finite set
  $S \subseteq V(G)$ of diameter $3$ in $G$, there exists a clique $K$
  dominating $S$ (i.e., $d(u,K) \leq 1$ for any $u \in S$).
\end{lemma}

\begin{proof}
  Consider a maximal clique $K$ of $G$ that maximizes the size of
  $N_S[K] = \{u \in S \mid d(u,K) \leq 1\}$ and assume that there
  exists $v \in S$ such that $d(v,K) >1$. Among all such cliques and
  vertices, consider a clique $K$ and a vertex $v$ that minimizes
  $d(v,K)$.

  \begin{claim}\label{c-clique-1}
    For any clique $K$ and any vertex $v$ such that
    $d(v,u) = d(v,K) = k$ for any $u \in K$, there exists
    $v' \in B(v,k-1)$ such that $v' \sim K$.
  \end{claim}

  \begin{proof}
    Consider $v' \in B(v,k-1)$ that maximizes $|N(v') \cap K|$ and
    assume that there exists $u'' \in K$ such that $v' \nsim
    u''$. Consider $u' \in K \cap N(v')$. By (TC), there exists
    $v''\in B(v,k-1)$ such that $v'' \sim u',u''$. Since
    $v', v'' \in B(v,k-1) \cap N(u')$ and $d(v,u') = k$, we have
    $v' \sim v''$ (otherwise, by (QC), there exists an induced
    square in $G$). For any $u \in N(v') \cap K$, the 4-cycle
    $v'uu''v''$ cannot be induced and thus $u \sim v''$.
    Therefore, $N(v')\cap K \subsetneq N(v'')\cap K$, contradicting the choice
    of $v'$.
  \end{proof}

  We also recall the following well-known property of systolic
  graphs~\cite{SoCh}.

  \begin{claim}\label{claim-INC}
    In a systolic graph $G$, for any $u,v,w,z$ such that $u \sim v,w$
    and $d(u,z) = d(v,z)+1 = d(w,z)+1$, we have $v \sim w$.
  \end{claim}

  \begin{proof}
    By (QC), there exists $x \sim v,w$ such that $d(u,x) =
    d(u,v)-1$. Since $G$ is systolic, the 4-cycle $xvzw$ cannot be
    induced and thus $v \sim w$.
  \end{proof}

  By Claim~\ref{c-clique-1} and our choice of $K$ and $v$, there
  exists $u \in K$ such that $d(v,u) = d(v,K)+1$. We distinguish
  several cases depending on the value of $d(v,K)$. Since the diameter
  of $S$ is $3$ and since $K$ is adjacent to at least one vertex of
  $S$, necessarily $d(v,K) \leq 4$.  If $d(v,K) = 4$, let
  $K_4 = K\cap B(v,4)$. Note that for any $u \in N_S[K]$, since
  $d(v,u) \leq 3$, we have $d(u,K_4) \leq 1$ and thus
  $N_S(K) = N_S(K_4)$. By Claim~\ref{c-clique-1}, there
  exists a clique $K'$ containing $K_4$ such that
  $N_S[K] = N_S[K_4] \subseteq N_S[K']$ and $d(v,K') = 3$,
  contradicting our choice of $K$ and $v$.

  If $d(v,K) = 3$, let $K_3 = \{u \in K: d(v,u) =3 \}$ and
  $K_4 = \{u \in K: d(v,u) =4 \}$. Pick any $u \in N_S(K)$ and
  consider a neighbor $u'$ of $u$ in $K$. We assert that
  $u \in N_S(K_3)$. If $u' \in K_3$, we are trivially done. Otherwise,
  pick any $t \in K_3$ and observe that by Claim~\ref{claim-INC}, we
  have $t \sim u$.
Therefore $N_S[K] = N_S[K_3]$ and by Claim~\ref{c-clique-1}, there
  exists a clique $K'$ containing $K_3$ such that
  $N_S[K] = N_S[K_3] \subseteq N_S[K']$ and $d(v,K') = 2$,
  contradicting our choice of $K$ and $v$.

  Finally, assume that $d(v,K) = 2$ and let
  $K_2 = \{u \in K: d(v,u) =2 \}$ and $K_3 = \{u \in K: d(v,u) =3
  \}$. Let $S_3 = N_S[K_3] \setminus N_S[K_2]$. For any $u \in S_3$,
  there exists $u' \in K_3 \cap N(u)$. If $d(u,v) = 2$, then for any
  $u'' \in K_2$, by Claim~\ref{claim-INC}, $u \sim u''$. Consequently,
  $u \sim K_2$ and thus $u \notin S_3$, a contradiction.
  Consequently, $d(u,v) = 3$, and by (TC), there exists $u'' \in B(v,2)$
  such that $u'' \sim u, u'$.  Since
  $K_2 \cup \{u''\} \subseteq N(u')\cap I(u',v)$, we conclude that
  $u''$ is adjacent to all vertices of $K_2$ by Claim~\ref{claim-INC}.

  \begin{claim} \label{c-S3}
    For any $u,w \in S_3$, either $u'' = w''$ or $u'' \sim w''$.
  \end{claim}

  \begin{proof}
    Suppose that $u'' \neq w''$ and $u'' \nsim w''$. If $w'' \sim u'$,
    we get a contradiction by Claim~\ref{claim-INC}. Similarly, we can
    assume that $u'' \nsim w'$. Let $v'' \in K_2$ and note that
    $v'' \sim u',w'$ and $v'' \nsim u,w$ since $u,w \notin
    N_S[K_2]$. By Claim~\ref{claim-INC}, we have $u''\sim v''$ and
    similarly $v''\sim w''$. By (TC), there exists $w^* \sim v,v'',w''$
    and $u^*\sim v,v'',u''$. If $u^*\sim w''$ (in particular if
    $u^* = w^*$), the vertices $v'', u^*,u'',u',w',w''$ induce a
    $W_5$. We can thus assume that $u^* \nsim w''$ and similarly that
    $w^* \nsim u''$.  By Claim~\ref{claim-INC}, we have $u^* \sim w^*$
    and thus the vertices $v'', u^*,u'',u',w',w'',w^*$ induce a $W_6$.
  \end{proof}

  Consider the clique $K'=K_2 \cup \{u'': u \in S_3\}$ and note that
  $N_S[K] \subseteq N_S[K']$. Since all vertices of $K'$ are at
  distance $2$ from $v$, by Claim~\ref{c-clique-1}, there exists a
  clique $K''$ containing $K'$ such that $d(v,K'') = 1$. Thus
  $N_S[K] \subsetneq N_S[K'']$, contradicting our choice of $K$ and
  $v$.
\end{proof}

We are ready to complete the proof of the second part of Theorem~\ref{7-systolic}.

\begin{lemma}\label{l:7systolic}
  Let $G$ be a locally finite $7$--systolic graph. Then the nerve
  graph $NG(\cX(G))$ of its clique-hypergraph $\cX(G)$ is a Helly
  graph.
\end{lemma}

\begin{proof}
  Since $G$ is locally finite, $NG(\cX(G))$ is also locally
  finite. Since $X(G)$ is simply connected, by Borsuk's Nerve
  Theorem~\cite{Bor,Bj}, $N(\cX(G))$ is simply connected and so is its
  flag-completion $X(NG(\cX(G)))$. Thus, by
  Theorem~\ref{t:lotogloHell}, it suffices to show that the nerve
  graph $NG(\cX(G))$ is finitely clique-Helly.  Consider a finite
  family $\cF = \{C_1, \ldots, C_n\}$ of pairwise intersecting maximum
  cliques in $NG(\cX(G))$. Each $C_i$ corresponds to a family
  $\{K_1^i, \ldots, K_{n_i}^i\}$ of pairwise intersecting cliques in
  $G$.

  Let $V_i = \bigcup_{j=1}^{n_i} V(K_j^i)$ for every $1 \leq i \leq n$
  and let $V_\cF = \bigcup_{i=1}^n V_i$.  First note that
  $\diam(V_\cF) \leq 3$. Indeed, for any $u \in K_j^i$ and
  $u' \in K_{j'}^{i'}$, there exists a clique $K \in C_i \cap C_{i'}$
  and two vertices $u_i \in K \cap K_j^i$ and
  $u_{i'} \in K \cap K_{j'}^{i'}$.  Therefore, by
  Lemma~\ref{lem-VC-simple} there exists a maximal clique $K$ of $G$
  such that $d(v,K) \leq 1$ for all $v \in V_\cF$.

  \begin{claim} \label{KKij}
    $K \cap K^i_j \neq \emptyset$ for all $i,j$.
  \end{claim}

  \begin{proof}
    Suppose that there exists $i, j$ such that
    $K\cap K^i_j = \emptyset$ and pick a vertex $v \in K$
    maximizing $|N(v) \cap K_j^i|$. If $v \sim K_j^i$, then $K_j^i$ is
    not a maximal clique of $G$, a contradiction. Thus, there exists
    $u' \in K^i_j$ such that $v \nsim u'$. Since $d(u',K) = 1$, there
    exists $v' \in K$ such that $v'\sim u'$. For any
    $u \in N(v) \cap K_j^i$, the cycle $uvv'u'$ cannot be induced and
    thus $u \sim v'$. Therefore
    $N(v)\cap K_j^i \subsetneq N(v') \cap K_j^i$, contradicting our
    choice of $v$.
  \end{proof}

  Since $K$ intersects all $K_j^i$, by maximality of $C_i$ in
  $NG(\cX(G))$, we have $K \in C_i$ for all $i$. Consequently,
  $NG(\cX(G))$ is finitely clique-Helly and thus Helly.
\end{proof}

\subsection{Rips complexes and nerve complexes of
  $\delta$-ball-hypergraphs}
\label{s:Rips}

The \emph{Rips complex} (also called the \emph{Vietoris-Rips complex})
$R_{\delta}(M)$ of a metric space $(M,d)$ and positive real $\delta$
is an abstract simplicial complex that has a simplex for every finite
set of points of $M$ that has diameter at most $\delta$. If $(M,d)$ is
a connected unweighted graph $G$, then for any positive real $\delta$,
$R_{\delta}(G)$ and $R_{\lfloor \delta \rfloor}(G)$ coincide. In this
case, we can thus assume that $\delta$ is a positive integer, and then
the Rips complex $R_{\delta}(G)$ is just the $\delta$th power
$G^{\delta}$ of $G$. Notice that for any $\delta \in \N$, the nerve
complex $N(\cB_{\delta}(G))$ of the $\delta$-ball-hypergraph
$\cB_{\delta}(G)$ is isomorphic to the the Rips complex
$R_{2\delta}(G)$.

\begin{lemma}Rips complexes $R_{\delta}(G)$ of a Helly graph $G$ are  Helly.
\end{lemma}

\begin{proof}
  As noted above, we can assume that $\delta$ is an integer and thus
  the Rips complex $R_{\delta}(G)$ coincides with the $\delta$th power
  $G^{\delta}$ of $G$.  Observe that for any vertex $v$ and any radius
  $r$, $B_{r}(v, G^{\delta}) = B_{r\delta}(v,G)$. Thus the result
  follows since the family of balls of $G$ satisfies the Helly
  property.
\end{proof}

\subsection{Face complexes}
\label{s:face}

The \emph{face complex} $\fcom{X}$ of a locally finite abstract
simplicial complex $X$ is the simplicial complex whose vertex set
$V(\fcom{X})$ is the set of non-empty simplices of $X$ and where
$\{F_1,F_2,\ldots,F_k\}$ is a simplex of $\fcom{X}$ if
$\bigcup_{i=1}^k F_i$ is contained in a common simplex $F$ of $X$. If
$X$ is the clique complex of a graph $G$, then the vertices of
$\fcom{X}$ are the cliques of $G$ and two cliques $K_1,K_2$ of $G$ are
adjacent in the $1$--skeleton of $\fcom{X}$ if $K_1 \cup K_2$ is a
clique.

Given a maximal simplex $\sigma = \{F_1,F_2,\ldots,F_k\}$ of $\fcom{X}$,
$\bigcup_{i=1}^k F_i$ is contained in a common simplex $F$ of $X$. By
maximality of $\sigma$, $F = \cup \sigma = \bigcup_{i=1}^k F_i$ and
$F$ is a maximal simplex of $X$. Moreover, $F \in \sigma$ and
consequently, since $\sigma$ is maximal,
$\sigma = \cP(F) \setminus \{\emptyset\}$ where $\cP(F)$ is the set of
all subsets of $F$. Conversely, for any maximal simplex $F$ of $X$, by
definition of $\fcom{X}$, $\sigma = \cP(F) \setminus \{\emptyset\}$ is a
simplex of $\fcom{X}$. Since $F$ is a maximal simplex of $X$, $\sigma$
must be a maximal simplex of $\fcom{X}$. As a result, we obtain the
following Lemma:

\begin{lemma}\label{l:triface}
  For any simplicial complex $X$, the map $\sigma \mapsto \cup \sigma$
  defines a bijection from the set of maximal simplices of $\fcom{X}$ to
  the set of maximal simplices of $X$, with inverse given by
  $F \mapsto \cP(F) \setminus \{\emptyset\}$.
\end{lemma}

If we start with the clique complex of a graph, then its face complex is
also the clique complex of a graph.

\begin{lemma}\label{l:trifaceclique}
  For any clique complex $X$, its face complex $\fcom{X}$ is also a
  clique complex.
\end{lemma}

\begin{proof}
  Let $G = G(X)$ be the $1$--skeleton of $X$ and let
  $G' = G(\fcom{X})$ be the $1$--skeleton of $\fcom{X}$. For any edge
  $F_1F_2$ in $G'$, $F_1$, $F_2$, and $F_1\cup F_2$ are cliques of
  $G$. Consequently, for any clique $\sigma=\{F_1, F_2, \ldots, F_k\}$
  in $G(\fcom{X})$, $F_1 \cup F_2 \cup \ldots \cup F_k$ is a clique of
  $G$. Since $X$ is the clique complex of $G$,
  $F_1 \cup F_2 \cup \ldots \cup F_k$ is a clique of $X$ and thus
  $\sigma$ is a simplex of $\fcom{X}$.
\end{proof}

\begin{proposition}\label{p:Hellyface}
  The face complex $\fcom{X}$ of a locally finite clique-Helly
  (respectively, Helly) complex $X$ is a locally finite clique-Helly
  (respectively, Helly) complex.
\end{proposition}

\begin{proof}
  By Lemma~\ref{l:trifaceclique}, $\fcom{X}$ is a clique complex.  Let
  $G = G(X)$ be the $1$--skeleton of $X$ and $G' = G(\fcom{X})$ be the
  $1$--skeleton of $\fcom{X}$. Since $G$ is locally finite, $G'$ is
  also locally finite and thus $\fcom{X}$ is a locally finite
  simplicial complex.

  Consider the bijection $\sigma \mapsto \cup \sigma$ between the
  maximal cliques of $\fcom{X}$ and the maximal cliques of $X$ defined
  in Lemma~\ref{l:triface}.  Observe that if $(\sigma_i)_{i \in I}$ is
  a family of maximal cliques of $F(X)$, then
  $\bigcap_{i\in I} \sigma_i \neq \emptyset$ if and only if
  $\bigcap_{i\in I} (\cup \sigma_i) \neq \emptyset$. Consequently,
  since $X$ is clique-Helly, $F(X)$ is also clique-Helly.

  Suppose now that $X$ is simply connected. Since the nerve complexes
  of the clique hypergraphs of $X$ and $\fcom{X}$ are isomorphic
  thanks to the bijection $\sigma \mapsto \cup \sigma$, $X$ and $F(X)$
  are homotopy equivalent by Borsuk's Nerve
  Theorem~\cite{Bor,Bj}. Consequently, $F(X)$ is simply connected and
  thus by Theorem~\ref{t:lotogloHell}, $F(X)$ is a Helly complex when
  $X$ is a Helly complex.
\end{proof}

 \section{Helly groups} \label{Helly_groups}

As we already defined above, a group is \emph{Helly} if it acts
geometrically on a Helly graph (necessarily, locally finite). The main
goal of this section is to provide examples of Helly groups. More
precisely, in this section we prove Theorems~\ref{t:classes},
\ref{t:prop-coarse-Helly-groups-Helly}, \ref{t:operations}, and
\ref{thm:HellyGeneral_int} from the Introduction, some of their
consequences, and related results.

\subsection{Proving Hellyness of a group}
To prove that a group $\Gamma$ (geometrically) acting on a cell
complex $X$ (or on its 1-skeleton $G(X)$) is Helly, we will derive
from $X$ a Helly complex $X^*$ and prove that $\Gamma$ acts
geometrically on $X^*$. The natural (and most canonical) way would be
to take as $X^*$ the Hellyfication $\He(X)$ of $X$.  By
Theorem~\ref{t:dinjhull}, $\He(X)$ is well-defined and Helly for all
complexes $X$. The group $\Gamma$ acts on $\He(X)$, but the group
action is not always geometrical. However, using the results from
Sections~\ref{injhull+hellyfication} and~\ref{sec-coarse}, and a
result of Lang~\cite{Lang2013}, we will prove that hyperbolic groups
acts geometrically on the Hellyfication of their Cayley graphs that are hyperbolic and thus hyperbolic groups are Helly.

In several other cases, there are more direct ways to derive $X^*$. In
case of \catz cubical groups, based on Proposition \ref{p:medHel} and
the bijection between median graphs and 1-skeletons of \catz cube
complexes \cite{Chepoi2000,Roller1998}, it follows that thickenings
along cubes of locally finite \catz cube complexes are Helly, thus \catz cubical
groups are Helly. By Proposition~\ref{swm-hypercellular}, the
thickenings of locally finite hypercellular complexes and of locally finite swm-complexes are
Helly. Consequently, groups acting geometrically on hypercellular
graphs or swm-graphs are Helly.  We use the same technique by
thickening (along cells) to show that classical \cftf small-cancelation and
graphical \cftf small-cancelation groups are Helly. In all these cases, the
maximal cliques of the thickenings correspond to cells of the original
complex. This allows us to establish that the group $\Gamma$ acts
geometrically on the thickening. Proposition~\ref{graph-cell-complex}
may be useful to establish similar results for groups acting
geometrically on other abstract cell complexes.  Another method is to
prove the Hellyness of the nerve complex (the clique complex of the
intersection graph of maximal cliques of $X$) $N(X)$ on which $\Gamma$ acts
geometrically. In this way, we establish that 7-systolic groups are
Helly (this also follows from the fact that 7-systolic groups are
hyperbolic).

By considering face complexes, we show that Helly groups are stable by
free products with amalgamation over finite subgroups and by quotients
by finite normal subgroups.
Using the theory of quasi-median groups of~\cite{Qm}, we provide
criteria allowing to construct Helly groups from groups acting on
quasi-median graphs. This allows us to show that Helly groups are
stable by taking graph products of groups, $\square$-products,
$\rtimes$-powers, and $\Join$-products. We also show that the
fundamental groups of right-angled graphs of Helly groups are Helly.

\subsection{\catz cubical, hypercellular,  and swm-groups via thickening}
A group $\Gamma$ is called \emph{cubical} if $\Gamma$ acts
geometrically on a median graph $G$ (or on the \catz cube complex of
$G$).  A group $\Gamma$ is called an \emph{swm-group} if it acts
geometrically on an swm-graph $G$ (or on the orthoscheme complex of
$G$). A group $\Gamma$ is called \emph{hypercellular} if it acts
geometrically on a hypercellular graph $G$ (or on the geometric
realization of $G$).

Any group $\Gamma$ acting geometrically on a median graph, swm-graph, or hypercellular graph $G$
also acts geometrically on its thickening $G^{\Delta}$. From Propositions \ref{p:medHel} and \ref{swm-hypercellular} we obtain:

\begin{proposition}\label{p:ccswmhc}
  Cubical groups, swm-groups, and hypercellular groups are Helly.

  More generally, any group acting geometrically on a simply connected
  abstract cell complex $X$ defined on a graph $G$ satisfying the
  conditions of Proposition~\ref{graph-cell-complex} is Helly.
\end{proposition}

In \cite{CCHO}, with every building $\Delta$ of type $C_n$ we
associated an swm-graph $H(\Delta)$ in such a way that any (proper or
geometric) type-preserving group action on $\Delta$ induces a (proper
or geometric) action on $H(\Delta)$.

\begin{corollary}
	\label{c:buildC}
	Uniform type-preserving lattices in isometry groups of buildings of type $C_n$ are Helly.
\end{corollary}

\subsection{Hyperbolic and quadric groups via Hellyfication}
\label{s:hypquad}

If a group $\Gamma$ acts geometrically on a graph $G$, it also acts on
its Hellyfication $\He(G)=E^0(G)$ and on its injective hull
$E(G)$. However in general, this action is no longer geometric. This
is because the injective hull $E(G)$ is not necessarily proper and
because the points of $E(G)$ may be arbitrarily far from $e(G)$.  This
does not happen if $G$ is a Helly graph:

\begin{theorem}\label{t:helly=inj}
  Let $G$ be a locally finite Helly graph.
  \begin{enumerate}[(1)]
  \item The injective hull $E(G)$ of $G$ is proper and has the structure of a
    locally finite polyhedral complex with only finitely many isometry
    types of $n$-cells, isometric to injective polytopes in
    $\ell_{\infty}^n$, for every $n\geq 1$. Moreover,
    $d_H(E(G),e(G))\leq 1$.  Furthermore, if $G$ has uniformly
    bounded degrees, then $E(G)$ has finite combinatorial dimension.

  \item A group acting cocompactly, properly or geometrically on $G$
    acts, respectively, cocompactly, properly or geometrically on its
    injective hull $E(G)$.
  \end{enumerate}
\end{theorem}

For $\beta \geq 1$, the graph $G$ has \emph{$\beta$-stable
  intervals}~\cite{Lang2013} if for every triplet of vertices $w,v,v'$
with $v\sim v'$, we have $d_H(I(w,v),I(w,v'))\leq \beta$, where $d_H$
denotes the Hausdorff distance. The proof of the first assertion of
Theorem~\ref{t:helly=inj}(1) is based on the following theorem of Lang~\cite{Lang2013}:

\begin{theorem}[\cite{Lang2013}*{Theorem~1.1}]\label{t:langstable}
  Let $G$ be a locally finite graph with $\beta$-stable
  intervals. Then the injective hull of $G$  is
  proper (that is, bounded closed subsets are compact) and has the
  structure of a locally finite polyhedral complex with only finitely
  many isometry types of $n$-cells, isometric to injective polytopes
  in $\ell_{\infty}^n$, for every $n\geq 1$.
\end{theorem}

Next we show that weakly modular graphs (and thus Helly graphs) have $\beta$-stable
intervals.

\begin{lemma}\label{l:wmstable}
  Every weakly modular graph has $1$-stable intervals.
\end{lemma}

\begin{proof}
  We need to show that for every triplet of vertices $w,v,v'$ with
  $d(v, v') \leq 1$, and every vertex $u\in I(w,v)$ there exists a
  vertex $u'\in I(w,v')$ with $d(u,u')\leq 1$. If $v = v'$, we are
  done by taking $u' = u$. Suppose now that $v \sim v'$.  We proceed
  by induction on $k=d(w,v)+d(w,v')$. For $k=0$ the statement is
  obvious. Assume now that the statement holds for any $j < k$ and
  that $d(w,v)+d(w,v')=k$. If $d(w,v') = d(w,v)+1$, then
  $I(w,v) \subseteq I(w,v')$ and the statement obviously holds. If
  $d(w,v)=d(w,v')$ then, by the triangle condition (TC) (see
  Subsection~\ref{s:bagraphs}) there exists a vertex $v^*\sim v,v'$
  such that $v^* \in I(w,v) \cap I(w,v')$. Since
  $d(w,v) + d(w,v^*) = d(w,v) +d(w,v')-1 = k-1$, by the induction
  hypothesis, for any $u \in I(w,v)$, there exists
  $u' \in I(w,v^*) \subseteq I(w,v')$ such that $d(u,u') \leq
  1$. Suppose now that $d(w,v') = d(w,v) -1$, i.e., $v' \in
  I(w,v)$. For any $u \in I(w,v)$, let $u^* \in N(v) \cap I(u,v)$. By
  the quadrangle condition, there exists $v^*$ such that
  $v^* \sim v',u^*$ and $v^* \in I(w,v') \cap I(w,u^*)$. Since
  $d(w,u^*) +d(w,v^*) = k-2$ and since $u \in I(w,u^*)$, by the induction
  hypothesis, there exists $u'$ such that $d(u,u') \leq 1$ and
  $u' \in I(w,v^*) \subseteq I(w,v')$.
\end{proof}

To establish the second assertion of  Theorem~\ref{t:helly=inj}(1), we use Lang's results  relating the
combinatorial dimension with the notion of cones.  In a graph $G$, the \emph{cone}~\cite{Lang2013}
determined by the directed pair $(x,v)$ of vertices of $G$ is the set $C(x,v) = \{y \in V(G): v \in
I(x,y)\}$. Given a vertex $v \in V(G)$, we denote by $\cC(v)$ the set
of all cones $C(x,v)$ for $x \in V(G)$. For a ball $B$ of $G$, we
denote by $\cC(B)$ the set of all pointed cones $(v,C(x,v))$ with
$v \in B$ and $x \in V(G)$. By~\cite[Lemma~5.8]{Lang2013}, the size of
$\cC(B)$ is finite and bounded by a function of the size of $B$.

\begin{proposition}[{\cite[Proposition~5.12]{Lang2013}}]\label{prop-Lang-rang}
  Let $G$ be a locally finite graph with $\beta$-stable
  intervals. Given a vertex $z \in V(G)$ and $\alpha > 0$, let $B$ be
  the ball $B_{2\alpha\beta}(z)$. Then for every $f \in E'(G)$ such
  that $f(z) \leq \alpha$, we have $\rk(A(f)) \leq \frac{1}{2} |\cC(B)|$.
\end{proposition}

\begin{proof}[Proof of Theorem~\ref{t:helly=inj}(1)]
  Properness and the structure of a locally finite polyhedral complex
  follow from Theorem~\ref{t:langstable} and Lemma~\ref{l:wmstable}.

  We now show that $d_H(E(G),e(G))\leq 1$.  Pick any $f \in E(G)$ and
  consider $f' \in \Delta^0(G)$ defined by setting
  $f'(x) = \lceil f(x) \rceil$ for any $x \in V(G)$. Let
  $f'' \in E^0(G)$ such that $f'' \leq f'$ and notice that for any
  $x \in V(G)$, we have $f''(x) \leq f'(x) < f(x) +1$. On the other hand, for
  any $x \in V(G)$, by Claim~\ref{tight}, for any $\epsilon > 0$,
  there exists $y \in V(G)$ such that
  $f(x) + f(y) < d(x,y) + \epsilon \leq f''(x) + f''(y) +\epsilon \leq
  f''(x) + f(y) +1 + \epsilon$. Consequently,
  $f(x) < f''(x) +1 + \epsilon$ for any $\epsilon > 0$ and thus,
  $f(x) \leq f''(x) +1$.  Since $G$ is a Helly graph, by
  Theorem~\ref{t:dinjhull}, $E^0(G)$ and $G$ coincide and thus, there
  exists a vertex $z \in V(G)$ such that $f'' = d_z$, establishing
  that $d_\infty(f,d_z) \leq 1$.

  Now, additionally suppose that $G$ has uniformly bounded degrees. To show that
  $E(G)$ is finite dimensional,  pick any $f \in E'(G)$ and consider the vertex
  $z \in V(G)$ such that $f(z) = d_\infty(f,d_z) \leq 1$. By
  Lemma~\ref{l:wmstable}, $G$ has $1$-stable intervals, and by
  Proposition~\ref{prop-Lang-rang} applied with $\alpha = \beta = 1$,
  we have that $\rk(A(f)) \leq \frac{1}{2}|\cC(B_2(z))|$. Since $G$
  has bounded degrees, the size of the balls of radius 2 in $G$ is
  also bounded and by~\cite[Lemma~5.8]{Lang2013}, the size of
  $|\cC(B_2(z))|$ is uniformly bounded by some constant
  $K$. Consequently, by Proposition \ref{prop-Lang-rang} all cells of $E'(X)$ are of dimension at most
  $\frac{1}{2}K$. By~\cite[Theorem~4.5]{Lang2013}, $E'(G)=E(G)$. This prove that
  $E(G)$ has finite combinatorial dimension.
\end{proof}

Theorem~\ref{t:helly=inj}(2) is an immediate corollary of
Theorem~\ref{t:helly=inj}(1) and of the next proposition.

\begin{proposition}\label{prop-proper-inj-hull-Helly}
  Let $G$ be a locally finite graph such that the injective hull
  $E(G)$ is proper and satisfies the bounded distance property and let
  $\Gamma$ be a group acting on $G$.
  \begin{enumerate}
  \item if $\Gamma$ acts cocompactly on $G$, then $\Gamma$ acts cocompactly
    on $E(G)$ and $E^0(G)$;
  \item if $\Gamma$ acts properly on $G$, then $\Gamma$ acts properly
    on $E(G)$ and $E^0(G)$;
  \item if $\Gamma$ acts geometrically on $G$, then $\Gamma$ acts
    geometrically on $E(G)$ and $E^0(G)$ and thus $\Gamma$ is a Helly
    group.
  \end{enumerate}
\end{proposition}

\begin{proof}
  Consider the Helly graph $E^0(G)$.  Since the set $E^0(G)$ is an
  integer-valued subspace of $E(G)$ and $E(G)$ is proper, the balls of
  $E^0(G)$ are compact. Therefore, the graph $E^0(G)$ is a proper
  metric space and thus is locally finite. In particular, all compact
  sets of $E^0(G)$ are finite. Since $E(G)$ satisfies the bounded
  distance property, there exists $\delta$ such that for each
  $f \in E(G)$, we have $d_{\infty}(f,e(G)) \leq \delta$.

  We first assume that $\Gamma$ acts cocompactly on $G$ and we show
  that $\Gamma$ acts cocompactly on $E^0(G)$ and $E(G)$. The proof is
  the same in both cases; we provide it for $E^0(G)$.  Since $\Gamma$ acts
  cocompactly on $G$, there exists $v \in V(G)$ and $r \in \N$ such
  that $V(G) = \bigcup_{g\in \Gamma} V(B_r(gv,G))$. Let
  $R = r + \delta$ and consider
  $\bigcup_{g\in \Gamma} V(B_R(ge(v),E^0(G)))$. For any
  $f \in E^0(G)$, there exists $v' \in V(G)$ such that
  $d_{\infty}(f,e(v')) \leq \delta$. Since there exists $g \in \Gamma$
  such that $d_G(v',gv) \leq r$,
  $d_{\infty}(f,ge(v)) = d_{\infty}(f,e(gv)) \leq d_{\infty}(f,e(v'))
  + d_{\infty}(e(v'),e(gv)) \leq \delta + d_G(v',gv) \leq \delta
  +r$. This shows that
  $E^0(G) = \bigcup_{g\in \Gamma} V(B_R(ge(v),E^0(G)))$ and thus
  $\Gamma$ acts cocompactly on $E^0(G)$.

  We now assume that $\Gamma$ acts properly on $G$ and we show that
  $\Gamma$ acts properly on $E^0(G)$ and $E(G)$.
  Consider a compact set $K$ in $E^0(G)$ or $E(G)$ and let
  $K' = \{v \in V(G): \exists f \in K, d_{\infty}(f,e(v)) \leq
  \delta\}$.  Since $K'$ is a bounded subset of $V(G)$, $K'$ is finite
  and thus $e(K')$ is also finite.  Pick any $g \in \Gamma$ such that
  $\bar{g}K \cap K \neq \emptyset$ (where $\bar{g}$ is the inverse of
  $g$ in $\Gamma$) and some $f \in K$ such that $\bar{g}f \in K$. Let
  $v \in K'$ such that $d_{\infty}(f,e(v)) \leq \delta$. Since
  $\Gamma$ acts on  $E^0(G)$ and $E(G)$,
  $d_{\infty}(\bar{g}f,\bar{g}e(v)) = d_{\infty}(f,e(v)) \leq
  \delta$. Since $\bar{g}e(v) = e(\bar{g}v)$, $\bar{g}v \in K'$ and
  thus $v \in K'\cap gK'$. This shows that
  $\{g \in \Gamma: \bar{g}K \cap K \neq \emptyset\} \subseteq \{g \in
  \Gamma: {g}K' \cap K' \neq \emptyset\}$. Since $\Gamma$ acts
  properly on $G$, the second set is finite and thus $\Gamma$ acts
  properly on $E^0(G)$ and $E(G)$.

  Finally, if $\Gamma$ acts properly and cocompactly on $E^0(G)$, since
  $E^0(G)$ is a Helly graph, $\Gamma$ is a Helly group.
\end{proof}

If we consider a group $\Gamma$ acting on a coarse Helly graph $G$,
then $\Gamma$ is a Helly group provided that $G$ has
$\beta$-stable intervals:

\begin{proposition}\label{prop-coarse-Helly-groups-Helly}
  A group acting geometrically on a coarse Helly graph with
  $\beta$-stable intervals is Helly.
\end{proposition}

This result is a particular case of Proposition~\ref{coarse},
Theorem~\ref{t:langstable}, and Proposition \ref{prop-proper-inj-hull-Helly}.

From Propositions~\ref{coarse-Helly-hyperbolic}
and~\ref{prop-coarse-Helly-groups-Helly}, we also get the following
corollary.

\begin{corollary}\label{cor-hyperbolic-groups-Helly}
  Hyperbolic groups are Helly.
\end{corollary}

\begin{proof}
  By Proposition~\ref{coarse-Helly-hyperbolic}, any
  $\delta$-hyperbolic graph $G$ is coarse Helly with constant
  $2\delta$. Moreover, if $G$ has $\delta$-thin geodesic triangles,
  then one can easily check that $G$ has $(\delta+1)$-stable
  intervals. The result then follows from
  Proposition~\ref{prop-coarse-Helly-groups-Helly}.
\end{proof}

A group $\Gamma$ is \emph{quadric} if it acts geometrically on a quadric
complex~\cite{HodaQuadric}. \emph{Quadric complexes} are cell complexes that
have hereditary modular graphs as $1$-skeletons.

\begin{corollary}
  Quadric groups are Helly.
\end{corollary}

\begin{proof}
  Since hereditary modular graphs are weakly modular, they have
  $1$-stable intervals by Lemma~\ref{l:wmstable} and they are
  coarse Helly by Proposition~\ref{coarse-Helly-hwm}.
\end{proof}

By~\cite[Theorem~B]{HodaQuadric}, any group admitting a finite
\cftf presentation acts geometrically on a quadric complex,
leading thus to the following corollary:

\begin{corollary}
  Any group admitting a finite \cftf presentation is Helly.
\end{corollary}

\subsection{7-Systolic groups via nerve graphs of clique-hypergraphs}
\label{s:7sys}

A group $\Gamma$ is called \emph{systolic} (respectively,
\emph{7-systolic}) if $\Gamma$ acts geometrically on a systolic
(respectively, \emph{7-systolic}) graph (or complex). Since
$7$-systolic groups are hyperbolic~\cite{JS}, they are Helly by
Corollary~\ref{cor-hyperbolic-groups-Helly}.  Since 7-systolic
graphs are coarse Helly, by
Proposition~\ref{prop-coarse-Helly-groups-Helly}, each 7-systolic
group acts geometrically on the Hellyfication $\He(G)$ of a 7-systolic
graph $G$.

Any group $\Gamma$ acting geometrically on a graph $G$
also acts geometrically on the nerve graph $NG(\cX(G))$ of its
clique-hypergraph $\cX(G)$. Since the nerve graph $NG(\cX(G))$ of the
clique-hypergraph $\cX(G)$ of a $7$-systolic graph is Helly by
Theorem~\ref{7-systolic}, a group $\Gamma$ geometrically acting on a
$7$-systolic graph $G$ act also geometrically on the Helly graph
$NG(\cX(G))$ and is thus Helly.

\begin{proposition}
  If a group $\Gamma$ acts geometrically on a $7$-systolic graph $G$,
  then $\Gamma$ acts geometrically on the Helly graphs $\He(G)$ and
  $NG(\cX(G))$, i.e., $7$-systolic groups are Helly.
\end{proposition}

\subsection{\cftf graphical small cancellation groups via thickening}
\label{s:c4t4}
The main goal of this subsection is to prove that finitely presented
graphical \cftf small cancellation groups are Helly. Our exposition
follows closely \cite[Section 6]{OsajdaPrytula}, where graphical
$C(6)$ groups were studied. We begin with general notions concerning
complexes, then graphical \cftf complexes, and proving the Helly
property for a class of graphical \cftf complexes. From this we
conclude the Hellyness of the corresponding groups.

\medskip

In this subsection, unless otherwise stated, all complexes are $2$--dimensional CW--complexes  with combinatorial attaching maps  (that is, restriction to an open cell is a homeomorphism onto an open cell) being immersions -- see \cite[Section 6]{OsajdaPrytula} for details.
 A \emph{polygon} is a $2$--disk with the cell structure that consists of $n$ vertices, $n$ edges, and a single $2$--cell. For any $2$--cell $C$ of a $2$--complex $X$ there exists a map $R \to X$, where $R$ is a polygon and the attaching map for $C$ factors as $S^{1} \to \partial R \to X$. In the remainder of this section by a \emph{cell} we will mean a map $R \to X$ where $R$ is a polygon. An \emph{open cell} is the image in $ X$ of the single $2$--cell of $R$.
  A \emph{path} in $X$ is a combinatorial
 map $P \to X$ where $P$ is either a subdivision of the interval or a single vertex. In the latter case we call $P\to X$ a \emph{trivial} path. The \emph{interior} of the path is the path minus its endpoints. 
 Given paths $P_1\to X$ and $P_2 \to X$ such that the terminal point of $P_1$ is equal to the initial point of $P_2$, their \emph{concatenation} is an obvious path $P_1P_2 \to X$ whose domain is the union of $P_1$ and $P_2$ along these points. 
A \emph{cycle} is a map $C \to X$, where $C$ is a subdivision of the circle $S^1$. The cycle $C\to X$ is \emph{non-trivial} if it does not factor through a map to a tree.
 A path or cycle is \emph{simple} if it is injective on vertices. Notice that a simple cycle (of length at least $3$) is non-trivial.
The \emph{length} of a path $P$ or a cycle $C$ denoted by $|P|$ or $|C|$ respectively is the number of $1$--cells in the domain.
A subpath $Q \to X$ of a path $P \to X$ (or a cycle) is a path that factors as $Q \to P \to X$ such that $Q \to P$ is an injective map. Notice that the length of a subpath does not exceed the length of the path.

 A \emph{disk diagram} is a contractible finite $2$--complex $D$ with a specified embedding into the plane. We call $D$ \emph{nonsingular} if it is homeomorphic to the $2$--disc, otherwise $D$ is called \emph{singular}. The \emph{area} of $D$ is the number of $2$--cells.
The boundary cycle $\partial D$ is the attaching map of the $2$--cell that contains the point $\{\infty\}$, when we regard $S^2= \mathbb{R}^2 \cup \{\infty\}$.
A \emph{boundary path} is any path $P\to D$ that factors as $P\to \partial D \to D$. An \emph{interior path} is a path such that none of its vertices, except for possibly endpoints, lie on the boundary of $D$.
 If $X$ is a $2$--complex, then a \emph{disk diagram in} $X$ is a  map $D \to X$.

A \emph{piece} in a disk diagram $D$ is a path $P \to D$ for which
 there exist two different lifts to $2$--cells of $D$, i.e., there are
 $2$--cells $R_i \to D$ and $R_j \to D$ such that $P \to D$ factors
 both as $P \to R_i \to D$ and $P \to R_j \to D$, but there does not
 exist an isomorphism $R_j \to R_i$ making the following diagram commutative:
 \[
   \begindc{\commdiag}[20]
   \obj(12,1)[a]{$R_j$}
   \obj(35,1)[b]{$D$}
   \obj(35,1)[b']{}
   \obj(35,16)[c]{$R_i$}
   \obj(35,17)[c']{}
   \obj(12,17)[d]{$P$}
   \obj(12,17)[d']{}
   \mor{a}{b'}{}
   \mor{c}{b}{}
   \mor{a}{c}{}
   \mor{d'}{c'}{}
   \mor{d'}{a}{}
   \enddc
 \]

 Let $\varphi \colon \bfG \to \Theta$ be an immersion of graphs, assume that $\Theta$
 is connected and that $\bfG$ does not have vertices of degree
 $0$ or $1$. For convenience we will write $\bfG$ as the union of its
 connected components $\bfG=\bigsqcup_{i \in I}G_i$, and refer to
 the connected graphs $G_i$ as \emph{relators}.  
 A \emph{thickened graphical complex} $X$ is a $2$--complex with
 $1$--skeleton $\Theta$ and a $2$--cell attached along every immersed
 cycle in $\bfG$, i.e., if a cycle $C \to \bfG$ is immersed, then in
 $X$ there is a $2$--cell attached along the composition
 $C \to \bfG \to \Theta$.  A (nonthickened) \emph{graphical complex}
 $X^{\ast}$ is a $2$--complex obtained by gluing a simplicial cone $C(G_i)$ along each
 $G_i \to \Theta$:
 \[X^{\ast}= \Theta \cup_{\varphi} \bigsqcup_{i\in I} C(G_i).\] 

 For any $G_i \to X$ we have a \emph{thick cell}
 $\thi{G_i} \to X$, where  $\thi{G_i}$ is formed by gluing $2$--cells along all immersed
 cycles in $G_i$. In $X^{\ast}$ a \emph{cone-cell} is the
 corresponding map $C(G_i) \to X$. Note that the two complexes $X$ and
 $X^{\ast}$ have the same fundamental groups. To be consistent with
 the approach in \cite{OsajdaPrytula} in the following material we
 work usually with the thickened complex $X$, however the results
 could be formulated also for $X^{\ast}$.

Let $X$ be a thickened graphical complex. A \emph{piece} in $X$ is a
 path $P \to X$ for which there exist two different lifts to $\bfG$,
 i.e., there are two relators $G_i$ and $G_j$ such that the path
 $P \to X$ factors as $P \to G_i \to X$ and $P \to G_j \to X$, but
 there does not exist an isomorphism $\thi{G_j} \to \thi{G_i}$ such that the
 following diagram commutes:	
 \[  \begindc{\commdiag}[30]
   \obj(12,1)[a]{$\thi{G_j}$}
   \obj(35,1)[b]{$X$}
   \obj(35,1)[b']{}
   \obj(35,17)[c]{$\thi{G_i}$}
   \obj(35,17)[c']{}
   \obj(12,17)[d]{$P$}
   \obj(12,17)[d']{}
   \mor{a}{b'}{}
   \mor{c}{b}{}
   \mor{a}{c}{}
   \mor{d'}{c}{}
   \mor{d'}{a}{}
   \enddc  \] 	
A disk diagram $D \to X$ is \emph{reduced} if for every piece  $P \to D$ the composition  $P \to D \to X$ is a piece in $X$.

 \begin{lemma}[Lyndon-van Kampen Lemma]\label{l:vank}
 	Let $X$ be a thickened graphical complex and let $C \to X$ be a closed homotopically trivial path. Then
 	\begin{enumerate} 
 		
 		\item \label{l:vk1} there exists a disk diagram $D\to X$ such that the path $C$ factors as $C \to \partial D \to X$, and $C \to \partial D$ is an isomorphism,
 		
 		\item \label{l:vk2} if a diagram $D \to X$ is not reduced, then there exists a diagram $D_1 \to X$ with smaller area and the same boundary cycle in the sense that there is a commutative diagram:
 		
 		 	$$  \begindc{\commdiag}[30]
 		\obj(35,1)[b]{$X$}
 		\obj(35,1)[b']{}
 		\obj(35,17)[c]{$\partial D$}
 		\obj(35,17)[c']{}
 		\obj(12,17)[d]{$\partial D_1$}
 		\obj(12,17)[d']{}
 		\mor{c'}{b}{}
 		\mor{d'}{c}{$\cong$}
 		\mor{d'}{b}{}
 		\enddc  $$

 		\item \label{l:vk3} any minimal area diagram $D \to X$ such that $C$ factors as $C \xrightarrow{\cong} \partial D \to X$ is reduced.
 	\end{enumerate}
 \end{lemma}

\begin{definition}\label{de:smallcancgraph} We say that a thickened graphical complex $X$ satisfies:
\begin{itemize}
\item 
the \emph{C$(4)$ condition} if no immersed cycle $C \to X$ that factors as $C \to G_i \to X$
is the concatenation of less than $4$ pieces;

\item 
the \emph{T$(4)$ condition} if there does not exist a reduced nonsingular disk diagram $D\to X$ with $D$ containing an internal $0$-cell $v$, of valence $3$, that is, contained in exactly $3$ corners of $2$--cells.

\end{itemize}

If $X$ satisfies both conditions we call it a \emph{\cftf thickened graphical complex}. The corresponding complex  $X^{\ast}$  is called then a  \emph{\cftf graphical complex}.
\end{definition}

If $D$ is a disk diagram we define small cancellation conditions in a
very similar way, except that a \emph{piece} is understood as a piece
in a disk diagram.

\begin{proposition}\label{reduceddiagram}
  If $X$ is a \cftf thickened
  graphical complex and $D \to X$ is a reduced disk diagram, then $D$
  is a \cftf diagram.
\end{proposition}

\begin{proof}
  The assertion follows immediately from the definitions of a reduced
  map and a piece.
\end{proof}

The following lemma is a graphical \cftf analogue of \cite[Theorem 6.10]{OsajdaPrytula} (the graphical $C(6)$ case) and \cite[Propositions 3.4, 3.5, 3.7 and Corollary 3.6]{HodaQuadric} (the classical \cftf case). 

\begin{lemma}\label{lem:intersections}
  Let $X$ be a simply connected \cftf thickened graphical complex.
Then the following hold:
  \begin{enumerate}[(1)]
  \item \label{1intersect} For every relator $G_i$, the map
    $G_i \to X$ is an embedding.
  \item \label{2intersect} The intersection of (the images of) any two
    relators is either empty or it is a finite tree.
  \item \label{3intersect} If three relators pairwise intersect then
    they triply intersect and the intersection is a finite tree.		
  \end{enumerate}
\end{lemma}

\begin{proof}
The proofs of all the items (1), (2), (3) follow the same lines: we assume the statement does not hold and we show that this leads to a forbidden reduced disk diagram, hence reaching a contradiction.

(1) Suppose there is a relator $G_1$ that does not embed. Let $v,v'$ be two vertices of $G_1$ mapped to
a common vertex $v_{11}$ in $X$, and let $\gamma$ be a geodesic path in $G_1$ between $v$ and $v'$. The path $\gamma$ is mapped to a loop $\gamma_{1}$ in $X$. By simple connectedness and by Lemma~\ref{l:vank}
there exists a reduced disk diagram $D$ for $\gamma_{1}$, see Figure~\ref{f:c4t4A} left. We may assume that we choose a counterexample so
that the area (the number of $2$-cells) of $D$ is minimal among all counterexamples.
	
\begin{figure}[h]
  \centering
  \includegraphics[width=0.84\textwidth]{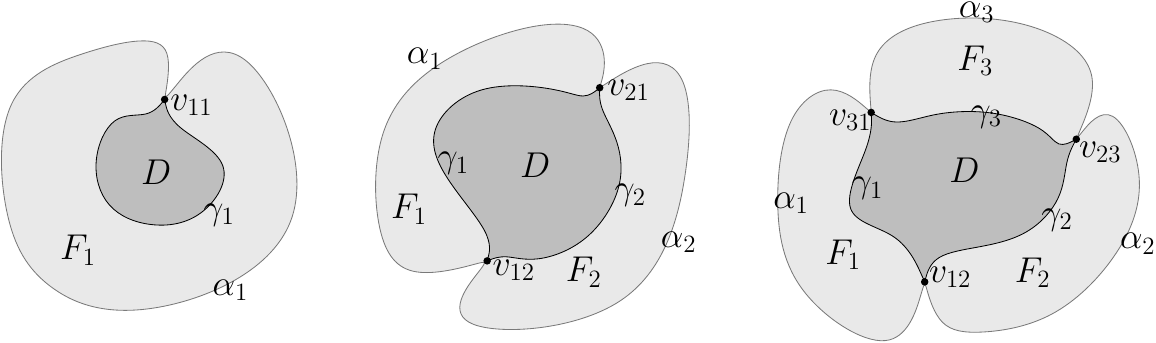}\caption{The proof of Lemma~\ref{lem:intersections}. From left to right: (1), (2), (3).}\label{f:c4t4A}
\end{figure}
	
Now, consider a larger disk diagram $D\cup F_1$ where $F_1$ is a cell
whose boundary is the concatenation $\gamma_1\alpha_1$ which is mapped
to a loop in $G_1$, and the only common point of $\gamma_1$ and
$\alpha_1$ is $v_{11}$, see Figure~\ref{f:c4t4A} left. The existence
of such cell $F_1$ follows from our assumptions on no degree-one
vertices in relators. The diagram $D\cup F_1$ cannot be reduced, since
otherwise it would be a \cftf diagram by
Proposition~\ref{reduceddiagram}, and this would contradict e.g.\
\cite[Proposition 3.4]{HodaQuadric}. Hence, by the definition of a
reduced diagram, there is a piece $P$ in $D\cup F_1$ that does not
lift to a piece in $X$.  Since $D$ is reduced, it follows that the
piece $P$ has to lie on $\gamma_1$. Since $P$ does not lift to a piece
in $X$, $P$ is a part the boundary of a cell $F'$ such that its other
boundary part $Q$ maps to $G_1$ as well, see Figure~\ref{f:c4t4C}. Thus
replacing the subpath $P$ of $\gamma_1$ by $Q$ and, if necessary, reducing the resulting
loop to get an immersed one  we get a new
counterexample with a diagram $D'$, such that $D=D'\cup F'$, of
smaller area --- contradiction proving (1).

\begin{figure}[h]
  \centering
  \includegraphics[width=0.29\textwidth]{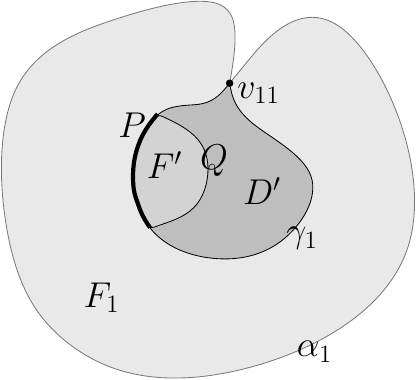}\caption{The proof of Lemma~\ref{lem:intersections}(1).}\label{f:c4t4C}
\end{figure}

(2) First we prove that the intersection of two relators is connected. We proceed analogously to the proof of (1). Suppose not, and let $G_1,G_2$ intersect in a non-connected subgraph leading to a reduced disk
diagram as in Figure~\ref{f:c4t4A} in the middle, with the boundary of $F_i$ mapping to $G_i$. Again, we assume that $D$ has the minimal area among counterexamples and we consider the extended disk diagram
$D\cup F_1 \cup F_2$. By \cite[Proposition 3.5]{HodaQuadric} the new diagram is not reduced and hence,
as in the proof of (1) we get to a contradiction by finding a new counterexample with a smaller area diagram.
This proves the connectedness of the intersection of two relators.

The fact that such intersections does not contain cycles follows immediately from the $C(4)$ condition.

(3) By (1) and (2) it is enough to show that the triple intersection is non-empty.
Here we proceed analogously to (1) and (2). The corresponding diagrams are depicted in Figure~\ref{f:c4t4A} on the right, and the fact that the extended diagram $D\cup F_1 \cup F_2 \cup F_3$ is not reduced follows
from \cite[Proposition 3.7]{HodaQuadric}.
\end{proof}

\begin{lemma}\label{l:c4t4strongH}
  Let $G_1,G_2,G_3$ be three pairwise intersecting relators in a
  simply connected \cftf thickened graphical complex $X$. Then the
  intersection $G_i\cap G_j$ of any two relators is contained in the
  third one.
\end{lemma}

\begin{proof}
  Suppose not. Let $v_i$ be a vertex in $G_j\cap G_k$ not in $G_i$, for
  $\{i,j,k\}=\{1,2,3\}$.  By Lemma~\ref{lem:intersections} there
  exists a vertex $v\in G_1\cap G_2 \cap G_3$ and immersed paths
  $\gamma_i\subseteq G_j\cap G_k$ from $v$ to $v_i$, for all
  $\{i,j,k\}=\{1,2,3\}$.  By our assumption on no degree-one vertices,
  we may find a reduced disk diagram consisting of cells $F_i$ mapped
  to $G_i$, for $i=1,2,3$, as in Figure~\ref{f:c4t4D}. This
  contradicts the $T(4)$ condition.
\end{proof}	
\begin{figure}[h]
  \centering
  \includegraphics[width=0.24\textwidth]{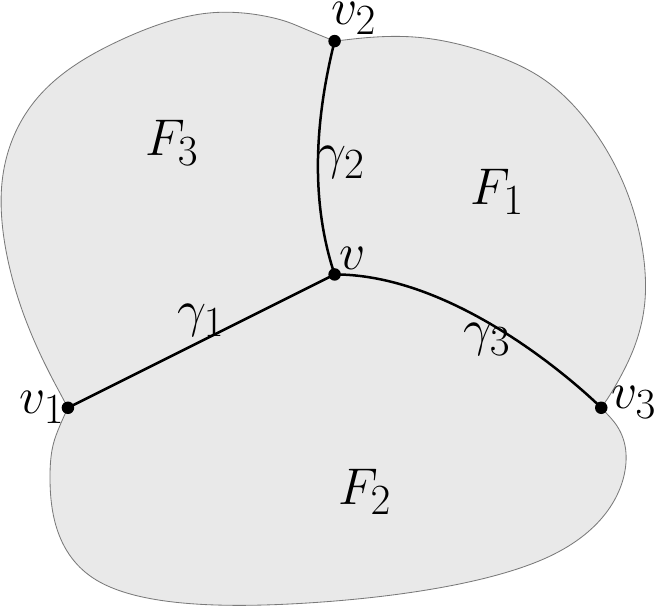}\caption{The proof of Lemma~\ref{l:c4t4strongH}.}\label{f:c4t4D}
\end{figure}

\begin{lemma}\label{lem:flag} Let $X$ be a simply connected \cftf thickened graphical complex and consider a collection $\{G_i \to X\}_{i \in I}$ of relators. If for every $i,j \in I$ the intersection $G_i \cap G_j$ is non-empty then the intersection $\bigcap_{i \in I} G_i$ is a non-empty tree.
\end{lemma}

\begin{proof}
  The lemma follows directly from Lemmas~\ref{l:c4t4strongH} and \ref{lem:intersections}(\ref{3intersect}).
\end{proof}

In view of Lemmas~\ref{l:c4t4strongH} and \ref{lem:intersections}, for a simply connected \cftf graphical complex $X^{\ast}$ we may define a flag simplicial complex $X^{\Delta}$, called its \emph{thickening} as follows:
vertices of $X^{\Delta}$ are the vertices of $X^{\ast}$, and two vertices are connected by an edge iff they are
contained in a common cone-cell. (Observe that the thickening of a graphical complex is not the corresponding thickened graphical complex.)

\begin{theorem}
  \label{l:thicksc}
  Let $X^{\ast}$ be a simply connected \cftf graphical complex.
  Then the $1$-skeleton of the thickening $X^{\Delta}$ of $X^{\ast}$
is Helly. Consequently, a group acting geometrically on $X^{\ast}$ is Helly.
\end{theorem}
\begin{proof}
  Since cone-cells are contractible and, by Lemma~\ref{lem:flag} all
  their intersections are contractible or empty, by Borsuk's Nerve
  Theorem~\cite{Bor,Bj}, the thickening $X^{\Delta}$ is homotopically
  equivalent to $X^{\ast}$.  By Lemmas~\ref{l:c4t4strongH} and
  \ref{lem:intersections}, the hypergraph defined by the thickening is
  triangle-free and hence, by
  Proposition~\ref{hypergraph-clique-Helly} the $1$-skeleton of
  $X^{\Delta}$ is clique-Helly.  The theorem follows by applying the
  local-to-global characterization of Helly graphs from \cite{CCHO} --
  Theorem~\ref{t:lotogloHell}.
\end{proof}

Examples of groups as in Theorem~\ref{l:thicksc} are given by the
following construction.  A \emph{graphical presentation}
$\mathcal P=\langle S \; | \; \varphi \rangle$ is a graph
$\bfG=\bigsqcup_{i \in I}G_i$, and an immersion
$\varphi \colon \bfG \to R_S$, where every $G_i$ is finite and
connected, and $R_S$ is a rose, i.e., a wedge of circles with edges
(cycles) labelled by a set $S$. Alternatively, the map
$\varphi \colon \bfG \to R_S$, called a \emph{labelling}, may be
thought of as an assignment: to every edge of $\bfG$ we assign a
direction (orientation) and an element of $S$.

A graphical presentation $\mathcal P$ defines a group $\Gamma= \Gamma(\mathcal{P})= \pi_1(R_S)/\left<\left<
    \varphi_{\ast}(\pi_1 (G_i))_{i\in I} \right>\right>$.
In other words $\Gamma$ is the quotient of the
free group $F(S)$ by the normal closure of the group generated by all words (over $S\cup S^{-1}$) read along cycles in $\bfG$ (where an oriented edge labelled by $s\in S$ is identified with the edge of the opposite
orientation and the label $s^{-1}$). Observe that removing vertices of degree one from $\bfG$ does not change the group hence we may assume that there are no such vertices in $\bfG$.
A \emph{piece} is a path $P$ labelled by $S$ such that there exist two immersions $p_1 \colon P \to \bfG$ and $p_2\colon P \to \bfG$, and there is no automorphism $\Phi \colon \bfG
\to \bfG$ such that $p_1=\Phi \circ p_2$.

Consider the following graphical complex:
$ X^{\ast}=R_S \cup_{\varphi} \bigsqcup_{i \in I} C(G_i)$.  The
fundamental group of $X^{\ast}$ is isomorphic to $\Gamma$. In the
universal cover $\widetilde X^{\ast}$ of $X^{\ast}$ there might be
multiple copies of cones $C(G_i)$ whose attaching maps differ by lifts
of Aut$(G_i)$.  After identifying all such copies, we obtain the
complex $\widetilde X^{+}$. The group $\Gamma$ acts geometrically, but
not necessarily freely on $\widetilde X^{+}$.  We call the
presentation $\mathcal P$ a \emph{\cftf graphical small cancellation
  presentation} when the complex $X^{\ast}$ is a \cftf graphical
complex.  The presentation $\mathcal P$ is \emph{finite}, and the
group $\Gamma$ is \emph{finitely presented} if the graph $\bfG$ is
finite and the set $S$ (of generators) is finite.  As an immediate
consequence of Theorem~\ref{l:thicksc} we obtain the following.

\begin{corollary}
	\label{c:c4t4}
	Finitely presented graphical \cftf small cancellation groups are Helly.
\end{corollary}

\subsection{Free products with amalgamation over finite subgroups}
\label{s:free_prod}

Let $H$ be a graph with vertex set $\{w_j\}_{j\in J}$. For a collection $\{H_j\}_{j\in J}$ of graphs indexed by vertices of $H$, we consider the collection $\mathcal F {\mathcal H}:=\{F(H_j)\}_{j\in J}$, of their face complexes.
For every edge $e=\{u_j,u_{j'}\}$ in $H$ we pick
vertices $w_j^e\in F(H_j)$ and $w_{j'}^e\in F(H_{j'})$.
The \emph{amalgam of $\mathcal F{\mathcal H}$ over $H$}, denoted $H(\mathcal F\mathcal H)$ is a graph defined as follows.
Vertices of $H(\mathcal F\mathcal H)$ are equivalence classes of the equivalence relation on $\bigcup_{j\in J} V(F(H_j))$ induced by the relation $w_j^e\sim w_{j'}^e$, for all edges $e$ of $H$.
Edges of $H(\mathcal F\mathcal H)$ are induced by edges in the disjoint union $\bigsqcup_{j\in J}F(H_j)$.
The part of Theorem~\ref{t:operations}(\ref{t:operations(1)}) concerning free products with amalgamations over finite subgroups follows from the following result. The case of HNN-extensions follows analogously.

\begin{theorem}
	\label{t:tree_amalgam_groups}
	For $i=1,2$, let $\Gamma_i$ act geometrically on a Helly graph $G_i$, and let $\Gamma'_i < \Gamma_i$ be a finite subgroup, such that $\Gamma'_1$ and $\Gamma'_2$ are isomorphic.
	Then the free product $\Gamma_1 \ast_{\Gamma'_1\cong \Gamma'_2} \Gamma_2$ of $\Gamma_1$ and $\Gamma_2$ with amalgamation over $\Gamma'_1\cong \Gamma'_2$ acts geometrically on an amalgam $H(\mathcal F\mathcal H)$ of
$\mathcal F{\mathcal H}$ over $H$, where $H$ is a tree, elements of $\mathcal H$ are copies of $G_1,G_2$,
and such that $H(\mathcal F\mathcal H)$ is Helly.
\end{theorem}
\begin{proof}
	Let $H$ be the Bass-Serre tree for  $\Gamma_1 \ast_{\Gamma'_1\cong \Gamma'_2} \Gamma_2$. For a vertex $w_j$ of $H$
	corresponding to $\Gamma_i$ we define $H_j$ to be a copy of $G_i$. For an edge $e$ in $H$
	we define $w_j^e$ to be a vertex fixed in $H_j$ by the corresponding conjugate of $\Gamma'_1\cong \Gamma'_2$ (such vertex exists by Theorem~\ref{t:fixedpt} and Proposition~\ref{p:Hellyface}). An equivariant choice of
	vertices $w_j^e$ leads to an amalgam $H(\mathcal F\mathcal H)$ acted geometrically upon $\Gamma_1 \ast_{\Gamma'_1\cong \Gamma'_2} \Gamma_2$. The graph $H(\mathcal F\mathcal H)$ is Helly since it can be obtained by consecutive gluings
	of two Helly graphs along a common vertex -- such gluing obviously results in a Helly graph (for a more general gluing procedure, see \cite{Miesch15}).
\end{proof}

\subsection{Quotients by finite normal subgroups}
\label{s:quotient}
Let $\Gamma$ act (by automorphisms) on a complex $X$. Then $\Gamma$ acts on $\fcom{X}$ and we define
the \emph{fixed point complex} $\fcom{X}^\Gamma$ in the face complex, as the subcomplex spanned by all
vertices of $\fcom{X}$ fixed by $\Gamma$ (that correspond to the cliques of $X$ stabilized by $\Gamma$).
Theorem~\ref{t:operations}(\ref{t:operations(5)}) follows from the following.

\begin{theorem}
	\label{t:quotient}
	Let $\Gamma$ be a group acting by automorphisms on a clique-Helly graph $G$. Let $N\lhd \Gamma$ be a finite normal subgroup. Then $\Gamma / N$ acts by automorphisms on the clique-Helly complex $F(X(G))^N$.
	If $G$ is Helly then $F(X(G))^N$ is Helly as well. If the $\Gamma$ action on $G$ is proper, or cocompact then the induced action of $\Gamma / N$ on  $F(X(G))^N$ is, respectively, proper, or cocompact.
\end{theorem}
\begin{proof}
	The $\Gamma$-action on $G$ induces the $\Gamma$-action on $F(X(G))$, and consequently the $\Gamma /N$-action on $F(X(G))^N$. It is clear that the latter is proper or cocompact if the initial action is so.
	By Lemma~\ref{l:fixsetcliqueHelly} and Corollary~\ref{c:fixsetHelly} the complex $F(X(G))^N$
	is (clique-)Helly if $G$ is so.
\end{proof}

\subsection{Actions with Helly stabilizers}
\label{s:qmedian}

Our goal now is to apply the general theory developed in \cite{Qm} in order to show that the family of Helly groups is stable under several group-theoretic operations.
The main theorem in this direction is Theorem \ref{thm:HellyGeneral} below, which shows that, if a group acts on a quasi-median graph in a specific way and if clique-stabilizers
are Helly, then the group must be Helly as well. We emphasize that, contrary to the rest of the article, our quasi-median graphs may not be locally finite; in particular, their cliques will be typically infinite. We begin by giving general definitions and properties related to quasi-median graphs.

\subsubsection{Preliminaries on quasi-median graphs} Recall that a graph is \emph{quasi-median} if it is weakly modular and does not contain $K_4^-$ and $K_{3,2}$ as induced subgraphs.  Several subgraphs are of interest in the study of quasi-median graphs:
\begin{itemize}
\item Contrary to the rest of the paper, in this subsection, by a
  \emph{clique}, we mean a maximal complete subgraph.
	\item A \emph{prism} is an induced subgraph which decomposes as a Cartesian product of cliques. The maximal number of factors of a prism in a quasi-median graph is referred to as its \emph{cubical dimension} (which may be infinite). (Observe that, by maximality of our cliques, a single vertex defines a prism of zero cubical dimension if and only if it is isolated.)
	\item A \emph{hyperplane} is an equivalence class of edges with respect to the transitive closure of the relation which identifies two edges whenever they belong to a common triangle or they are opposite sides of a square (i.e., a four-cycle). Two cliques are \emph{parallel} if they belong to the same hyperplane. Two hyperplanes are \emph{transverse} if their union contains two adjacent edges of some square.
	\item According to \cite[Proposition 2.15]{Qm}, a hyperplane \emph{separates} a quasi-median graph, i.e., the graph obtained by removing the interiors of the edges of a hyperplane contains at least two connected components. Such a component is a \emph{sector} delimited by the hyperplane.
\end{itemize}
According to \cite{quasimedian} and \cite[Lemmas 2.16 and 2.80]{Qm}, cliques and prisms are gated subgraphs. For convenience, in the sequel, we will refer to the map sending a vertex to its gate in a given gated subgraph as the \emph{projection} onto this subgraph.

\subsubsection{Systems of metrics} Given a quasi-median graph $G$, a \emph{system of metrics} is the data of a metric $\delta_C$ on each clique $C$ of $G$. Such a system is \emph{coherent} if for any two parallel cliques $C$ and $C'$ one has
$$\delta_C(x,y)= \delta_{C'}(t_{C \to C'}(x),t_{C\to C'}(y)) \ \text{for every vertices $x,y \in C$},$$
where $t_{C \to C'}$ denotes the projection of $C$ onto $C'$. As shown in \cite[Section 3.2]{Qm}, it is possible to extend a coherent system of metrics to a global metric on $G$. Several constructions are possible, we focus on the one which will be relevant for our study of Helly groups. A \emph{chain} $R$ between two vertices $x,y \in V(G)$ is a sequence of vertices
$(x_1=x, \ x_2, \ldots, x_{n-1}, \ x_n=y)$
such that, for every $1 \leq i \leq n-1$, the vertices $x_i$ and $x_{i+1}$ belong to a common prism, say $P_i$. The \emph{length} of $R$ is
$\ell(R) = \sum\limits_{i=1}^{n-1} \delta_{P_i}(x_i,x_{i+1})$
where $\delta_{P_i}$ denotes the $\ell_\infty$-metric associated to the local metrics defined on the cliques of $P_i$. Then the global metric extending our system of metrics is
\[\delta_\infty : (x,y) \mapsto \min \{ \ell(R):  \text{$R$ is a chain between $x$ and $y$} \}.\]
Along this section, all our local metrics will be graph-metrics. It is
worth noticing that, in this case, $\delta_\infty$ turns out to be a
graph-metric as well. Consequently, $(G, \delta_\infty)$ will be
considered as a graph. More precisely, this graph has $V(G)$ as its
vertex-set and its edges link two vertices if they are at
$\delta_\infty$-distance one. Notice that, if
$P=C_1 \times \cdots \times C_n$ is a prism of $G$, then the graph
$(P, \delta_\infty)$ is isometric to the direct product
$(C_1,\delta_{C_1}) \boxtimes \cdots \boxtimes (C_n,\delta_{C_n})$.

The main result of this section is that extending a system of Helly graph-metrics produces a global metric which is again Helly. More precisely:

\begin{proposition}\label{prop:ExtendingHelly}
Let $G$ be a quasi-median graph of finite cubical dimension endowed with a coherent system of graph metrics $\{ \delta_C: \text{$C$ clique of $G$} \}$. Suppose that $(C,\delta_C)$ is a locally finite Helly graph for every clique $C$ of $G$ and that each vertex belongs to only finitely many cliques. Then $(G,\delta_\infty)$ is a Helly graph.
\end{proposition}

We begin by proving the following preliminary lemma:

\begin{lemma}\label{lem:ExtendingSimplyConnected}
Let $G$ be a quasi-median graph endowed with a coherent system of graph metrics $\{ \delta_C:  \text{$C$ clique of $G$} \}$. Suppose that the clique complex of $(C,\delta_C)$ is simply connected for every clique $C$ of $G$. Then the clique complex of $(G,\delta_\infty)$ is simply connected as well.
\end{lemma}

\begin{proof}
Let $\gamma$ be a cycle in the one-skeleton of $(G,\delta_\infty)$. We want to prove by induction over the number of hyperplanes of $G$ crossed by $\gamma$ that $\gamma$ is null-homotopic in the clique complex of $(G,\delta_\infty)$. Of course, if $\gamma$ does not cross any hyperplane, then it has to be reduced to a single vertex and there is nothing to prove. So from now on we assume that $\gamma$ crosses at least one hyperplane.

Let $Y \subseteq V(G)$ denote the gated hull of the vertex-set of $\gamma$. Notice that the subgraph of $(G,\delta_\infty)$ spanned by the vertices of $Y$ coincides with $(Y,\delta_\infty)$. According to \cite[Proposition 2.68]{Qm}, the hyperplanes of $Y$ are exactly the hyperplanes of $G$ crossed by $\gamma$. If the hyperplanes of $Y$ are pairwise transverse, then it follows from \cite[Lemma 2.74]{Qm} that $Y$ is a single prism. Consequently, $(Y,\delta_\infty)$ is the direct product of graphs whose clique complexes are simply connected, so that $\gamma$ must be null-homotopic in the clique complex of $(G,\delta_\infty)$. From now on, assume that $Y$ contains at least two hyperplanes, say $J$ and $H$, which are not transverse.

Let $S$ denote the sector delimited by $H$ which contains $J$. Decompose $\gamma$ as a concatenation of subpaths $\alpha_1 \beta_1 \cdots \alpha_n \beta_n \alpha_{n+1}$ such that $\alpha_1, \ldots, \alpha_{n+1}$ are included in $S$ and $\beta_1, \ldots, \beta_n$ intersect $S$ only at their endpoints. For every $1 \leq i \leq n$, fix a path $\sigma_i \subset (Y,\delta_\infty)$ between the endpoints of $\beta_i$ which does not cross $J$ (such a path exists as a consequence of \cite[Proposition 3.16]{Qm}). Notice that $\beta_i\sigma_i^{-1}$ is a cycle which does not cross $H$, so by our induction assumptions we know that $\beta_i$ and $\sigma_i$ are homotopic (in the clique complex). Therefore, $\gamma$ is homotopic (in the clique complex) to the cycle $\alpha_1 \sigma_1 \cdots \alpha_n \sigma_n \alpha_{n+1}$ which does not cross $H$. We conclude that $\gamma$ is null-homotopic (in the clique complex) by our induction assumptions.
\end{proof}

\begin{proof}[Proof of Proposition \ref{prop:ExtendingHelly}.]
Fix a set $\mathcal{C}$ of representatives of cliques modulo parallelism. For every $C \in \mathcal{C}$, let $\pi_C : G \to C$ denote the projection onto $C$. We claim that
$$\pi : \left\{ \begin{array}{ccc} (G,\delta_\infty) & \to & \underset{C \in \mathcal{C}}{\boxtimes} (C,\delta_C) \\ x & \mapsto & (\pi_C(x)) \end{array} \right.$$
is an injective graph morphism.

Let $x,y \in (G,\delta_\infty)$ be two adjacent vertices, i.e., two vertices of $G$ satisfying $\delta_\infty(x,y)=1$. So there exists a prism $P$ of $G$, thought of as a product of cliques $C_1 \times \cdots \times C_n$, which contains $x,y$ and such that the projections of $x,y$ onto each $C_i$ are identical or $\delta_{C_i}$-adjacent. For every $1 \leq i \leq n$, let $C_i' \in \mathcal{C}$ denote the representative of $C_i$. Because our system of metrics is coherent, we also know that the projections of $x,y$ onto each $C_i'$ are identical or $\delta_{C_i'}$-adjacent. Therefore, $\pi(x)$ and $\pi(y)$ are adjacent in the subgraph $\underset{1 \leq i \leq n}{\boxtimes} (C_i',\delta_{C_i'})$ of $\underset{C \in \mathcal{C}}{\boxtimes} (C,\delta_C)$. Thus, we have proved that $\pi$ is a graph morphism.

Now, let $x,y \in (G,\delta_\infty)$ be two distinct vertices. As a consequence of \cite[Proposition 2.30]{Qm}, there exists a hyperplane separating $x$ and $y$. Therefore, if $C \in \mathcal{C}$ denotes the representative clique dual to this hyperplane, then $\pi_C(x) \neq \pi_C(y)$. Hence $\pi(x) \neq \pi(y)$, proving that $\pi$ is indeed injective.

Notice that the image of a prism of $G$ under $\pi$ is a finite subproduct of $\underset{C \in \mathcal{C}}{\boxtimes} (C,\delta_C)$. Moreover, because every vertex of $G$ belongs to only finitely many cliques and because each $(C,\delta_C)$ is locally finite, we know that $(G,\delta_\infty)$ must be locally finite.
As a consequence, $(G,\delta_\infty)$ is an UGP over $\{ (C,\delta_C), \ C \in \mathcal{C}\}$. Now, we claim that our UGP satisfies the 3-piece condition. So let $P_1,P_2,P_3$ be three pairwise intersecting prisms in $G$. Because prisms are gated, they satisfy the Helly property, so there exists a vertex $x \in P_1 \cap P_2 \cap P_3$. Let $\mathcal{J}$ denote the set of all the hyperplanes that have a clique in at least two prisms among $P_1,P_2,P_3$. Observe that any two distinct hyperplanes $J_1,J_2 \in \mathcal{J}$ are transverse (i.e., there exists a prism containing cliques from both $J_1$ and $J_2$). For every $J \in \mathcal{J}$, fix a clique $C_J \subset P_1 \cup P_2 \cup P_3$ in $J$ that contains $x$. And let $P$ denote the gated hull of the union of all the $C_J$, $J \in \mathcal{J}$. Because the hyperplanes in $\mathcal{J}$ are pairwise transverse, we deduce from \cite[Proposition~2.68 and Lemma~2.74]{Qm} that $P$ is a prism. Now, our goal is to show that $P$ is the piece of $G$ we are looking for. So let $C \in \mathcal{C}$ be such that at least two prisms among $P_1,P_2,P_3$ have projection $C$ on the $C$-coordinate. It follows from \cite[Lemma~2.20]{Qm} that the hyperplane $J$ containing $C$ intersects at least two prisms among $P_1,P_2,P_3$, hence $J \in \mathcal{J}$. By construction, $P$ contains a clique in $J$, hence a clique parallel to $C$. In other words, $C$ is also the projection of $P$ on the $C$-coordinate, as desired. Thus, we have verified that the 3-piece condition holds. We conclude that $(G,\delta_\infty)$ is a Helly graph by combining Theorem \ref{th-UGP-3piece} with Lemma \ref{lem:ExtendingSimplyConnected}.
\end{proof}

\subsubsection{Constructing Helly groups} We are now ready to
construct new Helly groups from old ones. Recall from \cite{Qm} that
the action of group $\Gamma$ on a quasi-median graph $G$ is
\emph{topical-transitive} if it satisfies the two following conditions:
\begin{enumerate}[(1)]
\item for every hyperplane $J$, every clique $C \subset J$ and every
  $g \in \mathrm{stab}(J)$, there exists  $h \in \mathrm{stab}(C)$
  such that $g$ and $h$ induce the same permutation on the set of
  sectors delimited by $J$;
\item  for every clique $C$ of $G$,
  \begin{itemize}
  \item either $C$ is finite and $\mathrm{stab}(C)= \mathrm{fix}(C)$;
  \item or $\mathrm{stab}(C) \curvearrowright C$ is free and
    transitive on the vertices.
  \end{itemize}
\end{enumerate}
Then the statement we are interested in is:

\begin{theorem}\label{thm:HellyGeneral}
Let $\Gamma$ be a group acting topically-transitively on a quasi-median graph $G$. Suppose that:
\begin{itemize}
	\item every vertex of $G$ belongs to finitely many cliques;
	\item every vertex-stabilizer is finite;
	\item the cubical dimension of $G$ is finite;
	\item $G$ contains finitely many $\Gamma$-orbits of cliques;
	\item for every maximal prism $P= C_1 \times \cdots \times C_n$, $\mathrm{stab}(P)= \mathrm{stab}(C_1) \times \cdots \times \mathrm{stab}(C_n)$.
\end{itemize}
If clique-stabilizers are Helly, then so is $\Gamma$.
\end{theorem}

Before turning to the proof of Theorem \ref{thm:HellyGeneral}, we need the following easy observation (which can be proved by following the lines of \cite[Lemma 4.34]{Qm}):

\begin{lemma}\label{lem:HellyVertexTrivial}
For every Helly group $\Gamma$, there exist a Helly graph $G$ and a vertex $x_0 \in G$ such that $\Gamma$ acts geometrically on $G$ and $\mathrm{stab}(x_0)$ is trivial.\qed
\end{lemma}

\begin{proof}[Proof of Theorem \ref{thm:HellyGeneral}.]
First of all, observe that $G$ contains only finitely many $\Gamma$-orbits of prisms. Indeed, let $\mathcal{C}$ be a finite collection of representatives of cliques modulo the action of $\Gamma$. For every $C \in \mathcal{C}$, fix a vertex $x_C \in C$. Let $\mathcal{P}$ denote the set of all the prisms in $G$ that contain a $x_C$ for some $C \in \mathcal{C}$. Because each vertex belongs to only finitely many cliques by assumption, we know that $\mathcal{P}$ is a finite collection. Now, if $P$ is an arbitrary prism in $G$, there must exist $g \in \Gamma$ and $C \in \mathcal{C}$ such that $gP$ contains $C$, and a fortiori $x_C$, hence $gP \in \mathcal{P}$. This proves our observation. By combining Lemma \ref{lem:HellyVertexTrivial} with \cite[Proposition 7.8]{Qm}, we know that there exists a new quasi-median graph $Y$ endowed with a coherent system of metrics $\{\delta_C: \text{$C$ clique of $Y$} \}$ such that $\Gamma$ acts geometrically on $(Y, \delta_\infty)$ and such that $(C,\delta_C)$ is a Helly graph for every clique $C$ of $Y$. Because $(Y,\delta_\infty)$ defines a Helly graph according to Proposition \ref{prop:ExtendingHelly}, we conclude that $\Gamma$ is a Helly group.
\end{proof}

We now record several applications of Theorem \ref{thm:HellyGeneral}.

\subsubsection{Graph products of groups.} Given a \emph{simplicial graph} $G$ and a collection of groups $\mathcal{G}=\{ \Gamma_u:  u \in V(G) \}$ indexed by the vertices of $G$ (called \emph{vertex-groups}), the \emph{graph product} $G \mathcal{G}$ is the quotient $$\left( \underset{u \in V(G)}{\ast} \Gamma_u \right) / \langle \langle [g,h]=1, g \in \Gamma_u, h \in \Gamma_v \ \text{if} \ (u,v) \in E(G) \rangle \rangle.$$
For instance, if $G$ has no edge, then $G \mathcal{G}$ is the free product of $\mathcal{G}$; and if $G$ is a complete graph, then $G \mathcal{G}$ is the direct sum of $\mathcal{G}$. One often says that graph products interpolate between free products and direct sums.

By combining Theorem \ref{thm:HellyGeneral} with \cite[Proposition 8.14]{Qm}, one obtains:

\begin{theorem}
Let $G$ be a finite simplicial graph and $\mathcal{G}$ a collection of groups indexed by $V(G)$. If the vertex-groups are Helly, then so is the graph product $G \mathcal{G}$.
\end{theorem}

\subsubsection{Diagram products of groups.} Let $\mathcal{P}= \langle \Sigma:  \mathcal{R} \rangle$ be a semigroup presentation. We assume that, if $u=v$ is a relation which belongs to $\mathcal{R}$, then $v=u$ does not belong to $\mathcal{R}$; in particular, $\mathcal{R}$ does not contain relations of the form $u=u$. The \emph{Squier complex} $S(\mathcal{P})$ is the square-complex
\begin{itemize}
	\item whose vertices are the positive words $w \in \Sigma^+$;
	\item whose edges $(a,u=v,b)$ link $aub$ and $avb$ where $(u=v) \in \mathcal{R}$;
	\item and whose squares $(a,u=v,b,p=q,c)$ are delimited by the edges $(a,u=v,bpc)$, $(a,u=v,bqc)$, $(aub,p=q,c)$, $(avb,p=q,c)$.
\end{itemize}
The connected component of $S(\mathcal{P})$ containing a given word $w \in \Sigma^+$ is denoted by $S(\mathcal{P},w)$. Given a collection of groups $\mathcal{G}= \{\Gamma_s, s \in \Sigma\}$ labelled by the alphabet $\Sigma$, the \emph{diagram product} $D(\mathcal{P},\mathcal{G},w)$ is isomorphic to the fundamental group of the following 2-complex of groups:
\begin{itemize}
	\item the underlying 2-complex is the 2-skeleton of the Squier complex $S(\mathcal{P},w)$;
	\item to any vertex $u=s_1\cdots s_r \in \Sigma^+$ is associated the group $\Gamma_u=\Gamma_{s_1} \times \cdots \times \Gamma_{s_r}$;
	\item to any edge $e=(a, u \to v,b)$ is associated the group $\Gamma_{e}=\Gamma_a \times \Gamma_b$;
	\item to any square is associated the trivial group;
	\item for every edge $e=(a,u \to v,b)$, the monomorphisms $\Gamma_e \to \Gamma_{aub}$ and $\Gamma_e \to \Gamma_{avb}$ are the canonical maps $\Gamma_a \times \Gamma_b \to \Gamma_a \times \Gamma_u \times \Gamma_b$ and $\Gamma_a \times \Gamma_b \to \Gamma_a \times \Gamma_v \times \Gamma_b$.
\end{itemize}
We refer to \cite{MR1725439} and \cite[Section 10]{Qm} for more information about diagram products of groups. By combining Theorem \ref{thm:HellyGeneral} with \cite[Proposition 10.33 and Lemma 10.34]{Qm}, one obtains:

\begin{theorem}\label{thm:DiagProducts}
Let $\mathcal{P}= \langle \Sigma: \mathcal{R} \rangle$ be a finite semigroup presentation, $\mathcal{G}$ a collection of groups indexed by the alphabet $\Sigma$ and $w \in \Sigma^+$ a baseword. If $\{ u \in \Sigma^+: \text{$u=w$ mod $\mathcal{P}$} \}$ is finite and if the groups of $\mathcal{G}$ are all Helly, then the diagram product $D(\mathcal{P}, \mathcal{G},w)$ is a Helly group.
\end{theorem}

Explicit examples of diagram products can be found in \cite[Section 10.7]{Qm}. For instance, the \emph{$\square$-product} of two groups $\Gamma_1$ and $\Gamma_2$, defined by the relative presentation
$$\Gamma_1 \square \Gamma_2 = \langle \Gamma_1,\Gamma_2,t: [g,h]=[g,tht^{-1}]=1, \ g \in \Gamma_1, h \in \Gamma_2 \rangle.$$
is a diagram product \cite[Example 10.65]{Qm}. As it satisfies the assumptions of Theorem~\ref{thm:DiagProducts}, it follows that:

\begin{corollary}
If $\Gamma_1$ and $\Gamma_2$ are two Helly groups, then so is $\Gamma_1 \square \Gamma_2$.
\end{corollary}

\subsubsection{Right-angled graphs of groups} Roughly speaking, right-angled graphs of groups are fundamental groups of graphs of groups obtained by gluing graph products together along ``simple'' subgroups. We refer to \cite{MR1954121} for more information about graphs of groups.

\begin{definition}
Let $G, H$ be two simplicial graphs and $\mathcal{G}, \mathcal{H}$ two families of groups respectively indexed by $V(G),V(H)$. A morphism $\Phi : G \mathcal{G} \to H \mathcal{H}$ is a \emph{graphical embedding} if there exists an embedding $f : G \to H$ and isomorphisms $\varphi_v : \Gamma_v \to \Gamma_{f(v)}$, $v \in V(G)$, such that $f(G)$ is an induced subgraph of $H$ and $\Phi(g)= \varphi_v(g)$ for every $v \in V(G)$ and $g \in \Gamma_v$.
\end{definition}

\begin{definition}
A \emph{right-angled graph of groups} is a graph of groups such that each (vertex- and edge-)group has a fixed decomposition as a graph product and such that each monomorphism of an edge-group into a vertex-group is a graphical embedding (with respect to the structures of graph products we fixed).
\end{definition}

In the following, a \emph{factor} will refer to a vertex-group of one of these graph products. Let $\mathfrak{G}$ be a right-angled graph of groups. Notice that, if $e$ is an oriented edge from a vertex $x$ to another $y$, then the two embeddings of $\Gamma_e$ in $\Gamma_x$ and $\Gamma_y$ given by $\mathfrak{G}$ provide an isomorphism $\varphi_e$ from a subgroup of $\Gamma_x$ to a subgroup of $\Gamma_y$. Moreover, if $\Gamma \subset \Gamma_x$ is a factor, then $\varphi_e(\Gamma):= \{ g \in \Gamma_y \mid \exists h \in \Gamma, \varphi_e(h)=g\}$ is either empty or a factor of $\Gamma_y$. Set
$$\Phi(\Gamma)= \{ \varphi_{e_k} \circ \cdots \circ \varphi_{e_1}: \text{$e_1,\ldots,e_k$ oriented cycle at $x$, $\varphi_{e_k} \circ \cdots \circ \varphi_{e_1}(\Gamma)=\Gamma$} \},$$
thought of as a subgroup of the automorphism group $\mathrm{Aut}(\Gamma)$.

By combining Theorem \ref{thm:HellyGeneral} with \cite[Proposition 11.26 and Lemma 11.27]{Qm}, one obtains:

\begin{theorem}\label{thm:RAGG}
Let $\mathfrak{G}$ be a right-angled graph of groups such that $\Phi(\Gamma)= \{ \mathrm{Id} \}$ for every factor $\Gamma$. Suppose that the underlying abstract graph and the simplicial graphs defining the graph products are all finite. If the factors are Helly, then so is the fundamental group of $\mathfrak{G}$.
\end{theorem}

Explicit examples of fundamental groups of right-angled graphs of groups can be found in \cite[Section 11.4]{Qm}. For instance, the \emph{$\rtimes$-power} of a group $\Gamma$ \cite[Example~11.38]{Qm}, defined by the relative presentation
$$\Gamma^{\rtimes} = \langle \Gamma,t: [g,tgt^{-1}]=1, \ g \in \Gamma \rangle,$$
is the fundamental group of a right-angled graph of groups satisfying the assumptions of Theorem \ref{thm:RAGG}, hence:

\begin{corollary}
If $\Gamma$ is a Helly group, then so is $\Gamma^\rtimes$.
\end{corollary}

Also, the \emph{$\Join$-product} of two groups $\Gamma_1$ and $\Gamma_2$ \cite[Example 11.39]{Qm}, defined by the relative presentation
$$\Gamma_1 \Join \Gamma_2 = \langle \Gamma_1,\Gamma_2,t: [g,h]=[g,tht^{-1}]= [h,tht^{-1}]=1, \ g \in \Gamma_1, h\in \Gamma_2 \rangle,$$
is the fundamental group of a right-angled graph of groups satisfying the assumptions of Theorem \ref{thm:RAGG}, hence:

\begin{corollary}
If $\Gamma_1$ and $\Gamma_2$ are Helly groups, then so is $\Gamma_1 \Join \Gamma_2$.
\end{corollary}

 \section{Properties of Helly groups}\label{s:Hprops}
The main goal of this section is proving
Theorem~\ref{t:properties1}(\ref{t:properties1(2)})-(\ref{t:properties1(4)})(\ref{t:properties1(6)})-(\ref{t:properties1(9)})
from the
Introduction. (Theorem~\ref{t:properties1}(\ref{t:properties1(1)}) is
proved in the subsequent Section~\ref{s:biautomatic} and
Theorem~\ref{t:properties1}(\ref{t:properties1(5)}) follows from
Theorems~\ref{t:injbicomb} and~\ref{t:helly=inj}). On the way we show
also some immediate consequences of the main results and prove related
facts concerning groups acting on Helly graphs.

\subsection{Fixed points for finite group actions}
\label{s:fixedpts}

In this subsection we prove Theorem~\ref{t:properties1}(\ref{t:properties1(2)}), stating that every Helly group has only finitely many conjugacy classes of finite subgroups. It is an immediate consequence of the following result interesting on it own.
\begin{theorem}[Fixed Point Theorem]\label{t:fixedpt}
  Let $\Gamma$ be a group acting by automorphisms on a Helly graph $G$
  without infinite cliques. If $\Gamma$ has bounded orbits, then there
  exists a clique of $G$ stabilized by $\Gamma$.  In particular, there
  is a fixed vertex of the induced action of $\Gamma$ on the face
  complex $F(G)$.
\end{theorem}
\begin{proof}
  Pick a vertex $v$ of $G$ and consider its $\Gamma$-orbit $\Gamma v$.
  Let $N$ be the diameter of $\Gamma v$. The intersection
  $B:=\bigcap_{g \in \Gamma}B_N(gv)$ of $N$-balls centered at vertices
  of the orbit $\Gamma v$ is a non-empty bounded $\Gamma$-invariant
  Helly graph.  Since $G$ does not contain infinite simplices, by
  \cite[Theorem A]{Polat1993}, the graph $B$ contains a clique
  stabilized by $\Gamma$.
\end{proof}

\noindent \emph{Proof of Theorem~\ref{t:properties1}(\ref{t:properties1(2)}).}  This follows immediately from the Fixed Point Theorem~\ref{t:fixedpt}, as e.g.\ in the case of \catz groups in \cite[Proposition I.8.5]{BrHa}.
\hfill $\square$

\begin{remark}
	\label{r:conjcl}
	Theorem~\ref{t:properties1}(\ref{t:properties1(2)}) can be also deduced from \cite{Dress1989} or \cite[Proposition 1.2]{Lang2013} combined with our Theorem~\ref{t:helly=inj}.
\end{remark}

\subsection{Flats vs hyperbolicity}
\label{s:flats-hyper}

\noindent
\emph{Proof of Theorem~\ref{t:properties1}(\ref{t:properties1(3)}).}
Suppose that $\Gamma$ is hyperbolic. Then $G$ is hyperbolic and,
clearly, does not contain an isometric $\ell_{\infty}$--square-grid.
For the converse, recall that if $\Gamma$ is not hyperbolic then $G$
contains isometric finite $\ell_{\infty}$--square-grids of arbitrary
size, by Proposition~\ref{prop-Helly-hyp}. Since $\Gamma$ acts
geometrically on $G$ (and, in particular, $G$ is locally finite), by a
diagonal argument it follows that $G$ contains an isometric infinite
$\ell_{\infty}$--grid (see e.g.\ \cite[Lemma II.9.34 and
Theorem II.9.33]{BrHa}). \hfill $\square$

\subsection{Contractibility and Hellyness of the fixed point
  set}\label{s:contrfix} The aim of this section is to prove that for
a group acting on a Helly complex, its fixed point set is
contractible. This leads to a proof of
Theorem~\ref{t:properties1}(\ref{t:properties1(4)}) showing that the
Helly complex is a model for the classifying space for proper actions.
Furthermore, we show that the fixed point subcomplex of the face
complex (of the Helly complex on which the group acts) is Helly.

\begin{lemma}\label{l:barSubCont}
  Let $\Gamma < \mr{Aut}(X)$ be a group of automorphisms of a locally
  finite Helly complex $X$. The fixed point set $X'^\Gamma$ of the
  barycentric subdivision $X'$ of $X$ is contractible.
\end{lemma}

\begin{proof}
  Let $\sigma$ be a simplex of $X$ stabilized by $\Gamma$. For every
  $N>0$, the intersection $B_N:=\bigcap_{v\in \sigma^{(0)}}B_N(v)$ of
  $N$-balls centered at vertices of $\sigma$ is Helly, hence
  dismantlable.  It is also $\Gamma$-invariant, by construction.  The
  fixed point set $B_N'^\Gamma$ in the barycentric subdivision $B_N'$
  of $B_N$ is contractible by \cite[Theorem 6.5]{BarmakMinian2012} or
  \cite[Theorem 1.2]{HenselOsajdaPrzytycki}.  Since the sets $B_N$
  exhaust $X$ it follows that the fixed point set $X'^\Gamma$ in the
  barycentric subdivision $X'$ of $X$ is contractible.
\end{proof}

Theorem~\ref{t:properties1}(\ref{t:properties1(4)}) is a part of the following corollary of Theorem~\ref{t:fixedpt} and Lemma~\ref{l:barSubCont}.

\begin{corollary}\label{c:EG}
  Let $\Gamma$ be a group acting properly on a locally finite Helly
  graph $G$. Then, the Helly complex $X(G)$ is a model for the
  classifying space $\underline{E}\Gamma$ for proper actions of
  $\Gamma$. If the action is cocompact then the model is finite
  dimensional and cocompact.
\end{corollary}

In view of Theorem~\ref{t:helly=inj} and \cite[Theorem 1.4]{Lang2013} there exists also another model for $\underline{E}\Gamma$, defined as follows.
\begin{theorem}\label{t:EG}
  Let $\Gamma$ be a group acting properly on a locally finite Helly
  graph $G$.  The injective hull $E(G)$ of $G$ is a model for the
  classifying space $\underline{E}\Gamma$ for proper actions of
  $\Gamma$. If the action is cocompact then the model is finite
  dimensional and cocompact.
\end{theorem}
\begin{remark}
	\label{r:EG}
	Observe that $X(G)$ can be non-homeomorphic to $E(G)$. For example, if $G$ is an  $(n+1)$-clique then obviously
	the clique complex $X(G)$ is an $n$-simplex, whereas the injective hull $E(G)$ is a cone over $n+1$ points,
	that is, a tree.
\end{remark}

Recall that the fixed point complex $\fcom{X}^\Gamma$ in the face
complex is the subcomplex spanned by all vertices of $\fcom{X}$ fixed
by $\Gamma$. We now prove that $\fcom{X}^\Gamma$ is Helly.

\begin{lemma}[Clique-Helly fixed point set]\label{l:fixsetcliqueHelly}
  Let $\Gamma < \mr{Aut}(X)$ be a group of automorphisms of a locally
  finite clique-Helly complex $X$. Then the fixed point complex
  $\fcom{X}^\Gamma$ is clique-Helly.
\end{lemma}
\begin{proof}
  Let $uvw$ be a triangle in $\fcom{X}^\Gamma$.
By the clique-Helly property for $\fcom{X}$
  (Propostion~\ref{p:Hellyface}) there is a vertex $z\in \fcom{X}$
  adjacent to all vertices of $F(X)$ spanning triangles with an edge
  of $uvw$ (Proposition~\ref{clique_Helly_triangle}). Since $uvw$
  belongs to $\fcom{X}^\Gamma$, all vertices in the orbit $\Gamma
  z$ have the same property as $z$, i.e., they are adjacent to all
  vertices of $\fcom{X}$ spanning triangles with an edge of
  $uvw$. Consequently, they span a simplex of $F(X)$. Let $\sigma$ be
  the union of the simplices of $X$ corresponding to the vertices of
  $\Gamma z$ in $\fcom{X}$.  By Lemma~\ref{l:triface}, $\sigma$ is a
  simplex of $X$. Let $y$ be the vertex of $\fcom{X}$ corresponding to
  $\sigma$. Notice that $y$ belongs to $\fcom{X}^\Gamma$. We now prove
  that $y$ satisfies the assumption of
  Proposition~\ref{clique_Helly_triangle}. Pick a vertex $x$ of
  $\fcom{X}^\Gamma$ spanning a triangle with an edge of $uvw$, say
  with $uv$. By the definition of $z$, $x$ is adjacent to $z$ and to any
  $z' \in \Gamma z$. Consequently, for any $z' \in \Gamma z$, $x$ and
  $z'$ correspond to two subsimplices $\tau_x$ and $\tau_{z'}$ of a
  common simplex of $X$. Therefore, all vertices of $\tau_x$ are
  adjacent to all vertices of $\tau_{z'}$. Since $\sigma = \bigcup_{z'
    \in \Gamma z} \tau_{z'}$, $\tau_x$ and $\sigma$ are also
  subsimplices of a common simplex of $X$. This implies that $x$ and
  $y$ are adjacent in $\fcom{X}$.
\end{proof}

\begin{corollary}[Helly fixed point set]\label{c:fixsetHelly}
  Let $\Gamma < \mr{Aut}(X)$ be a group of automorphisms of a locally
  finite Helly complex $X$. Then the fixed point complex
  $\fcom{X}^\Gamma$ is Helly.
\end{corollary}
\begin{proof}
  Since every edge in $\fcom{X}^\Gamma$ is homotopic to a path in
  $X'^\Gamma$, we have that every cycle in $\fcom{X}^\Gamma$ is
  homotopic to a cycle in $X'^\Gamma$, and hence $\fcom{X}^\Gamma$ is
  simply connected by Lemma~\ref{l:barSubCont}. Hence by
  Lemma~\ref{l:fixsetcliqueHelly} and Theorem~\ref{t:lotogloHell},
  $\fcom{X}^\Gamma$ is Helly.
\end{proof}

\subsection{EZ-boundaries}
\label{s:EZ}

For a group $\Gamma$ acting geometrically on $X$, by an \emph{EZ-structure for $\Gamma$} we mean a pair $(\overline{X},\partial X)$, where $\overline{X}=X\cup  \partial X$ is a compactification of $X$ being an Euclidean retract with the following additional properties. The \emph{EZ-boundary} $\partial X$ is a
$Z$-set in $\overline{X}$ such that, for every compact $K\subset X$ the sequence $(gK)_{g\in \Gamma}$ is
a null sequence, and the action $\Gamma \curvearrowright X$ extends to an action $\Gamma \curvearrowright \overline{X}$ by homeomorphisms. This notion was first introduced by Bestvina \cite{Besvina1996} (without
the requirement of extending $\Gamma \curvearrowright X$ to $\Gamma \curvearrowright \overline{X}$), then by
Farrell-Lafont \cite{FaLa2005} (for free actions), and finally in \cite{OsaPrzy2009} (in the form above).
Homological invariants of the boundary are related to homological invariants of the group, and the existence of
an EZ-structure has some important consequences (e.g.\ it implies the Novikov conjecture in the torsion-free case). Conjecturally,
all groups with finite classifying spaces admit EZ-structures, but such objects were constructed only for
limited classes of groups -- notably for hyperbolic groups and for \catz groups. Theorem~\ref{t:properties1}(\ref{t:properties1(6)}) is a consequence of the following.

\begin{theorem}
	\label{t:EZ}
	Let $\Gamma$ act geometrically on a Helly graph $G$. Then there exists an EZ-boundary $\partial G$
	such that $(X(G)\cup \partial G,\partial G)$ and $(E(G)\cup \partial G,\partial G)$ are
	EZ-structures for $\Gamma$.
\end{theorem}
\begin{proof}
  It is shown in \cite{DesLang2015} that for a complete metric space
  $E(G)$ with a convex and consistent bicombing there exists
  $\partial G$ (space of equivalence classes of combing rays) such
  that $(E(G)\cup \partial G,\partial G)$ is a so-called
  Z-structure. The proof is easily adapted to show that it is an
  EZ-structure (see e.g.\ \cite{OsaPrzy2009} where a much weaker
  version of a `coarse bicombing' is used to define an
  EZ-structure). It follows that $(X(G)\cup \partial G,\partial G)$ is
  an EZ-structure as well.
\end{proof}

\subsection{Farrell-Jones conjecture}
\label{s:FJC}

For a discrete group $\Gamma$ the \emph{Farrell-Jones Conjecture} asserts that the $K$-theoretic (resp.\ $L$-theoretic) \emph{assembly map}
\begin{align*}
\label{e:FJK}
H_n^\Gamma(E_{\mathcal{VCY}}(\Gamma);{\bf K}_R) \to K_n(R\Gamma)\\
(\mathrm{resp.}\; H_n^\Gamma(E_{\mathcal{VCY}}(\Gamma);{\bf L}_R^{\langle -\infty \rangle}) \to L_n^{\langle -\infty \rangle}(R\Gamma))
\end{align*}
is an isomorphism. Here, $R$ is an associative ring with a unit, $R\Gamma$
is the group ring, and $K_n(R\Gamma)$ are the algebraic $K$--groups of
$R\Gamma$. By $E_{\mathcal{VCY}}(\Gamma)$ we denote the classifying space for
the family of virtually cyclic subgroups of $\Gamma$, and ${\bf K}_R$ is
the spectrum given by algebraic $K$--theory with coefficients from $R$
(resp.\ we have the $L$-theoretic analogues) (see e.g.\
\cite{BarLuc2012,KasRup2017} for more details). We say that $\Gamma$
satisfies the \emph{Farrell-Jones conjecture with finite wreath
  products} if for any finite group $F$ the wreath product $\Gamma \wr F$
satisfies the Farrell-Jones conjecture.
\medskip

\noindent
\emph{Proof of Theorem~\ref{t:properties1}(\ref{t:properties1(7)}).} Kasprowski-R{\"u}ping
\cite{KasRup2017} showed that the Farrell-Jones conjecture with finite wreath
products holds for
groups acting geometrically on spaces with convex geodesic
bicombing. Hence our result follows from Theorem~\ref{t:helly=inj} and Theorem~\ref{t:injbicomb}.
\hfill $\square$

\subsection{Coarse Baum-Connes conjecture}
\label{s:coarseBC}

For a metric space $X$ the \emph{coarse assembly map} is a homomorphism from the coarse $K$-homology of $X$ to the $K$-theory of the Roe-algebra of $X$. The space $X$ satisfies the \emph{coarse Baum-Connes conjecture} if the coarse assembly map is an isomorphism. A finitely generated group $\Gamma$ satisfies the {coarse Baum-Connes conjecture} if the conjecture holds for $\Gamma$ seen as a metric space with a word metric given by a finite
generating set. Equivalently, the conjecture holds for $\Gamma$ if a metric space (equivalently: every metric space) acted geometrically upon by $\Gamma$ satisfies the conjecture.

\medskip

\noindent
\emph{Proof of Theorem~\ref{t:properties1}(\ref{t:properties1(8)}).}
Fukaya-Oguni \cite{FukayaOguni2019} introduced the notion of \emph{geodesic coarsely convex} space, and
proved that the coarse Baum-Connes conjecture holds for such spaces. A geodesic coarsely convex space is a metric space
with a coarse version of a bicombing satisfying some coarse convexity condition. In particular, metric spaces
with a convex bicombing -- hence all proper injective metric spaces (Theorem~\ref{t:injbicomb}) -- are geodesic coarsely convex spaces.
Therefore, our result follows from Theorem~\ref{t:helly=inj}.
\hfill $\square$

\subsection{Asymptotic cones}\label{s:AsCones}

\noindent
In this section, we are interested in asymptotic cones of Helly groups. More precisely, we prove Theorem~\ref{t:properties1}(\ref{t:properties1(9)}). Before turning to the proof, let us begin with a few definitions.

\medskip \noindent
An \emph{ultrafilter} $\omega$ over a set $S$ is a collection of subsets of $S$ satisfying the following conditions:
\begin{itemize}
	\item $\emptyset \notin \omega$ and $S \in \omega$;
	\item for every $A,B \in \omega$, $A \cap B \in \omega$;
	\item for every $A \subset S$, either $A\in \omega$ or $A^c \in \omega$.
\end{itemize}
Basically, an ultrafilter may be thought of as a labelling of the subsets of $S$ as ``small'' (if they do not belong to $\omega$) or ``big'' (if they belong to $\omega$). More formally, notice that the map
$$\left\{ \begin{array}{ccc} \mathfrak{P}(S) & \to & \{ 0,1\} \\ A & \mapsto & \left\{ \begin{array}{cl} 0 & \text{if $A \notin \omega$} \\ 1 & \text{if $A \in \omega$} \end{array} \right. \end{array} \right.$$
defines a finitely additive measure on $S$.

\medskip \noindent
The easiest example of an ultrafilter is the following. Fixing some $s \in S$, set $\omega= \{ A \subset S: s \in A \}$. Such an ultrafilter is called \emph{principal}. The existence of non-principal ultrafilters is assured by Zorn's lemma; see \cite[Section 3.1]{KapovichLeebCones} for a brief explanation.

\medskip \noindent
Now, fix a metric space $(X,d)$, a non-principal ultrafilter $\omega$ over $\mathbb{N}$, a \emph{scaling sequence} $\epsilon= (\epsilon_n)$ satisfying $\epsilon_n \to 0$, and a sequence of basepoints $o=(o_n) \in X^{\mathbb{N}}$. A sequence $(r_n) \in \mathbb{R}^{\mathbb{N}}$ is \emph{$\omega$-bounded} if there exists some $M \geq 0$ such that $\{ n \in \mathbb{N}: |r_n| \leq M \} \in \omega$ (i.e., if $|r_n| \leq M$ for ``$\omega$-almost all $n$''). Set
$$B(X,\epsilon,o) = \{ (x_n) \in X^{\mathbb{N}}: \text{$(\epsilon_n \cdot d(x_n,o_n))$ is $\omega$-bounded} \}.$$
We may define a pseudo-distance on $B(X,\epsilon,o)$ as follows. First, we say that a sequence $(r_n) \in \mathbb{R}^{\mathbb{N}}$ \emph{$\omega$-converges} to a real $r \in \mathbb{R}$ if, for every $\epsilon>0$, $\{ n \in \mathbb{N}: |r_n-r| \leq \epsilon \} \in \omega$. If so, we write $r= \lim\limits_\omega r_n$. It is worth noticing that an $\omega$-bounded sequence of $\mathbb{R}^\mathbb{N}$ always $\omega$-converges; see \cite[Section 3.1]{KapovichLeebCones} for more details. Then, our pseudo-distance is
$$\left\{ \begin{array}{ccc} B(X,\epsilon,o)^2 & \to & [0,+ \infty) \\ (x,y) & \mapsto & \lim\limits_\omega \epsilon_n \cdot d(x_n,y_n) \end{array} \right.$$
Notice that the previous $\omega$-limit always exists since the sequence under consideration is $\omega$-bounded.

\begin{definition}
The \emph{asymptotic cone} $\mathrm{Cone}_\omega(X,\epsilon,o)$ of $X$ is the metric space obtained by quotienting $B(X,\epsilon,o)$ by the relation: $(x_n) \sim (y_n)$ if $d\left( (x_n),(y_n) \right)=0$.
\end{definition}

\noindent
The picture to keep in mind is that $(X, \epsilon_n \cdot d)$ is a sequence of spaces we get from $X$ by ``zooming out'', and the asymptotic cone is the ``limit'' of this sequence. Roughly speaking, the asymptotic cones of a metric space are asymptotic pictures of the space. For instance, any asymptotic cone of $\mathbb{Z}^2$, thought of as the infinite grid in the plane, is isometric to $\mathbb{R}^2$ endowed with the $\ell_1$-metric; and the asymptotic cones of a simplicial tree (and more generally of any Gromov-hyperbolic space) are real trees.

\medskip \noindent
Because quasi-isometric metric spaces have bi-Lipschitz-homeomorphic asymptotic cones \cite[Proposition 3.12]{KapovichLeebCones}, one can define asymptotic cones of finitely generated groups up to bi-Lipschitz homeomorphism by looking at word metrics associated to finite generating sets.

\medskip \noindent
We are now ready to turn to Theorem~\ref{t:properties1}(\ref{t:properties1(9)}), which will be a consequence of the following statement:

\begin{proposition}\label{prop:AsCone}
Let $(X,d)$ be a finite dimensional proper injective metric space. Then its asymptotic cones are contractible.
\end{proposition}

\begin{proof}
Let $\sigma : X \times X \times [0,1]$ denote the combing provided by Theorem \ref{t:injbicomb}. Fix a non-principal ultrafilter $\omega$, a sequence of basepoints $o=(o_n)$ and a sequence of scalings $\epsilon= (\epsilon_n)$. For every point $x=(x_n) \in \mathrm{Cone}_\omega(X,o,\epsilon)$ and every $t \in [0,1]$, let $\rho(t,x)$ denote $(\sigma(o_n,x_n,t))$. Notice that, because $\sigma$ is geodesic, $\rho(t,x)$ defines a point of $\mathrm{Cone}_\omega(X, o,\epsilon)$. Also, because $\sigma$ is convex, the map
$$\rho : \left\{ \begin{array}{ccc} [0,1] \times \mathrm{Cone}_\omega(X,o,\epsilon) & \to & \mathrm{Cone}_\omega(X,o,\epsilon) \\ (t,x) & \mapsto & \rho(t,x) \end{array} \right.$$
is continuous. In other words, $\rho$ defines a retraction of $\mathrm{Cone}_\omega(X,o,\epsilon)$ to the point $o$.
\end{proof}

\begin{proof}[Proof of Theorem~\ref{t:properties1}(\ref{t:properties1(9)}).]
Let $\Gamma$ be a group acting geometrically on a Helly graph $G$. As a consequence of Theorem~\ref{t:helly=inj}, $\Gamma$ acts geometrically on the injective hull $E(G)$ of $G$, which is a finite dimensional proper injective metric space. As every asymptotic cone of $\Gamma$ must be bi-Lipschitz homeomorphic to an asymptotic cone of $E(G)$, the desired conclusion follows from Proposition \ref{prop:AsCone}.
\end{proof}

 \section{Biautomaticity of Helly groups}
\label{s:biautomatic}

Biautomaticity is a strong property implying numerous algorithmic and
geometric features of a group \cites{ECHLPT,BrHa}.  Sometimes the fact
that a group acting on a space is biautomatic may be established from
the geometric and combinatorial properties of the space. For example,
one of the important and nice results about \catz cube complexes is a
theorem by Niblo and Reeves \cite{NiRe} stating that the groups acting
geometrically on such complexes are biautomatic. Januszkiewicz
and {\'S}wi{\c{a}}tkowski \cite{JS} established a similar result for
groups acting on systolic complexes. It is also well-known that
hyperbolic groups are biautomatic \cite{ECHLPT}.
{\'S}wi{\c{a}}tkowski \cite{Swiat} presented a general framework of
locally recognized path systems in a graph $G$ under which proving
biautomaticity of a group acting geometrically on $G$ is reduced to proving local
recognizability and the $2$--sided fellow traveler property for some
paths. 
\medskip

In this section we use a different meaning of the term `bicombing'. Here the bicombing is a combinatorial
object that should not be confused with the (continuous) geodesic bicombing from Subsection~\ref{s:bicombing}.

\subsection{Main results} In this section, similarly to the results of
\cite{NiRe} for \catz cube complexes, of \cite{JS} for systolic
complexes, and of \cite{CCHO} for swm-graphs, we define the concept of
normal clique-path and prove the existence and uniqueness of normal
clique-paths in all Helly graphs $G$. These clique-paths can be viewed
as usual paths in the $1$--skeleton of the face complex $F(X(G))$ of
$X(G)$ and also give rise to paths in the $1$--skeleton $\beta(G)$ of the
first barycentric subdivision of $X(G)$. From their definition, it
follows that the sets of normal clique-paths are locally recognized
sensu \cite{Swiat}. Moreover, we prove that they satisfy the
$2$--sided fellow traveler property. As a consequence, we conclude
that groups acting geometrically on Helly graphs are biautomatic.

\begin{theorem} \label{biautomatic}
The set of normal clique-paths between all vertices of a Helly graph
$G$ defines a regular geodesic bicombing in $\beta(G)$.  Consequently,
a group acting geometrically on a Helly graph is biautomatic.
\end{theorem}

\begin{remark}
  A natural generalization of this theorem would be to prove that
  injective groups (i.e., groups acting geometrically on injective
  metric spaces) are biautomatic. Recently, Hugues and
  Valiunas~\cite[Corollary D]{HuVa} proved this is not the case: they
  constructed an injective group that is not biautomatic and thus not
  Helly.
\end{remark}

\subsection{Bicombings and biautomaticity}

We continue by recalling the definitions of (geodesic) bicombing and
biautomatic group \cites{ECHLPT,BrHa}.  Let $G=(V,E)$ be a graph and
suppose that $\Gamma$ is a group acting geometrically by automorphisms
on $G$.  These assumptions imply that the graph $G$ is locally finite
and that the degrees of the vertices of $G$ are uniformly bounded.
Denote by ${\mathcal P}(G)$ the set of all paths of $G$.  A \emph{path
  system} $\mathcal P$ \cite{Swiat} is any subset of
${\mathcal P}(G)$.  The action of $\Gamma$ on $G$ induces the action
of $\Gamma$ on the set ${\mathcal P}(G)$ of all paths of $G$.  A path
system ${\mathcal P}\subseteq {\mathcal P}(G)$ is called
$\Gamma$--\emph{invariant} if $g\cdot \gamma \in \mathcal P$, for all
$g\in \Gamma$ and $\gamma \in \mathcal P$.

Let $[0,n]^*$ denote the set of integer points
from the segment $[0,n].$ Given a path $\gamma$ of length $n=|\gamma|$ in $G$, we can parametrize it  and denote it by $\gamma:[0,n]^*\rightarrow V(G)$. It will be convenient to extend
 $\gamma$ over $[0,\infty]$ by setting $\gamma(i)=\gamma(n)$ for any $i>n$.  A path system $\mathcal P$ of a graph $G$ is said to satisfy the
 \emph{2-sided fellow traveler property} if there are constants $C>0$ and $D\ge 0$ such that for any two paths $\gamma_1,\gamma_2\in \mathcal P$, the following inequality holds  for all natural $i$:
$$d_G(\gamma_1(i),\gamma_2(i))\le C\cdot \max\{ d_G(\gamma_1(0),\gamma_2(0)),d_G(\gamma_1(\infty),\gamma_2(\infty))\}+D.$$
A path system $\mathcal P$ is \emph{complete} if any two vertices are endpoints of some path in $\mathcal P$.
A \emph{bicombing} of a graph $G$ is a complete path system $\mathcal P$ satisfying the $2$--sided fellow traveler property.   If all paths in
the bicombing ${\mathcal P}$ are shortest paths of $G$, then
${\mathcal P}$ is called a \emph{geodesic bicombing}.

We recall here quickly the definition of a biautomatic structure for a group. Details can be found in  \cites{ECHLPT,BrHa,Swiat}.
Let $\Gamma$ be a group generated by a finite set $S$. A \emph{language} over $S$ is some set of words in $S\cup S^{-1}$ (in the
free monoid $(S\cup S^{-1})^{\ast}$).
A language over $S$ defines a $\Gamma$--invariant path system in the Cayley graph Cay$(\Gamma,S)$.
A language is \emph{regular} if it is accepted by some finite state automaton.
A \emph{biautomatic structure} is a pair $(S,\mathcal L)$, where $S$ is as above, $\mathcal L$ is a regular language over $S$, and
the associated path system in Cay$(\Gamma,S)$ is a bicombing. A group is
\emph{biautomatic} if it admits a biautomatic structure.
In what follows we use specific conditions implying biautomaticity for groups acting geometrically on graphs. The method, relying on the notion of locally recognized path system, was developed by {\'S}wi{\c{a}}tkowski \cite{JS}.

Let $G$ be a graph and let $\Gamma$ be a group acting geometrically on
$G$.  Two paths $\gamma_1$ and $\gamma_2$ of $G$ are
$\Gamma$-\emph{congruent} if there is $g\in \Gamma$ such that
$g\cdot\gamma_1=\gamma_2$. Denote by ${\mathcal S}_k$ the set of
$\Gamma$-congruence classes of paths of length $k$ of $G$.  Since
$\Gamma$ acts geometrically on $G$, the sets ${\mathcal S}_k$ are
finite for any natural $k$. For any path $\gamma$ of $G$, denote by
$[\gamma]$ its $\Gamma$-congruence class.

For a subset $R\subset {\mathcal S}_k$, let ${\mathcal P}_R$ be the path system in $G$ consisting of all paths $\gamma$ satisfying the following two conditions:
\begin{itemize}
\item[(1)] if $|\gamma|\ge k$, then $[\eta]\in R$ for any subpath $\eta$ of length $k$ of $\gamma$;
\item[(2)] if $|\gamma|<k$, then $\gamma$ is a prefix of some path $\eta$ such that $[\eta]\in R$.
\end{itemize}

A path system $\mathcal P$ in $G$ is $k$--\emph{locally recognized} if for some $R\subset {\mathcal S}_k$, we have ${\mathcal P}={\mathcal P}_R$, and $\mathcal P$ is \emph{locally recognized}
if it is $k$--locally recognized for some $k$. The following result of  {\'S}wi{\c{a}}tkowski  \cite{Swiat} provide sufficient conditions for biautomaticity in terms of local recognition
and bicombing.

\begin{theorem} \label{swiat} \cite[Corollary 7.2]{Swiat} \label{th:swiat} Let $\Gamma$ be group acting geometrically on a graph $G$ and let $\mathcal P$ be a path system in
$G$ satisfying the following conditions:
\begin{itemize}
\item[(1)] $\mathcal P$ is locally recognized;
\item[(2)] there exists $v_0\in V(G)$ such that any two vertices from the orbit $\Gamma \cdot v_0$ are connected by a path from $\mathcal P$;
\item[(3)] $\mathcal P$ satisfies the $2$--sided fellow traveler property.
\end{itemize}
Then $\Gamma$ is biautomatic.
\end{theorem}

\subsection{Normal clique-paths in Helly-graphs}
For a set $S$ of vertices of a graph $G=(V,E)$ and an integer $k\ge
0$, let $B^*_k(S):=\bigcap_{s \in S}B_k(s)$. In particular, if $S$ is
a clique, then $B_1^*(S)$ is the union of $S$ and the set of vertices
adjacent to all vertices in $S$.  Notice also that if $S \subseteq
S'$, then $B^*_k(S) \supseteq B^*_k(S')$.  For two cliques $\tau$ and
$\sigma$ of $G$, let $\dm(\tau,\sigma):=\max \{ d(t,s): t\in \tau,
s\in \sigma\}$. We recall also the notation $d(\tau,\sigma)=\min \{
d(t,s): t\in \tau, s\in \sigma\}$ for the standard distance between
$\tau$ and $\sigma$.  We  say that two cliques $\sigma, \tau$ of a
graph $G$ are at \emph{uniform-distance} $k$ (notation
\du{\sigma}{\tau}{k}) if $d(s,t)=k$ for any $s\in \sigma$ and any
$t\in \tau$. Equivalently, \du{\sigma}{\tau}{k} if and only if
$\dm(\tau,\sigma)=d(\tau,\sigma)=k$.

Given two cliques $\sigma, \tau$ of $G$ with $\dm(\tau,\sigma)=k\geq
2$, let $\hR_\tau(\sigma):=B^*_k(\tau)\cap B^*_1(\sigma)$ and let
$f_\tau(\sigma):= B^*_{k-1}(\tau)\cap B^*_1(\hR_\tau(\sigma))$. The following observations can be helpful for the understanding of these notions:
\begin{itemize}
\item  $\hR_\tau(\sigma)$ is the union of the maximal cliques of $B_k^*(\tau)$ that contain $\sigma$.
\item $B^*_1(\hR_\tau(\sigma))$ is the intersection of the maximal
  cliques of $B_k^*(\tau)$ that contain $\sigma$.
\item $f_{\tau}(\sigma)$ is the intersection of $B_{k-1}^*(\tau)$ and
  the maximal cliques of $B_k^*(\tau)$ that contain $\sigma$.
\end{itemize}

Since $G$ is a Helly graph, the set $f_\tau(\sigma)$ is non-empty and we
 call it the \emph{imprint} of $\sigma$ with respect
to $\tau$.  Note that since $\sigma$ is a clique, we have $\sigma
\subseteq \hR_\tau(\sigma)$ and thus we also have $f_\tau(\sigma)
\subseteq \hR_\tau(\sigma)$. Note also that each vertex  in $f_{\tau} (\sigma)$ is
adjacent to all other vertices in $\hR_{\tau} (\sigma)$, whence $\hR_{\tau}(\sigma) \subseteq B_1^*(f_\tau (\sigma))$ and $f_\tau(\sigma)$ is a clique.

\begin{lemma} \label{uniform} For any two cliques $\sigma, \tau$ of a
  Helly graph $G$ such that $\dm(\tau,\sigma)=k\geq 2$, the imprint
  $f_\tau(\sigma)$ is a non-empty clique such that
  $\dm(\tau,f_{\tau}(\sigma)) = k-1 = \dm(\tau,s')$ for any
  $s' \in f_\tau(\sigma)$. Moreover, if $\du{\sigma}{\tau}{k}$, then
  $\du{f_\tau(\sigma)}{\tau}{k-1}$.
\end{lemma}

\begin{proof}
Note that by definition, $f_\tau(\sigma) \subseteq B^*_{k-1}(\tau)$.
Note also that for any $r, r' \in \hR_\tau(\sigma)$, $\sigma \subseteq
B_1(r) \cap B_1(r')$. Moreover, for any $r \in \hR_\tau(\sigma)$ and
any $t \in \tau$, $d(r,t) \leq k$ and thus $B_{k-1}(t)\cap B_1(r)
\neq \emptyset$. Note also that since $\tau$ is a clique and $k \geq
2$, $\tau \subseteq B^*_{k-1}(\tau)$.  Consequently, since $G$ is a
Helly graph, $f_\tau(\sigma) \neq \emptyset$. Since
$f_\tau(\sigma)\cup \sigma \subseteq \hR_\tau(\sigma)$ and each vertex of $f_\tau(\sigma)$ is
adjacent to all other vertices of $\hR_\tau(\sigma)$, necessarily
$f_\tau(\sigma)\cup \sigma$ is a clique. Therefore, for any $t \in
\tau$, $s \in \sigma$ such that $d(t,s) = \dm(\tau,\sigma) = k$, and
any $s' \in f_\tau(\sigma)$, we have $d(s',t) \geq d(s,t) - d(s,s') =
k-1$. Since $s' \in f_\tau(\sigma) \subseteq B^*_{k-1}(\tau)$, we have
$d(s',t) = k-1$. Thus, $\dm(\tau,f_\tau(\sigma)) = k-1$ and
$\du{f_\sigma(\tau)}{\tau}{k-1}$ when $\du{\sigma}{\tau}{k}$.
\end{proof}

\begin{lemma}\label{lem-ncp-inclusion}
  Consider three cliques $\sigma, \sigma', \tau$ of a Helly graph $G$
  such that $\dm(\tau,\sigma) = \dm(\tau,\sigma') = k \geq 2$. If
  $\sigma' \subseteq \sigma$, then $\hR_\tau(\sigma) \subseteq
  \hR_\tau(\sigma')$ and $f_\tau(\sigma') \subseteq f_\tau(\sigma)$.
  In particular, if $\du{\sigma}{\tau}{k}$, then for every $s \in
  \sigma$, we have $f_\tau(s) \subseteq f_\tau(\sigma)$.
\end{lemma}

\begin{proof}
  Recall that $\hR_\tau(\sigma):=B^*_k(\tau)\cap B^*_1(\sigma)$ and
  $\hR_\tau(\sigma'):=B^*_k(\tau)\cap B^*_1(\sigma')$. Since
  $\sigma' \subseteq \sigma$, we have $B^*_1(\sigma) \subseteq
  B^*_1(\sigma')$ and thus $\hR_\tau(\sigma) \subseteq
  \hR_\tau(\sigma')$. Consequently, $B^*_1(\hR_\tau(\sigma'))
  \subseteq B^*_1(\hR_\tau(\sigma))$ and thus $f_\tau(\sigma') =
  B^*_{k-1}(\tau)\cap B^*_1(\hR_\tau(\sigma')) \subseteq
  B^*_{k-1}(\tau)\cap B^*_1(\hR_\tau(\sigma)) = f_\tau(\sigma)$.
\end{proof}

A sequence of cliques $(\sigma_0, \sigma_1, \ldots, \sigma_k)$ of a
Helly graph $G$ is called a \emph{normal clique-path} if the following
local  conditions hold:
\begin{enumerate}[(1)]
  \item for any $0 \leq i \leq k-1$, $\sigma_i$ and
    $\sigma_{i+1}$ are disjoint and $\sigma_i \cup \sigma_{i+1}$ is
    a clique of $G$,
  \item for any $1 \leq i \leq k-1$, $\sigma_{i-1}$ and
    $\sigma_{i+1}$ are at uniform-distance 2,
  \item for any $1 \leq i \leq k-1$, $\sigma_i =
    f_{\sigma_{i-1}}(\sigma_{i+1})$.
\end{enumerate}

Notice that if $k \geq 2$, then condition (1) follows from conditions (2) and (3).

\begin{theorem}[{Normal clique-paths}]\label{th-ncp}
For any pair $\tau,\sigma$ of cliques of a Helly graph $G$ such that
$\du{\sigma}{\tau}{k}$, there exists a unique normal clique-path
$\gamma_{\tau\sigma} = (\tau = \sigma_0, \sigma_1, \sigma_2, \ldots,
\sigma_k=\sigma)$, whose cliques are given by
\begin{equation}\label{eqn:given1}
\sigma_{i} = f_\tau(\sigma_{i+1}) \mbox{ for each } i=k-1,\ldots,
2,1,
\end{equation}
and any sequence of vertices $P = (s_0, s_1, \ldots, s_k)$ with $s_i \in \sigma_i$ for $0 \leq i \leq k$ is a shortest path from
$s_0$ to $s_k$. In particular, any two vertices $p,q$ of $G$ are connected by a unique normal
clique-path $\gamma_{pq}$.
\end{theorem}

\begin{proof} We first prove that $\gamma_{\tau\sigma}$ is a normal clique-path.
The proof  is based on the following result.

\begin{lemma}\label{lem-ncp-gl} Let $\sigma, \sigma', \sigma''$, and $\tau$ be four cliques of a Helly graph $G$
such that $\du{\sigma}{\tau}{k}$ with $k\ge 3$, $\sigma' \subseteq
f_\tau(\sigma)$, and $\sigma'' \subseteq f_\tau(\sigma')$. Then
$f_\tau(\sigma) = f_{\sigma''}(\sigma)$.
\end{lemma}

\begin{proof}  Note that our conditions and Lemma \ref{uniform} imply
that $\du{\sigma'}{\tau}{k-1}$, $\du{\sigma''}{\tau}{k-2}$, and
$\du{\sigma}{\sigma''}{2}$.

We first show that $\hR_{\sigma''}(\sigma) = \hR_\tau(\sigma)$. Recall
that $\hR_{\tau}(\sigma) = B^*_{k}(\tau) \cap B^*_1(\sigma)$ and
$\hR_{\sigma''}(\sigma)=B^*_2(\sigma'')\cap B^*_1(\sigma)$. Since
$\du{\tau}{\sigma''}{k-2}$, we have
$B^*_2(\sigma'')\subseteq B^*_k(\tau)$.  Consequently,
$\hR_{\sigma''}(\sigma) \subseteq \hR_\tau(\sigma)$.  Conversely, by
the definition of $\sigma''$, we have
$\sigma' \subseteq B^*_1(\sigma'')$. Indeed, since
$\sigma'' \subseteq f_\tau(\sigma')$, we have
$B^*_1(\sigma'') \supseteq B^*_1(f_\tau(\sigma')) \supseteq
\hR_\tau(\sigma') \supseteq \sigma'$.  Since
$\sigma' \subseteq f_\tau(\sigma)$, we have
$\hR_\tau(\sigma) \subseteq B_1^*(f_\tau(\sigma)) \subseteq
B_1^*(\sigma') \subseteq B_2^*(\sigma'')$ where the last containment
follows from $\sigma'' \subseteq f_{\tau}(\sigma')$.  Consequently,
$\hR_\tau(\sigma) = B^*_{k}(\tau) \cap B^*_1(\sigma) \subseteq
B^*_2(\sigma'')\cap B^*_1(\sigma) = \hR_{\sigma''}(\sigma)$, and thus
$\hR_{\sigma''}(\sigma) = \hR_\tau(\sigma)$.

  Set $\hR: = \hR_{\sigma''}(\sigma) = \hR_\tau(\sigma)$. Set also
  $\varrho': = f_{\sigma''}(\sigma)$ and $\nu' :=  f_{\tau}(\sigma)$. Recall
  that $\nu' =f_\tau(\sigma) = B^*_{k-1}(\tau)\cap B^*_1(\hR)$ and
  $\varrho' = f_{\sigma''}(\sigma)=B^*_1(\sigma'')\cap B^*_1(\hR)$.
  Since $\du{\tau}{\sigma''}{k-2}$, we have $B^*_1(\sigma'')\subseteq
  B^*_{k-1}(\tau)$ and thus $\varrho' \subseteq \nu'$.
  Conversely, since $\nu' \subseteq \hR_\tau(\nu') =
  B^*_1(\nu') \cap B_{k-1}(\tau) \subseteq B^*_1(\sigma') \cap
  B_{k-1}(\tau)= \hR_\tau(\sigma')$, we have $\nu' \subseteq
  B^*_1(\sigma'')$ by definition of $\sigma''$. Consequently,
  $\nu' \subseteq  B^*_1(\sigma'')\cap B^*_1(\hR) =
  \varrho'$. Therefore $\nu' = \varrho'$ and the lemma holds.
\end{proof}

To prove that $\gamma_{\tau\sigma}$ is a normal clique-path we proceed
by induction on $k$.  If $k \leq 2$, there is nothing to prove. Assume
now that $k \geq 3$. Since $\du{\tau}{\sigma_k}{k}$, $\sigma_{k-1} =
f_\tau(\sigma_k)$, and $\sigma_{k-2} = f_\tau(\sigma_{k-1})$, we have
that $\du{\tau}{\sigma_{k-1}}{k-1}$, $\du{\tau}{\sigma_{k-2}}{k-2}$,
and $\du{\sigma_{k-2}}{\sigma_k}{2}$. By the induction hypothesis,
$(\sigma_0=\tau, \sigma_1, \sigma_2, \ldots, \sigma_{k-1})$ is a
normal clique-path. Applying Lemma~\ref{lem-ncp-gl} with $\sigma=
\sigma_k$, $\sigma' = \sigma_{k-1}$ and $\sigma''= \sigma_{k-2}$, we
have that $\sigma_{k-1} = f_{\sigma_{k-2}}(\sigma_k)$ and thus
$\gamma_{\tau\sigma}$ is a normal clique-path as well.

\medskip

We now prove that an arbitrary normal clique-path
$\gamma'_{\tau\sigma}=(\tau = \varrho_0, \varrho_1, \varrho_2,
\ldots,\varrho_l=\sigma)$ coincides with $\gamma_{\tau\sigma}$. In fact,
we prove this result under a weaker assumption than $\du{\sigma}{\tau}{k}$.

\begin{proposition}\label{prop-ncp-lg} Let $\sigma,\tau$ be two cliques of a Helly graph $G$
and let $k$ be an integer such that for every $s \in \sigma$, $\dm(s,\tau) = k$.
Then any normal clique-path $\gamma'_{\tau\sigma}=(\tau = \varrho_0, 
  \varrho_1, \varrho_2, \ldots,\varrho_l=\sigma)$ coincides with
  $\gamma_{\tau\sigma} = (\tau = \sigma_0, \sigma_1, \sigma_2, \ldots, \sigma_k=\sigma)$,
whose cliques are given by (\ref{eqn:given1}).
\end{proposition}

\begin{proof}
The proof of the proposition is based on the following result.

\begin{lemma}\label{lem-ncp-lg1}
  Let $\varrho, \varrho', \varrho''$, and $\tau$ be four cliques of a
  Helly graph $G$ such that
  $\dm(\tau,\varrho)= 1 +\dm(\tau,\varrho')=:k\ge 3$,
  $\dm(\varrho, \varrho'') \geq 2$, $\varrho'=f_{\varrho''}(\varrho)$
  and $\varrho'' \subseteq f_\tau(\varrho')$. Then
  $\varrho' = f_{\tau}(\varrho)$.
 \end{lemma}

 \begin{proof}
   Let $\sigma' = f_\tau(\varrho)$ and note that our conditions and
   Lemma \ref{uniform} imply that $\dm(\tau,\sigma')=1+ \dm(\tau,\varrho'')=k-1$
   and that  $\dm(\varrho,\varrho'')=2$.

   We first show that $\hR_{\varrho''}(\varrho) = \hR_\tau(\varrho)$.
   Recall that $\hR_{\tau}(\varrho) = B^*_{k}(\tau) \cap
   B^*_1(\varrho)$ and $\hR_{\varrho''}(\varrho) = B^*_2(\varrho'')
   \cap B^*_1(\varrho)$.  Since  $\dm(\tau,\varrho'')=k-2$,
   necessarily $B^*_2(\varrho'') \subseteq
   B^*_k(\tau)$, and consequently, $\hR_{\varrho''}(\varrho) \subseteq
   \hR_\tau(\varrho)$.
In particular, note that $\varrho' \subseteq
   \hR_{\varrho''}(\varrho) \subseteq \hR_\tau(\varrho)$. Consequently,
   $\sigma' \subseteq B^*_1(\hR_\tau(\varrho)) \subseteq
   B^*_1(\varrho')$.
Since $\sigma' \subseteq B^*_{k-1}(\tau)$, we have $\sigma'
   \subseteq B^*_{k-1}(\tau) \cap B^*_1(\varrho') =
   \hR_\tau(\varrho')$. Therefore, by the definition of $\varrho'' \subseteq
   f_\tau(\varrho')$, we have $\sigma' \subseteq B^*_1(\varrho'')$.
Consequently, $B^*_1(\sigma') \subseteq B^*_2(\varrho'')$ and thus
   $\hR_\tau(\varrho) \subseteq B^*_1(\sigma') \subseteq
   B^*_2(\varrho'')$. Therefore $\hR_\tau(\varrho) \subseteq
   B^*_2(\varrho'') \cap B^*_1(\varrho) = \hR_{\varrho''}(\varrho)$ and
   thus $\hR_\tau(\varrho) = \hR_{\varrho''}(\varrho)$.

   Let $\hR = \hR_\tau(\varrho) = \hR_{\varrho''}(\varrho)$ and recall
   that $\varrho' = f_{\varrho''}(\varrho) = B^*_1(\varrho'') \cap
   B^*_1(\hR)$ and that $\sigma' = f_{\tau}(\varrho) = B^*_{k-1}(\tau)
   \cap B^*_1(\hR)$. Since $\sigma' \subseteq B^*_1(\varrho'')$, necessarily
   $\sigma' \subseteq \varrho'$. Conversely, since $\dm(\tau, \varrho'')=k-2$,
   necessarily $B^*_1(\varrho'') \subseteq B^*_{k-1}(\tau)$, and consequently,
   $\varrho' \subseteq \sigma'$. Therefore $\varrho' = \sigma'$ and the
   lemma holds.
\end{proof}

  We prove the proposition by induction on the length $l$ of the normal clique-path
  $\gamma'_{\tau\sigma}$. If $l \leq 2$, there is nothing to
  prove. Assume now that $l \geq 3$ and let $k = \dm(\tau,\sigma)$.

  Suppose first that $\dm(\tau,\varrho_{l-1}) = k-1$. Since
  $\varrho_{l-1}\cup \sigma$ is a clique and since $\dm(s,\tau) = k$
for every $s \in \sigma$, necessarily $\dm(p',\tau) = k-1$
for every $p' \in \varrho_{l-1}$.  By the
  induction hypothesis, the clique-path $\gamma'_{\tau\varrho_{l-1}} =
  (\tau = \varrho_0, \varrho_1, \varrho_2, \ldots,\varrho_{l-1})$
  coincides with $\gamma_{\tau\varrho_{l-1}}$. Consequently, $l = k$
  and $\varrho_{l-2} = f_\tau(\varrho_{l-1})$. Applying
  Lemma~\ref{lem-ncp-lg1} with $\varrho = \sigma$, $\varrho' =
  \varrho_{l-1}$ and $\varrho''=\varrho_{l-2}$, we have that
  $f_\tau(\sigma) = f_{\varrho_{l-2}}(\sigma) = \varrho_{l-1}$. Hence,
  $\gamma'_{\tau\sigma}$ and $\gamma_{\tau\sigma}$ coincide.

  Suppose now that $\dm(\tau,\varrho_{l-1}) \geq k$. Note that in this
  case, necessarily $l \geq k+1$ and so $\dm(\varrho_l, \tau) = k \leq l-1$. Consider the minimal index $i$ for
  which there exists $p \in \varrho_i$ such that $\dm(p,\tau) \leq i-1$.
Note that $i \geq 2$ since otherwise $\tau =
  \varrho_0=\{p\}$ and $\varrho_0 \cap \varrho_1 \neq \emptyset$,
  contradicting the fact that $\gamma'_{\tau\sigma}$ is a normal
  clique-path. Note also that since $\gamma'_{\tau\sigma}$ is a normal
  clique-path, we have $\du{\varrho_0}{\varrho_2}{2}$ and thus $i \geq
  3$.  By the induction hypothesis, $\gamma'_{\tau\varrho_{i-1}} = (\tau =
  \varrho_0, \varrho_1, \varrho_2, \ldots,\varrho_{i-1})$ and
  $\gamma_{\tau\varrho_{i-1}}$ coincide. In particular, this implies
  that $\varrho_{i-2} = f_\tau(\varrho_{i-1})$. Note that $p \in
  B^*_{i-1}(\tau)$ by our choice of $p$ and that $p \in
  B^*_1(\varrho_{i-1})$ since $\varrho_{i-1} =
  f_{\varrho_{i-2}}(\varrho_i)$. Consequently, $p \in
  \hR_{\tau}(\varrho_{i-1}) \subseteq B^*_1(\varrho_{i-2})$. But then
  $\varrho_{i}$ and $\varrho_{i-2}$ are not at uniform-distance $2$,
  contradicting the fact that $\gamma'_{\tau\sigma}$ is a normal
  clique-path. This finishes the proof of Proposition~\ref{prop-ncp-lg}.
\end{proof}

  To conclude the proof of Theorem~\ref{th-ncp}, consider any sequence
  $P = (s_0, s_1, \ldots, s_k)$ such that $s_i \in \sigma_i$ for $0
  \leq i \leq k$. Note that $P$ is a path since $\sigma_i \cup
  \sigma_{i+1}$ is a clique for every $0\leq i \leq k-1$, and that it
  is a shortest path since $d(s_0,s_k) = \dm(\sigma_0,\sigma_k) = k$.
\end{proof}

\subsection{Normal paths in Helly-graphs}

In this subsection, we define the notion of a normal path between any
two vertices $t$ and $s$ of a Helly graph.  Analogously to normal
clique-paths, normal paths can be characterized in a local-to-global
way, and therefore they are locally recognized.  Any two vertices
$t,s$ of $G$ can be connected by at least one normal path, and all
normal $(t,s)$-paths are hosted by the normal clique-path
$\gamma_{ts}$.

A path  $(t=s_0, s_1, \ldots, s_k=s)$ between two vertices $t$ and $s$ of a
Helly graph $G$ is called a \emph{normal path} if the following
local  conditions hold:
\begin{enumerate}[(1)]
  \item for any $1 \leq i \leq k-1$, $d(s_{i-1},s_{i+1})=2$,
  \item for any $1 \leq i \leq k-1$, $s_i\in f_{s_{i-1}}(s_{i+1})$.
\end{enumerate}

\begin{proposition}[{Normal paths}]\label{th-np}
  A path $P_{ts}=(t=s_0, s_1, \ldots, s_k=s)$ between two vertices $t$
  and $s$ of a Helly graph $G$ is a normal path if and only
  $s_i\in f_t(s_{i+1})$ for any $1 \leq i \leq k-1$. In particular,
  this implies that $P_{ts}$ is a shortest path of $G$. If
  $\gamma_{ts}=(\{t\}=\sigma_0, \sigma_1, \ldots, \sigma_k=\{s\})$ is
  the unique normal clique-path between $t$ and $s$, then for any
  normal path $P'_{ts}=(t=s_0, s_1, \ldots, s_k=s)$, we have
  $s_i \in \sigma_i$ for $0 \leq i \leq k$.
\end{proposition}

\begin{proof}
  The proof of the first statement of the proposition is similar to
  the proof of Theorem~\ref{th-ncp}. We first prove that $P_{ts}$ is a
  normal path. Observe that Lemma~\ref{uniform}, $P_{ts}$ is a
  shortest path of $G$. To do so, we proceed by induction on the
  distance $k=d(t,s)$.  If $k \leq 2$, there is nothing to
  prove. Assume now that $k \geq 3$. Since $d(t,s_k) = k$,
  $s_{k-1} \in f_t(s_k)$, and $s_{k-2}\in f_t(s_{k-1})$, we have
  $d(t,s_{k-1}) = k-1$ and $d(t,s_{k-2}) = k-2$.  By the induction
  hypothesis, $(s_0=t, s_1, s_2, \ldots, s_{k-1})$ is a normal
  path. Applying Lemma~\ref{lem-ncp-gl} with $\sigma= \{s_k\}$,
  $\sigma' = \{s_{k-1}\}$, $\sigma''= \{s_{k-2}\}$, and
  $\tau = \{t\}$, we conclude that
  $s_{k-1} \in f_t(s_k) = f_{s_{k-2}}(s_k)$ and thus $P_{ts}$ is a
  normal path as well.

We now prove that any normal path $P'_{ts}=(t = p_0,
p_1, \ldots,p_l=s)$ is a shortest path of $G$ and that
$p_i \in f_t(p_{i+1})$ for every $1 \leq
i \leq l$. To do so, we proceed by induction on the length $l$ of
$P'_{ts}$. If $l \leq 2$, there is nothing to
prove. Assume now that $l \geq 3$ and let $k = d(t,p_l)$. By the induction hypothesis
applied to the normal path $P'_{tp_{l-1}} = (t = p_0, p_1, \ldots,p_{l-1})$,
$P'_{tp_{l-1}}$ is a shortest path of $G$ and
we have $p_i\in f_t(p_{i+1})$ for every $1 \leq i \leq
l-2$. In particular, $d(t,p_{l-1})=l-1$.

Suppose first that $d(t,p_{l-1}) = k-1$.   Then $l = k$, therefore $P'_{ts}$ is a shortest path.
Since $p_{l-2} \in f_\tau(p_{l-1})$, applying Lemma~\ref{lem-ncp-lg1} with $\varrho =\{s\}$,
$\varrho' = f_{p_{l-2}}(s)$, and $\varrho''=\{p_{l-2}\}$, we have
that $f_t(s) = f_{p_{l-2}}(s)$, and thus $p_{l-1} \in
f_{p_{l-2}}(s) = f_t(s)$. Consequently, we have $p_i \in f_t(p_{i+1})$
for every $1 \leq i \leq l$ and the proposition holds in
this case. Suppose now that $l-1=d(t,p_{l-1}) \geq k$, i.e., $l \geq k+1$.
By the induction hypothesis applied to the path $P'_{tp_{l-1}}$,
we have $p_{l-2} \in f_t(p_{l-1})$. Note that $p_l \in B_{l-1}(t)$
because $d(t,p_l)=k\le l-1$ and that $p_l \in B_1(p_{l-1})$.
Consequently, $p_l \in \hR_{t}(p_{l-1}) \subseteq
B_1(p_{l-2})$. But then $d(p_l,p_{l-2}) \leq 1$, contradicting the
fact that $P'_{ts}$ is a normal path. This ends the proof of the
first statement of the proposition.

Consider now the normal clique-path $\gamma_{ts}=(\{t\}=\sigma_0,
\sigma_1, \ldots, \sigma_k=\{s\})$ between two vertices $t$ and $s$
and any normal path $P_{ts}=(t=s_0, s_1, \ldots, s_k=s)$. We show by reverse
induction on $i$ that $s_i \in \sigma_i$ for $0 \leq i \leq k$. For $i
= k$, there is nothing to prove.  Suppose now that $i < k$ and that
$s_{i+1}\in \sigma_{i+1}$.  Since $s_i\in f_t(s_{i+1})$ by the first assertion of
the proposition and since
$f_t(s_{i+1}) \subseteq f_t(\sigma_{i+1}) = \sigma_i$ by Lemma
\ref{lem-ncp-inclusion}, we have $s_i\in \sigma_i$.
\end{proof}

\begin{remark}
  The example of Figure~\ref{fig-normal-path-clique} is a Helly graph
  and contains two vertices $s, t$ such that the cliques of the normal
  clique-path $\gamma_{ts}$ contain a vertex not included in any
  normal $(t,s)$-path.
\end{remark}

\begin{figure}[h]
\begin{center}
\includegraphics{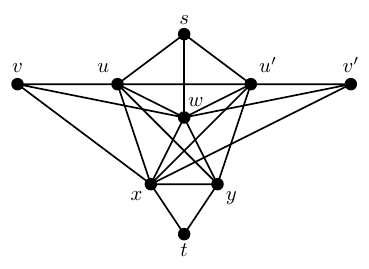}
\end{center}
\caption{In this graph, $y$ appears in a clique of the normal
  clique-path $\gamma_{ts} = (t, \{x,y\},\{u,u',w\}, s)$. However, for
  any normal path $(t=s_0, s_1, s_2, s_3=s)$, $\hR_t(s_2)$ contains
  either $v$ or $v'$ and thus $y \notin
  f_t(s_2)$.}\label{fig-normal-path-clique}
\end{figure}

\subsection{Normal (clique-)paths are fellow travelers}

\begin{proposition}\label{prop-bicombing} Let $G$ be a Helly graph.
Consider two cliques $\sigma, \tau$,  two
  vertices $p,q$ of $G$, and two integers $k' \geq k$ such that
  $\du{p}{\sigma}{k'}$, $\du{q}{\tau}{k}$, $d(\sigma, \tau) \leq 1$,
  and $d(p,q) \leq 1$. For the normal clique-paths
  $\gamma_{p\sigma}=(p=\sigma_0, \sigma_1, \ldots, \sigma_{k'}=\sigma)$
  and $\gamma_{q\tau}=(q=\tau_0, \tau_1, \ldots, \tau_k=\tau)$, we
  have $d(\sigma_i,\tau_i) \leq 1$ for every $0 \leq i \leq k$ and
  $d(\sigma_i,\tau_k) \leq 1$ for every $k \leq i \leq k'$.
\end{proposition}

\begin{proof}
  We prove the result by induction on $k'$.  If $ k' \leq 1$, there is
  nothing to prove. Assume now that $k' \geq 2$ and that the lemma
  holds for any cliques $\sigma, \tau$, any vertices $p,q$, and any
  integers $l \leq l' \leq k'-1$ such that $\du{p}{\sigma}{l'}$,
  $\du{q}{\tau}{l}$, $d(\sigma, \tau) \leq 1$, and $d(p,q) \leq 1$.

  Suppose first that $k < k'$.  Note that $k+1 \leq k' \leq k+2$ since
  $d(p,q)\leq 1$ and $d(\sigma,\tau)\le 1$. Let $s \in \sigma$ and $t
  \in \tau$ such that $d(s,t) = d(\sigma,\tau) \leq 1$. Note that
  $d(p,t) \leq d(q,t)+1 = k+1 \leq k'$. Consequently, $t \in \hR_p(s)$
  and thus $f_p(s) \subseteq B_1(t)$. Consequently, since $f_p(s)
  \subseteq f_p(\sigma) = \sigma_{k'-1}$ by
  Lemma~\ref{lem-ncp-inclusion}, we have $d(\sigma_{k'-1},\tau_k) \leq
  1$. By Lemma~\ref{uniform}, we have $\du{p}{\sigma_{k'-1}}{k'-1}$
  and thus we can apply the induction hypothesis to $\sigma_{k'-1},
  \tau, p$, and $q$. Therefore, we have $d(\sigma_i,\tau_i) \leq 1$
  for every $0 \leq i \leq k$ and $d(\sigma_i,\tau_k) \leq 1$ for
  every $k \leq i \leq k'-1$.  Since by our assumptions, we have
  $d(\sigma_{k'},\tau_k) \leq 1$, we are done.

  Suppose now that $k = k'$. By the induction hypothesis, it is enough to
  show that $d(f_p(\sigma),f_q(\tau)) \leq 1$. Consider any two
  vertices $s \in \sigma$ and $t \in \tau$ such that $d(s,t) =
  d(\sigma,\tau)$. By Lemma~\ref{lem-ncp-inclusion}, it is enough to
  show that $d(f_p(s),f_q(t)) \leq 1$.

  Assume first that $d(p,t) \leq k$ (note that we are in this case
  when $s = t$ or $p = q$). Note that $t \in B_k(p) \cap B_1(s) =
  \hR_p(s)$ and consequently, $f_p(s) \subseteq B_1(t)$. Since $f_p(s)
  \subseteq B_{k-1}(p) \subseteq B_k(q)$, we have $f_p(s) \subseteq
  B_k(q) \cap B_1(t) = \hR_q(t)$. Therefore, $f_q(t) \subseteq
  B^*_1(f_p(s))$ and $d(f_p(s),f_q(t)) \leq 1$. Using symmetric
  arguments, we have $d(f_p(s),f_q(t)) \leq 1$ when $d(q,s) \leq k$.

  Assume now that $d(q,s) = d(p,t) = k+1$. Note that this implies that
  $p \neq q$, $s \neq t$, $\du{p}{f_q(t)}{k}$ and $\du{q}{f_p(s)}{k}$.
  Since $d(p,s) = k$ and $\du{p}{f_q(t)}{k}$, we have $\{s,t\}\cup
  f_q(t) \subseteq \hR_p(t)$.  Consider a vertex $u \in f_p(t)$. By
  definition of $u$, we have $d(p,u) = k$ and $\{s,t\}\cup f_q(t) \subseteq
  B_1(u)$. Note also that $d(q,u) = k$ since $d(q,s) = k+1$ and since
  $\dm(q,f_q(t)) = k-1$. Therefore, by the previous case replacing $t$
  by $u$, we have $d(f_p(s),f_q(u)) \leq 1$.
  Note that $\hR_q(t) = B_1(t) \cap B_k(q) \subseteq B_1(t) \cap
  B_{k+1}(p) = \hR_p(t)$.  Since $u \in f_p(t)$, we obtain
  $\hR_q(t) \subseteq \hR_p(t) \subseteq B^*_1(f_p(t)) \subseteq B_1(u)$. Consequently,
  $\hR_q(t) \subseteq B_1(u) \cap B_k(q) =\hR_q(u)$ and $f_q(u) \subseteq f_q(t).$  Therefore
  $d(f_p(s),f_q(t))\leq d(f_p(s),f_q(u)) \leq 1$, concluding the proof.
\end{proof}

From Propositions~\ref{th-np} and~\ref{prop-bicombing}, we immediately
get the following result.

\begin{corollary}\label{normal-bicombing}
  In a Helly graph $G$, the set of normal paths satisfies the 2-sided
  fellow traveler property. More precisely, for any four vertices $s, t, p, q$ and two integers
  $k' \geq k$ such that $d(p, s) = k'$, $d(q, t) = k$, $d(s, t) \leq
  1$ and $d(p,q) \leq 1$ and for any normal paths $P = (p=s_0, s_1,
  \ldots, s_{k'}=s)$ and $Q = (q=t_0, t_1, \ldots, t_k=t)$, we have
  $d(s_i,t_i) \leq 3$ for every $0 \leq i \leq k$ and $d(s_i,t_k) \leq
  3$ for every $k \leq i \leq k'$.
\end{corollary}

Now, we are ready to conclude the proof of biautomaticity from Theorem \ref{biautomatic}.

\begin{proposition}\label{2recog_Helly} Let a group $\Gamma$ act geometrically on a Helly graph $G$.  Then $\Gamma$ is biautomatic.
\end{proposition}

\begin{proof} Let $\mathcal P$ denote the set of all normal paths of $G$. We will prove now that the path system
${\mathcal P}$   satisfies the conditions (1)-(3) of Theorem~\ref{swiat}. Condition (2) is
satisfied because  any two vertices of $G$ are connected by a path of ${\mathcal P}$. That ${\mathcal P}$ satisfies the
$2$--sided fellow traveler property follows from Corollary~\ref{normal-bicombing}. Finally, condition (1) that the set ${\mathcal P}$
can be $2$--locally recognized  follows from the definition of normal paths and the fact that conditions (1) and (2)
of this definition can be tested within balls of $G$ of radius $2$. Since   $\Gamma$ acts geometrically on $G$, there exists only a constant number of types of  such balls.
\end{proof}

\begin{remark}
  Proposition \ref{2recog_Helly} can be also proved by viewing the set
  ${\mathcal P}^*$ of normal clique-paths of a Helly graph $G$ as
  paths of the face complex $F(X(G))$ of the clique complex of $G$ and
  establishing that ${\mathcal P}^*$ satisfies conditions (1)-(3) of
  Theorem \ref{swiat}.

  The set ${\mathcal P}^*$ in $F(X(G))$ gives rise to a set
  ${\mathcal P}'$ of paths of the first barycentric subdivision
  $\beta(G)$ of the clique complex $X(G)$ of $G$.  Combinatorially,
  $\beta(G)$ can be defined in the following way: the cliques of $G$
  are the vertices of $\beta(G)$ and two different cliques $\sigma$
  and $\sigma'$ are adjacent in $\beta(G)$ if and only if
  $\sigma\subset \sigma'$ or $\sigma'\subset \sigma$. For each path
  $P$ in ${\mathcal P}^*$, each edge $\sigma\sigma'$ of $P$ is
  replaced by the 2-path $(\sigma, \sigma\cup \sigma', \sigma')$ in
  the path $P'$ of ${\mathcal P}'$ corresponding to $P$. Again, one
  can establish that ${\mathcal P}^*$ satisfies conditions (1)-(3) of
  Theorem \ref{swiat}.
\end{remark}

 \section{Final remarks and questions}
\label{s:questions}

We strongly believe that the theory of Helly graphs, injective metric spaces and groups acting on them
deserves intensive studies on its own. In this article we focused mostly on geometric actions of groups on Helly graphs but, similarly to other nonpositive curvature settings, just proper or cocompact actions
should be studied as well. 

Below we pose a few arbitrary problems following the overall scheme of our main results: the first two concern 
examples of Helly groups, the last one is about their properties.

\begin{Prob}
	\label{q:Hgroups}
	(When) Are the following groups (virtually) Helly: mapping class groups, cubical small cancellation groups,  Artin groups, Coxeter groups?
\end{Prob}

Note that confirming a conjecture stated by the authors of the current article, Nima Hoda \cite{HodanonHelly}
proved recently that the Coxeter group acting on the Euclidean plane and generated by three reflections in the sides of the equilateral Euclidean triangle is not Helly. This group is \catz and systolic (hence also biautomatic).

\begin{Prob}
	\label{q:Hops}
	Combination theorems for group actions with Helly stabilisers. Is a free product of two Helly groups with amalgamation over an infinite cyclic subgroup Helly? Are groups hyperbolic relative to Helly subgroups Helly?
	(When) Are small cancellation quotients of Helly groups Helly?
\end{Prob}

As for general properties of Helly groups it is natural to ask which of the properties of \catz groups
are true in the Helly setting. For a choice of such properties a standard reference is the book \cite{BrHa}. 

\begin{Prob}
	\label{q:Hprops}
	Are abelian subgroups of Helly groups finitely generated? Is there a Solvable Subgroup Theorem for Helly groups? Describe centralizers of infinite order elements in Helly groups. 
	Construct low-dimensional
	models for classifying spaces for families of subgroups (e.g.\ for virtually cyclic subgroups) of Helly groups. Describe quasi-flats in Helly groups.
\end{Prob}

\section*{Acknowledgements}
We are very grateful to the anonymous referee for the careful reading
of all parts of the paper, for all his efforts and time spent, and for
his numerous corrections and improvements.

This work was partially supported by the grant 346300 for IMPAN from
the Simons Foundation and the matching 2015-2019 Polish MNiSW fund.
J.C.\ and V.C.\ were supported by ANR project DISTANCIA
(ANR-17-CE40-0015).
A.G.\ was partially supported by a public grant as part of the Fondation Math\'ematique Jacques Hadamard. 
H.H.\ was supported by JSPS KAKENHI Grant Number JP17K00029
and JST PRESTO Grant Number JPMJPR192A, Japan. 
D.O.\ was partially supported by (Polish) Narodowe Centrum Nauki, grants UMO-2017/25/B/ST1/01335 and UMO-2018/31/G/ST1/02681.

\tableofcontents
\setcounter{tocdepth}{2}

\bibliographystyle{plainurl}\bibliography{mybib}

\end{document}